\documentclass[11pt]{amsart} 

\usepackage{amssymb,amsmath,amscd,amsxtra}
\usepackage[hypertex]{hyperref}
\usepackage{color} 
\usepackage[margin=3.5cm]{geometry}
\usepackage[all]{xy}

\newtheorem{theorem}{Theorem}[section]
\newtheorem{lemma}[theorem]{Lemma}
\newtheorem{proposition}[theorem]{Proposition}  
\newtheorem{corollary}[theorem]{Corollary} 

\theoremstyle{definition}
\newtheorem{remark}[theorem]{Remark} 
\newtheorem{example}[theorem]{Example} 

\newtheorem{definition}[theorem]{Definition}

\newcommand{\R}{\mathbb{R}}
\newcommand{\D}{\mathbb{D}}

\newcommand{\PP}{\mathbb{P}}
\newcommand{\C}{\mathbb{C}} 
\newcommand{\Z}{\mathbb{Z}}
\newcommand{\Q}{\mathbb{Q}}  
\newcommand{\K}{\mathbb{K}}
\newcommand{\HH}{\mathbb{H}} 
\newcommand{\hHH}{\widehat{\HH}}

\newcommand{\cO}{\mathcal{O}} 
\newcommand{\cA}{\mathcal{A}}
\newcommand{\cF}{\mathcal{F}} 
\newcommand{\cM}{\mathcal{M}}
\newcommand{\cD}{\mathcal{D}} 
\newcommand{\cU}{\mathcal{U}}
\newcommand{\cV}{\mathcal{V}}
\newcommand{\cL}{\mathcal{L}}
\newcommand{\cH}{\mathcal{H}}
\newcommand{\cX}{\mathcal{X}}
\newcommand{\cZ}{\mathcal{Z}}

\newcommand{\rel}{{\rm rel}}
\newcommand{\irr}{{\rm irr}}
\newcommand{\mar}{{\rm mar}} 

\newcommand{\Jac}{\operatorname{Jac}} 
\newcommand{\Res}{\operatorname{Res}} 
\newcommand{\End}{\operatorname{End}} 
\renewcommand{\Re}{\operatorname{Re}}
\renewcommand{\Im}{\operatorname{Im}} 
\newcommand{\Aut}{\operatorname{Aut}} 
\newcommand{\Mir}{\operatorname{Mir}} 
\newcommand{\id}{\operatorname{id}}
\newcommand{\Pic}{\operatorname{Pic}} 

\newcommand{\frh}{\mathfrak{h}} 
\newcommand{\be}{\mathbf{e}}

\def\parfrac#1#2{\frac{\partial #1}{\partial #2}}
\def\fracp#1{\langle #1 \rangle} 
\def\ceil#1{\lceil #1 \rceil}

\newcommand{\beq}{\begin{equation}}
\newcommand{\eeq}{\end{equation}}
\newcommand{\beqa}{\begin{eqnarray}}
\newcommand{\eeqa}{\end{eqnarray}}
\newcommand{\ben}{\begin{eqnarray*}}
\newcommand{\een}{\end{eqnarray*}}

\def\det{\mathop{\rm det}\nolimits}
\usepackage{mathrsfs}
\makeatletter
\newsavebox{\@brx}
\newcommand{\llangle}[1][]{\savebox{\@brx}{\(\m@th{#1\langle}\)}%
  \mathopen{\copy\@brx\kern-0.5\wd\@brx\usebox{\@brx}}}
\newcommand{\rrangle}[1][]{\savebox{\@brx}{\(\m@th{#1\rangle}\)}%
  \mathclose{\copy\@brx\kern-0.5\wd\@brx\usebox{\@brx}}}
\makeatother

\def\lf{\left\lfloor}   
\def\rf{\right\rfloor}

\def\lan{\Big\langle}   
\def\ran{\Big\rangle}
\def\one{\mathbf{1}}

\begin{document} 

\title[GW Theory of Quotient of Fermat Calabi-Yau varieties]{\textbf{Gromov-Witten Theory of Quotient of Fermat Calabi-Yau varieties}}
\author{Hiroshi Iritani, Todor Milanov, Yongbin Ruan, Yefeng Shen}

\address{Department of Mathematics, Graduate School of Science, Kyoto University, Kitashirakawa-Oiwake-cho,
Sakyo-ku, Kyoto, 606-8502, Japan}
\email{iritani@math.kyoto-u.ac.jp}

\address{Kavli IPMU (WPI), UTIAS, The University of Tokyo, Kashiwa, Chiba 277-8583, Japan}
\email{todor.milanov@ipmu.jp}

\address{Department of Mathematics, University of Michigan, Ann Arbor, MI 48105, USA}
\email{ruan@umich.edu}

\address{Department of Mathematics, Stanford University, Stanford, CA 94305, USA}
\email{yfshen@stanford.edu}
%\date{April 19, 2008}
\maketitle

\let\oldtocsection=\tocsection
\let\oldtocsubsection=\tocsubsection

\renewcommand{\tocsection}[2]{\hspace{0em}\oldtocsection{#1}{#2}}
\renewcommand{\tocsubsection}[2]{\hspace{1em}\oldtocsubsection{#1}{#2}}
{\tableofcontents }

\section{Introduction}

    Gromov-Witten theory started as an attempt to provide a rigorous mathematical foundation for the so-called A-model topological string theory of Calabi-Yau varieties.
    Even though it can be defined for all the K\"ahler/symplectic manifolds, the theory on Calabi-Yau varieties remains the most difficult one.
    In fact, a great deal of techniques were developed for non-Calabi-Yau varieties during the last twenty years. These techniques have only limited bearing
    on the Calabi-Yau cases. In a certain sense, Calabi-Yau cases are very special too. There are two outstanding problems for the Gromov-Witten theory of Calabi-Yau varieties 
    and they are the focus of our investigation.
    
    More than twenty years ago, physicists Bershadsky-Cecotti-Ooguri-Vafa \cite{BCOV} studied the higher genus B-model theory. One of the consequences of their investigation is the following mathematical
    conjecture.

    {\bf Modularity Conjecture: }{\em Suppose that $X$ is a Calabi-Yau manifold/orbifold  and $\cF^{GW}_g$ is the genus $g$ generating function of its Gromov-Witten theory.
    Then $\cF^{GW}_g$ is a quasi-modular form in an appropriate sense (see Definition \ref{def:qmf}, Definition \ref{def:amf} and Remark \ref{rem:qmf} ).}
    
    One of main intellectual advances of the field during the last several years was the realization that the modularity conjecture should be extended to orbifold quotients $[X/G]$ of a Calabi-Yau manifold/orbifold $X$.
    
    When $X$ is a Calabi-Yau hypersurface of weighted projective space, there is another famous duality from physics as follows. Suppose $X_W=\{W=0\}\subset \mathbb{P}(w_1, \cdots, w_n)$
    is a degree $d$ hypersurface. $X_W$ is a Calabi-Yau orbifold iff $d=\sum_i w_i$. 
    Let $G_W$ be the group of diagonal matrices preserving $W$. $G_W$  is nontrivial and it contains a special matrix $J=(\exp(2\pi \sqrt{-1}w_1/d), \cdots, \exp(2\pi \sqrt{-1} w_n/d))$. $J$ acts trivially on $X_W$. 
    In addition to $W$, we can choose a so-called {\em admissible group}  $\langle J \rangle \subset G \subset G_W$. Then, $\tilde{G}=G/\langle J \rangle$ acts faithfully on $X_W$.
     There are two curve counting theories built out
    of data $(W, G)$: the Gromov-Witten theory of an orbifold $[X_W/\tilde{G}]$ and the FJRW theory of $(W, G)$ \cite{FJR, FJR2}. Let $\cF_g^{\rm GW}, \cF_g^{\rm FJRW}$ be the generating functions of
    each theory. Define {\em partition functions}
    $$\cD_{\rm GW}=\sum_{g\geq0} \hbar^{g-1} \cF_g^{\rm GW}, \  \cD_{\rm FJRW}=\sum_{g\geq0} \hbar^{g-1} \cF^{\rm FJRW}_g.$$
    
    The second outstanding problem for Calabi-Yau varieties is the following conjecture \cite{W, R}.

    {\bf Landau-Ginzburg/Calabi-Yau correspondence Conjecture:  }{\em There is a differential operator $\hat{}\cU$ built out of genus zero data (the quantization of symplectic transformation in the sense of
    Givental) such that up to an analytic continuation
    $$\cD_{\rm GW}=\cU(\cD_{\rm FJRW}).$$}
    
    The above two conjectures are central for our understanding the GW-theory of Calabi-Yau varieties. For example, they are at the heart of a recent
    spectacular advance in physics \cite{HKQ} to compute higher genus Gromov-Witten invariants of the quintic 3-fold up to genus 51!
    
  It is clear that both conjectures are difficult.  In \cite{CHR}, it was proposed to put both conjectures into a single framework using {\em global mirror symmetry}.  Here, the word {\em global} refers to the global property of the B-model. The traditional version of mirror symmetry is
    {\em local} in the sense that we study a neighborhood of so-called {\em a large complex structure limit}. Global
    mirror symmetry emphasizes the idea of moving away from a large complex structure limit. In fact, we want to move around
    the entire B-model moduli space and study all the interesting limits including (not exclusively) the large complex limit. One of the special ones  is the Gepner limit, corresponding to
    FJRW theory.  Therefore, the knowledge of the Gepner limit (FJRW theory) will yield a wealth of information
    at the large complex structure limit (GW  theory). This provides an effective way to compute higher-genus Gromov-Witten invariants of Calabi-Yau hypersurfaces, which is a central and yet difficult problem in geometry and physics. Furthermore, one can study global properties of the
    entire family. The global properties of B-model naturally lead to the
    modularity of Gromov-Witten theory, a remarkable bonus of {\em global mirror symmetry}.
 This was exactly the way that BCOV discovered the modularity more than twenty years ago. Since then, there has been steady progress in
 physics on the modularity conjecture \cite{ABK, HK, HKQ, HKK, ASYZ}.     In a sense,  the mathematicians are finally catching up! However, the recent mathematical development did not follow
    the physical blueprint. Recall that the physical discussion for last 15 years focused on the Calabi-Yau
    B-model (see a mathematical formulation in \cite{CS}). An unexpected twist of recent events in mathematics is the development of
    the above framework in the set-up of
    the Landau-Ginzburg model over $[X/G]$, a related but much larger model.
    
    The main result of this article is to prove both conjectures 
for $(W, G_W)$ (Theorem \ref{thm:main}) 
in the case that $W$ is a Fermat polynomial. 
          
          \begin{theorem}
          Suppose that $W$ is a Fermat polynomial with $d=\sum_i c_i$ (hence  $X_W$ defines a Calabi-Yau hypersurface). Then, 
          \begin{itemize}
         \item[(1)]  LG/CY correspondence conjecture holds for the pair $(W, G_W)$.
         \item[(2)] The modularity conjecture holds for $[X_W/\tilde{G}_W]$.
         \end{itemize}
          \end{theorem}
          
          We would like to mention that there are two other parts of LG/CY correspondences, {\em cohomological corespondence} and {\em genus zero correspondence}.
          The cohomological correspondence was solved for an arbitrary admissible pair $(W, G)$ by Chiodo-Ruan \cite{CHR3}. The genus zero correspondence for Fermat polynomial $W$ was solved by
          Chiodo-Iritani-Ruan \cite{CHR2, CIR} for the pair $(W, \langle J \rangle)$ (see wall-crossing proof in \cite{RR}) and by Lee-Priddis-Shoemaker \cite{PS, PLS} for the pair $(W, SL_W)$. The all-genus correspondence for simple elliptic singularities was solved by Krawitz-Shen \cite{KSh} and Milanov-Ruan \cite{MR}.
          There are also  very interesting versions for complete intersections by Clader \cite{CL} and  non-Calabi-Yau cases by Acosta \cite{AC}. Our focus is the higher genus correspondence as we stated in the theorem. 
          However, an intermediate
          step is a  proof of the genus zero correspondence for the pair $(W, G_W)$. The modularity conjecture was solved in dimension one\cite{MR, MRS, SZ} (see \cite{Coates-Iritani:Fock} for a related work on compact toric orbifolds). 
    
    Let's spell out our  general strategy.  The original version of  the LG/CY correspondence is a conjectural statement
          connecting the GW  theory of $X_W$ and the FJRW theory of $(W, \langle J \rangle)$. The computation of higher-genus
          Gromov-Witten invariants is a very difficult problem, which we hope to solve using the LG/CY correspondence. However, we can improve the situation by taking a certain
          maximal quotient $[X_W/\tilde{G}_W]$. By the Berglund-H\"ubsch-Krawitz LG-to-LG mirror symmetry \cite{BH, Kr}, $[X_W/\tilde{G}_W]$ should be mirror to the large
          complex structure limit of the B-model family of the  dual polynomial $W^T$ (a Fermat polynomial is self-dual). The Gepner limit in the B-model family is mirror to the FJRW theory of
          $(W, G_W)$ \cite{HLSW}. The B-model family of $W^T$ corresponds to miniversal deformation of $W^T$. Its genus zero theory is
          known as Saito's Frobenius manifold theory \cite{KS}. Saito's Frobenius manifold is generically semi-simple and Givental has
          defined a higher-genus potential function on the semi-simple locus \cite{G1}. Namely, we have a rigorous mathematical definition of the B-model theory
          in this case for all genera.   Using Teleman's
          solution of the Givental conjecture \cite{Te}, the higher genus theory of a semi-simple GW-theory is determined by the genus zero theory. Therefore, the all-genus
          LG/CY correspondence is reduced to the genus zero correspondence. On the other hand, there is no such reduction for CY cases such as $X_W$.
          We should mention that the extension of the Givental-Teleman higher genus function to non-semisimple locus is a well-known difficult problem
          and has been solved recently by Milanov \cite{Mi}.
                    
          We shall implement our strategy in  two steps:
          (i) a construction of the global LG B-model of $W^T$, and (ii) two mirror symmetry theorems connecting the B-model at the large complex structure
          limit to GW-theory and the B-model at the Gepner limit to FJRW-theory.           
          We have applied the above strategy successfully  for quotients of elliptic curves by $\Z_3, \Z_4, \Z_6$ \cite{KSh, MR}. But the B-model construction in \cite{MR}
          does not generalize to higher dimensions. In this article, we develop the higher dimensional theory 
          using a different approach.

          The main results of this article have been reported in various conferences during last five years. We apologize for
          the long delay.
          
           The article is organized as follows. In the Section 2, we will review the global CY-B-model  to motivate our global LG-B-model
           construction and the appearance of quasi-modular forms in Gromov-Witten theory. Sections 3-5 form the technical core of
           the paper where we construct the global LG-B-model. We should mention that  many ingredients were already in the literature \cite{He}. The two mirror symmetric theorems  as well as the proof of the main theorem will be  presented in Sections 6 and 7.
           The proof of the main theorem (Theorem \ref{thm:main}) 
will be presented in the Section 7. 
    
We thank Rachel Webb for careful reading of our manuscript 
and for helpful comments. 
Y.~R.~would like to thank Albrecht Klemm from whom he learned a great deal about the modularity conjecture. 
Y.~S.~would like to thank Si Li and Zhengyu Zong for helpful discussions. 

The work of H.~I.~is partially supported by JSPS Grant-In-Aid 
16K05127, 16H06337, 25400069, 26610008, 23224002. 
The work of T.~M.~is partially supported by JSPS Grant-In-Aid 26800003 
and by the World Premier International Research Center Initiative (WPI Initiative), 
MEXT, Japan. 
The work of Y.~R.~ is partially supported by NSF grants DMS 1159265
and DMS 1405245. 
The work of Y.~S.~is partially supported by NSF grant DMS 1159156.

\section{Global CY-B-model and quasi-modular form}
We are primarily working in the LG-setting.  
In this section, we review some basic properties expected 
for the Calabi-Yau B-model to motivate our construction. 
In the process, quasi-modular forms appear naturally in GW-theory. 
We follow closely the presentation of Dolgachev-Kondo \cite{DK}.

The B-model on a Calabi--Yau manifold $X$ concerns 
the moduli space of complex structures (possibly with a marking) on $X$. 
Traditionally, the moduli space is studied by its Hodge structure. 
Let us start from the abstract set-up. 
Let $V$ be a real-vector space and let $V_\C = V\otimes_\R \C$ 
denote the complexification of $V$. 
An {\em Hodge structure of weight $k$} on $V$ is 
the direct sum decomposition of $V_\C$: 
    $$V_{\C}=\bigoplus_{p+q=k} V^{p,q}$$
that satisfies $\overline{V^{p,q}}=V^{q,p}$. 
A {\em polarization} is a $(-1)^k$-symmetric 
non-degenerate bilinear form $Q \colon V\times V \to \R$, 
where $(-1)^k$-symmetricity means $Q(x,y) = (-1)^k Q(y,x)$. 
Furthermore, we require that $Q$ satisfies the conditions 
\begin{itemize}
\item[(i)] $Q(x,y)=0$ for all $x\in V^{p,q}, y\in V^{p',q'}$ with 
$(p,q) \neq (q',p')$;  
\item[(ii)] $i^{p-q}(-1)^{k(k-1)/2}Q(x,\bar{x})>0$ 
for all non-zero $x\in V^{p,q}$ 
\end{itemize}
where $Q$ is extended to a complex bilinear form on $V_\C$. 
We can associate the {\em Hodge filtration} 
 \[
0\subset F^k\subset F^{k-1}\subset \dots \subset F^0=V_{\C}, 
\]
by $F^p=H^{p,k-p} \oplus H^{p+1, k-p-1} \oplus \cdots 
\oplus H^{k,0}$. 
The above Hodge filtration defines a {\em flag} of $V_{\C}$. 
We can recover the Hodge decomposition from the flag by 
\[
H^{p,q} = F^p \cap \overline{F^{k-p}}. 
\]
Alternatively, when we have a polarization $Q$, we can also write 
\[
H^{p,q}=\{x\in F^p: Q(x,\bar{y})=0, \forall y\in F^{p+1}\}.
\]
Take a decreasing sequence ${\bf m}=(\dim V \ge 
m_1\ge m_2 \ge \cdots \ge m_k\ge 0)$ 
of integers in the range $[0,\dim V]$ 
and denote by $Fl({\bf m}, V_{\C})$ the partial flag variety 
consisting of flags $(F^k \subset F^{k-1} \subset \cdots \subset F^0= V_\C)$ 
of linear subspaces with $\dim F^p = m_p$, $1\le p\le k$. 
% $Fl({\bf m}, V_{\C})$ is a closed algebraic subvariety of the product of 
% the Grassmann varieties $G(m_p,V_{\C})$ and carries a sequence of 
% tautological bundles of rank $m_p$ pulled back from that of $G(m_p,V_{\C})$. 
A polarized Hodge structure of weight $k$ defines a point 
$(F^k\subset F^{k-1} \subset \cdots \subset F^0=V_\C)
\in Fl({\bf m}, V_{\C})$  
satisfying the following conditions
    \begin{itemize}
    \item[(i)] $V_{\C}=F^p\oplus \overline{F^{k-p+1}}$ for all $1\le p\le k$;
    \item[(ii)] $ Q(F^p, F^{k-p+1})=0$ for all $1\le p\le k$;
    \item[(iii)] $(-1)^{k(k-1)/2}Q(Cx, \bar{x})>0$ for all 
non-zero $x\in V_\C$, where $C$ is the real 
    endomorphism of $V_\C$ such that $C|_{H^{p,q}} = i^{p-q} \id_{H^{p,q}}$, 
    called the \emph{Weil operator}. 
    \end{itemize}
We denote by $\D_{\bf{m}}(V,Q)$ the subspace of $Fl({\bf m}, V_\C)$ 
satisfying the above conditions. 
The space $\D_{\bf{m}}(V,Q)$ is called the \emph{period domain} 
of $(V,Q)$ of type ${\bf m}$. 
By choosing a basis of $V$ and that of $F_p$, the subspace $F_p$ can be represented 
by a complex matrix $\Pi_p$ of size $r\times m_p$, which is 
called the \emph{period matrix}. 
An additional important structure is the integral structure: 
it is a free $\Z$-module 
$\Lambda \subset V$ of rank $\dim V$ such that 
the polarization $Q$ takes values in $\Z$ on $\Lambda$. 

\begin{example} 
\label{exa:Hodgestr}
Let $X$ be a compact K\"ahler manifold of dimension $k$. 
The middle cohomology $H^k(X,\R)$ carries a natural Hodge structure 
of weight $k$ via 
the Hodge decomposition $H^k(X,\C) = \bigoplus_{p+q =k} H^{p,q}(X)$. 
A polarization on $H^k(X,\R)$ is given by the intersection form 
$Q(\alpha,\beta) = \int_X \alpha \cup \beta$ and 
the integral structure is given by $\Lambda=H^k(X, \Z)$. 
\end{example}

Let $f\colon \cX \rightarrow T$ be a family of $k$-dimensional 
compact complex manifolds such that the total space $\cX$ is K\"ahler. 
As we saw in Example \ref{exa:Hodgestr}, the middle cohomology 
$H^k(X_t,\R)$ of each fiber $X_t = f^{-1}(t)$ carries a polarized 
Hodge structure of weight $k$ for $t\in T$. 
Let $V$ be a real vector space $V$ and let $Q_0$ be a  
$(-1)^k$-symmetric pairing on $V$. 
An isomorphism 
\[
\varphi_{t}\colon (V, Q_0) \cong (H^k(X_t, \R), \text{intersection pairing}). 
\]
is called a \emph{marking} of $X_t$. 
By pulling back the Hodge structure on $H^k(X_t,\R)$ by $\varphi_t$, 
we obtain a $Q_0$-polarized Hodge structure of weight $k$ on $V$. 
We analytically continue the marking $\varphi_t$ in $t$ so that 
it is flat with respect to the Gauss-Manin connection. 
Then $\varphi_t$ is extended to a multi-valued period map 
\[
\phi \colon T \dashrightarrow \D_{\bf{m}} = \D_{\bf{m}}(V,Q_0), 
\quad 
t\mapsto (\varphi_t^{-1}(F^p_t))
\]
where $F^p_t$ denotes the Hodge flag on $H^k(X_t,\C)$ and 
$\D_{\bf{m}}$ denotes the period domain of type 
${\bf m} = (m_p)$ with $m_p = \dim F^p_t$ 
(which is independent of $t$). 
The period map becomes single-valued on the universal cover of $T$. 
The multi-valuedness of the period map 
is measured by the \emph{monodromy representation} 
\[
\alpha\colon 
\pi_1(T, t_0)\rightarrow G_{\Lambda}:=\Aut(\Lambda, Q|_{\Lambda}). 
\]
Let $\Gamma$ be the image of $\alpha$. 
We obtain a single valued period map
 \[
\bar{\phi}\colon 
T\rightarrow \D_{\bf{m}}/\Gamma\rightarrow 
\D_{\bf m}/G_\Lambda.
\]
The global Torelli theorem is a statement that 
$\bar\phi(t_1) = \bar\phi(t_2)$ in $\D_{\bf{m}}/G_\Lambda$ 
implies $X_{t_1} \cong X_{t_2}$, 
which is basically true in dimension one and two. 
It is unknown if the global Torelli theorem holds in higher dimensions. 
Another important property is whether or not 
$\D_{\bf{m}}/G_\Lambda$ is a hermitian symmetry space, 
which makes the connection to number theory. 
Again, this is the case in dimension one and two and false in higher dimension. 

When $X$ is Calabi-Yau, $F^k=H^{k,0}$ is one-dimensional. 
An element of $H^{k,0}$ is  a 
{\em holomorphic $(k,0)$ form or a Calabi-Yau form}.
$F^k$ induces a holomorphic line bundle
\[
\cL\rightarrow \D_{\bf{m}}.
\]
In physics literature, $\cL$ is called a {\em vacuum} line bundle. 
It is equivariant with respect to the $G_{\Lambda}$ action and hence descends 
to $\D_{\bf{m}}/G_{\Lambda}$. 
We use the same $\cL$ to denote its pull back to $T$.
      Using $\cL$, we can define the {\em modular form}.
\begin{definition}\label{def:qmf}
We call an analytic (holomorphic) section $\Psi$ of $\cL^{w}$ 
a {\em (holomorphic) modular form of weight $w$}  of $T$.
Alternatively, $\Psi$ can be viewed as an analytic function 
on the total space of $\cL$ such that $\Psi(zv)=z^{-w}\Psi(v)$. 
We call a holomorphic function $\psi$  on $\D_{\bf m}$ a {\em quasi-modular form}
if it is the holomorphic part of a ``non-holomorphic''  modular form. 
In other words, there is a (non-holomorphic) modular form $\Psi$ and
functions $h_1, \dots, h_k$ (anti-holomorphic generators)  
such that $\Psi$ is a polynomial of $h_1, \cdots, h_k$ 
with holomorphic functions as coefficients and $\psi$ as the constant term.
\end{definition}

    \begin{remark}\label{rem:qmf}
    The above definition is unsatisfactory since it also includes  other objects such as mock modular forms. We use it as the working
    definition of this paper because of the lack of a better definition. The main point of BCOV's paper is that B-model  GW-theory generating
    function should be a {\em almost holomorphic }section of $\cL^k$ and hence almost holomorphic modular form. Here, the almost holomorphic
    means that its anti-holomorphic generators satisfy the so-called {\em holomorphic anomaly equation}. The A-model GW-theory generating function
    corresponds to the holomorphic part of B-model generating function. An important future problem is to study these anti-holomorphic generators,
    which will lead to a definition closer to that in number theory. 
    \end{remark}
    
\begin{example}[{\cite[\S 4]{DK}}]
A Hodge structure of weight 1 on $V$ gives rise to a decomposition 
\[
V_{\C}=V^{1,0}+V^{0,1}
\]
such that $\overline{V^{1,0}} = V^{0,1}$. 
In this case $V$ is necessarily even dimensional; we set 
$2g = \dim_\R V$. 
A polarization on $V$ is given by a symplectic form $Q$ that satisfies 
the following conditions: 
\begin{itemize}
     \item $Q|_{V^{1,0}}=0$, $Q|_{V^{0,1}}=0$ and; 
     \item $i Q(x, \bar{x})>0$, $\forall  x\in V^{1,0}\setminus \{0\}$. 
\end{itemize}
It is easy to check that $H(x,y)=-iQ(x, \bar{y})$ (or $Q(x,y)=iH(x, \bar{y})$) 
defines a hermitian form on $V_{\C}$ of signature $(g,g)$ 
\cite[Lemma 4.2]{DK}. 
Let $G(g, V_{\C})$ be the Grassmannian of $g$-dimensional subspaces 
of $V_{\C}$.
Set
\[
G(g, V_{\C})_H=\{W\in G(g, V_{\C}) :  Q|_{W}=0, H|_{W}>0\}.
\]
Then we have a one-to-one correspondence between Hodge structures 
on $V$ with polarization form $Q$ and points in $G(g,V_\C)_H$ 
\cite[Theorem 4.3]{DK}. 
Choose a symplectic basis of $V$. Then we can find a unique 
basis $w_1,\dots,w_g$ of $W$ of the form 
\[
\begin{pmatrix} 
\vert &  & \vert \\ 
w_1 & \dots & w_g \\
\vert &  & \vert 
\end{pmatrix} 
= 
\begin{pmatrix} 
Z \\ 
I_g 
\end{pmatrix} 
\]
where $I_g$ is the identity matrix of size $g$. 
The $g\times g$-matrix $Z$ must satisfy  
\[
Z^T=Z, \quad 
\Im(Z)=\frac{1}{2i}(Z-\bar{Z})>0.
\]
Therefore we can identify the period domain $\D_{(2g,g)}$ of 
polarized Hodge structures of weight $1$ with the space 
\[
\cZ_g=\{Z\in \operatorname{Mat}_{g}(\C): Z^T=Z, \Im(Z)>0\}. 
\]
This is the {\em Siegel upper half plane of degree $g$}. 
The dimension is given by $g(g+1)/2$.  
Suppose that $V$ has an integral structure $\Lambda$ 
such that the symplectic form $Q$ induces a perfect
pairing $\Lambda\times \Lambda \to \Z$. 
In this case, the automorphisms group $G_\Lambda 
= \Aut(\Lambda,Q)$ equals $Sp(2g,\Z)$.  
When we choose an integral symplectic basis of $V$, 
an element $M\in Sp(2g,\Z)$ acts on the period domain 
$\D_{(2g,g)} \cong \cZ_g$ by 
\[
M(Z)=(AZ+B) (CZ+D)^{-1} 
\]
where we write 
$M = 
\begin{pmatrix} 
         A&B\\
         C&D
\end{pmatrix}$. 

Suppose that $\cX\rightarrow T$ is a one-dimensional family of elliptic curves. 
The middle cohomology $H^1(X_t)$ of each fiber has a Hodge structure 
of weight 1 (with $g=1$). 
In this case, the period domain is the upper half plane $\cZ_1 \cong \mathbf{H}$ 
and the monodromy group $\Gamma$ is a subgroup of $SL(2,\Z)$.   
Suppose that $\cX$ is not a constant family. 
Then, the universal cover of $T$ is an open subset of $\mathbf{H}$ and we have 
$T \subset \mathbf{H}/\Gamma$. 
We would like to consider a modular form on the period domain. 
Note that we can consider $\D_{2,1}$ as a sub-domain of $\PP(V_{\C})$. 
Then the vacuum line bundle $\cL$ is the pull-back of the tautological 
line bundle of $\PP(V_{\C})$.

Let $\omega$ be a holomorphic $(1,0)$-form. 
Choose a symplectic basis (marking) $A, B$ of $H_1(X_t,\Z)$. 
The periods
\[
\alpha=\int_B \omega, \quad \beta=\int_A \omega
\]
define a homogeneous coordinate system on $\D_{(2,1)}$.  
The inhomogeneous coordinate is $\tau=\alpha/\beta\in \mathbf{H}$.
Moreover, the total space of $\cL$ (minus the zero-section) 
can be identified as 
$(V_{\C}\setminus \{0\})/\Gamma$, where 
an element $\begin{pmatrix} a & b \\ c & d \end{pmatrix} \in 
\Gamma$ acts by 
\[
\begin{pmatrix} 
\alpha \\ 
\beta 
\end{pmatrix} 
\longmapsto 
\begin{pmatrix} 
a\alpha+b\beta \\ 
c\alpha+d\beta 
\end{pmatrix} 
\]
By definition, a modular form of weight $k$ is a holomorphic function
\[
f\colon V_{\C} \setminus \{0\}\rightarrow \C
\]
such that
\begin{itemize}
\item[(i)] $f$ is invariant under the $\Gamma$-action;
\item[(ii)] $f(z v_0, z v_1)=z^{-k}f(v_0, v_1).$
\end{itemize}
We can normalize $\omega$ so that $\omega(A)=1$. 
This corresponds to considering the section 
\[
\tau\mapsto (\tau,1)
\]
of $\cL$. 
Set $F(\tau)=f(\tau,1)$. 
Under a fractional linear transformation 
$\tau \mapsto (a\tau+b)/(c\tau+d)$, $F(\tau)$ changes as 
\begin{align*} 
F\left(\frac{a\tau+b}{c\tau+d}\right) 
& =f\left(\frac{a\tau+b}{c\tau+d},1\right) 
=(c\tau+d)^{k}f(a\tau+b, c\tau+d) \\
& =(c\tau+d)^{k}f(\tau,1) =(c\tau+d)^{k}F(\tau)
\end{align*} 
which agrees with the usual definition of modular forms.

The above example can be generalized to higher rank cases. 
The $(1,0)$-part $V^{1,0}$ defines a rank $g$ bundle $\cV$ 
over the Siegel upper half space $\cZ_g$. Set $\cL=\det( \cV)$. 
A {\em Seigel modular form} $f$ of weight $k$ is a section of $\cL^k$. 
Similarly, we can work out its inhomogeneous presentation.  
It corresponds to a function $F\colon \cZ_g\rightarrow \C$ such that
\[
F\left(\frac{A\tau+B}{C\tau+D}\right)=\det 
\begin{pmatrix} 
  A&B\\
  C&D
\end{pmatrix}^k F(\tau)
\]
for $\tau\in \cZ_g$ and  
$\begin{pmatrix} 
A&B\\
C&D
\end{pmatrix} \in Sp(2g, \Z)$. 
\end{example}

\begin{example}
Next, we consider the B-model moduli space of a K3 surface. 
Let $X$ be an algebraic K3 surface. 
The second cohomology group $H^2(X, \Z)$ 
is a free abelian group of rank 22 and the intersection form 
is even unimodular of signature $(3,19)$. 
Let $\omega$ be an ample class. 
We consider the Hodge structure on the primitive cohomology 
$V=H^2_{\rm prim}(X, \Z)$ which is defined to be 
the orthogonal complement of $\omega$. 
The Hodge structure on $V_{\C}$ is of weight $2$, of type $(1,19,1)$ 
and polarized by the intersection form $Q$. 
The Hodge filtration is: 
\[ 
0\subset F^2=H^{2,0}(X)\subset F^1=H^{2,0}(X)+
H^{1,1}_{\rm prim}(X)\subset F^0=H^2_{\rm prim}(X, \C).
\]
Note that the Hodge filtration is completely determined by $F^2$ since 
$F^1=(F^2)^{\perp}$. 
Therefore the period domain $\D_{(21,20,1)}(V,Q)$ is identified with 
the complex manifold 
\[
\D_{(21,20,1)}(V,Q)=\{\C v\in \PP(V_{\C}) : Q(v,v)=0, Q(v, \bar{v})>0\}.
\]
The integral structure $\Lambda$ is again given by $H^2_{\rm prim}(X, \Z)$. 
The monodromy group $\Gamma$ is a subgroup of
$\Aut(\Lambda, Q|_{\Lambda})$.

This can be generalized to the so-called lattice polarized K3 surfaces. 
Let $M$ be an even lattice of signature $(1, r-1)$. 
An $M$-polarized K3 surface is a pair $(X,j)$ 
of an algebraic K3 surface $X$ and a primitive embedding 
$j\colon M\rightarrow \Pic(X)$ of lattices such that the image of $j$ 
contains an ample class. 
If $M$ is of rank one, it reduces to the previous case. 
A family of $M$-polarized K3 surfaces is a family $\pi\colon Y\rightarrow T$ 
of K3 surfaces equipped with primitive embeddings 
$j_t\colon M\rightarrow H^2(X_t, \Z)$ containing an ample class on $X_t$ 
and depending continuously on $t$. 

Let $L_{K3}$ denote the K3 lattice, that is, a lattice isomorphic to $H^2(X,\Z)$. 
We fix a primitive embedding $M \hookrightarrow L_{K3}$ 
and write $N=M^{\perp}$ for the orthogonal complement of $M$. 
The period domain for $M$-polarized K3 surfaces is given by 
\[
\D_M=\{\C v\in \PP(N_{\C}) : Q(v,v)=0, Q(v, \bar{v})>0\}.
\]
The space $\D_M$ is a hermitian symmetric space and 
of great interest to number theorist. 
The modular form in this context is referred to as 
an \emph{automorphic form} in the literature. 
There is an inhomogeneous description for $\D_M$ similar 
to that for the upper half plane. 
Suppose that $e,f\in N$ span a hyperbolic lattice, i.e.~ 
$Q(e, e)=Q(f,f)=0, Q(e,f)=1$. 
Consider the decomposition
\[
N_\R=V_0 \oplus \R f\oplus \R e.
\]
with $V_0$ the orthogonal complement of $\R f \oplus \R e$ 
in $N_\R$. Note that $(V_0,Q)$ is of signature $(1,19-r)$.  
We can identify $\D_M$ with the complex manifold 
\[
\{ z\in V_0 \otimes \C :  Q(\Im z, \Im z)>0\}
\]
via the map 
\[
z\mapsto w(z)=z + f - \frac{1}{2} Q(z,z)e.
\]
Using the above map, we can figure out the automorphic factor --- 
a generalization of $(c\tau+d)^{k}$. 
\end{example}

\begin{example}
Suppose that $X$ is a Calabi-Yau 3-fold. We obtain a Hodge structure of weight 3 
\[
0\subset F^3=H^{30}\subset F^2=H^{30}+H^{21}\subset F^1=H^{30}+H^{21}+H^{12}\subset F^0=V_{\C}.
\]
on $V=H^3(X,\R)$. 
The polarization $Q$ is symplectic in this case.   
The moduli space of complex structures $M_X$ 
on $X$ is smooth of dimension $h=\dim H^{2,1}(X)$. 
The period domain $\D_{\bf{m}}$ for 
$\mathbf{m}=(2h+2, 2h+1, h+1, 1)$ is not a hermitian symmetric space in general. 
The relation to number theory is not clear. 
However, we can define a modular form formally as a section of $\cL^{k}$. 
What is lacking is a inhomogeneous description similar to the upper half plane. 
However, we can again use the periods to define a convenient coordinate system.   
The monodromy group $\Gamma$ can be viewed as a subgroup of $Sp(2h+2, \Z)$.

One can conveniently forget about $F^3, F^2$. 
Then, we obtain a weight 1 Hodge structure
\[
0\subset F^2\subset F^0.
\]
This defines an embedding of the moduli space of complex structures 
into the Siegel upper half plane
\[
i\colon M_X  \rightarrow   \cZ_{h+1}/\Gamma.
\]
However, the image of above embedding is generally complicated.
\end{example}

\section{Global Landau-Ginzburg B-model at genus zero} 
\label{sec:twdR}

In this section we construct the genus-zero data
(Saito structure) of the global B-model over a deformation space of 
quasi-homogeneous polynomials. 
This is given as a vector bundle formed by the
twisted de Rham cohomology, equipped with the Gauss-Manin 
connection and the higher residue pairing. 
In many ways, the material in this section is already standard 
to the experts (see, e.g.~\cite{KS,MS, Sabbah:tame,He}); 
the (only) novel point in our construction is that we restrict 
ourselves to 
relevant and marginal deformations so that the resulting 
structure is global and algebraic.

\subsection{A family of polynomials} 
\label{subsec:family_poly}
Let $x_1,\dots,x_n$ be variables of degrees 
$c_1,\dots, c_n$ with $0<c_i<1$, $c_i\in \Q$. 
Let $\cM_{\rm mar}$ denote the space of all weighted 
homogeneous polynomials of degree one. 
\[
\cM_{\rm mar} = \{f \in \C[x_1,\dots,x_n] :\deg f =1 \}. 
\]
Here the subscript ``mar'' means marginal deformations 
following the terminology in physics. 
We recall the following standard fact: 
\begin{proposition}[\cite{Dimca:topics}]  
\label{prop:regular} 
For a weighted homogeneous polynomial $f\in \cM_\mar$, 
the following conditions are equivalent: 
\begin{enumerate} 
\item $f(x) =0$ has an isolated singularity at the origin; 
\item $\partial_{x_1} f(x),\dots, \partial_{x_n} f(x)$ form a regular 
sequence in $\C[x_1,\dots,x_n]$. 
\end{enumerate} 
For such $f$, the dimension of the Jacobi ring 
\[
\Jac(f):= \C[x_1,\dots,x_n]/(\partial_{x_1} f, \dots \partial_{x_n} f) 
\]
is independent of $f$ and is given by 
\[
N := \frac{(1-c_1) (1-c_2) \cdots ( 1-c_n)}{c_1c_2\cdots c_n}.  
\]
Moreover polynomials $f$ satisfying  
the conditions (1) and (2) form a (possibly empty) Zariski open 
subset of $\cM_\mar$. 
\end{proposition}

\begin{definition} 
We say that a weighted homogeneous polynomial 
$f\in \cM_\mar$ is \emph{regular} if 
one of the equivalent conditions in Proposition \ref{prop:regular} 
holds. 
Let $\cM_\mar^\circ \subset \cM_\mar$ denote the Zariski open 
subset consisting of regular homogeneous polynomials.  
We denote by $\cM$ the space of 
polynomials of degree $\le 1$ with regular leading terms: 
\[
\cM := \left\{ f \in \C[x_1,\dots,x_n] : 
f = \sum_{0<d\le 1} f_d, \ \deg(f_d) = d, \ 
f_1 \in \cM_\mar^\circ\right\}. 
\]
\end{definition} 

We will henceforth assume that $\cM_\mar^\circ$ is nonempty. 
For a point $t \in \cM$, we write $f(x;t) \in \C[x_1,\dots,x_n]$ 
for the polynomial represented by $t\in \cM$. 
Setting $X := \C^n \times \cM$, we have the following diagram: 
\begin{equation} 
\label{eq:universal_polynomial} 
\begin{CD} 
X  @>{f(x;t)}>> \C \\ 
@V{\pi}VV @. \\
\cM 
\end{CD} 
\end{equation}
where $\pi \colon X = \C^n\times \cM \to \cM$ is the projection 
to the second factor. 
The space $\cO(\cM)$ of regular functions on $\cM$ is graded 
as follows.
A finite cover of $\C^\times$ acts on $\cM$ by 
$\lambda \cdot f := \lambda^{-1} f(\lambda^{c_1} x_1,\dots,\lambda^{c_n} x_n)$; 
this action induces the action on functions $\varphi \in \cO(\cM)$ by 
$(\lambda \cdot \varphi)(f) = \varphi(\lambda^{-1} \cdot f)$. 
We say that $\varphi \in \cO(\cM)$ is of degree $d\in \Q$ 
if $\lambda \cdot \varphi = \lambda^d \varphi$. 
The grading on $\cO(\cM)$ and $\deg x_i = c_i$ together define 
a grading on $\cO(X) = \cO(\cM \times \C^n)$. 
The universal polynomial $f(x;t)\in \cO(X)$ 
is of degree one with respect to this grading. 
What is important for us is the fact that 
$\cO(\cM)$ and $\cO(X)$ are \emph{non-negatively} graded. 

We introduce the \emph{critical scheme} $C \subset X$ as follows: 
\[
\cO_C = \cO_X/
( \partial_{x_1} f(x;t),\dots, \partial_{x_n} f(x;t) ). 
\]
Proposition \ref{prop:regular} implies that $(\pi_*\cO_C)|_{\cM_\mar^\circ}$ 
is a locally free cohrerent sheaf of rank $N$; we will see that 
$\pi_*\cO_C$ is also locally free in Corollary \ref{cor:Jacobi_ring} below. 

\begin{remark}\label{group:cc}
A group of coordinate changes on $\C^n$ acts on the  
parameter space $\cM$ and our global B-model is equivariant 
with respect to the group. 
Let $G$ be the group of ring automorphisms of $\C[x]=\C[x_1,\dots,x_n]$ 
preserving the degree filtration $\C[x]^{\le d}=\{f \in\C[x]: \deg f \le d\}$. 
Then $G$ acts on $\cM$ and the diagram \eqref{eq:universal_polynomial}.  
The quotient stack $[\cM/G]$ should be viewed 
as a genuine moduli space.  
\end{remark} 

\begin{remark} 
In practice, it is convenient to work with a family of polynomials 
of the following form:  
for a weighted homogeneous polynomial $f_0(x)$ of degree one 
and a set of homogeneous polynomials $\phi_\alpha(x)$, 
we can consider a family $f(x;t) = f_0(x)+ \sum_\alpha t_\alpha \phi_\alpha(x)$. 
For such a family, 
we say that the deformation parameter $t_\alpha$ is \emph{relevant}  
(resp.~\emph{marginal}, \emph{irrelevant}) if $\deg \phi_\alpha<1$ 
(resp.~$\deg \phi_\alpha =1$, $\deg \phi_\alpha >1$). 
We can assign the degree of parameters as 
$\deg t_\alpha := 1-\deg \phi_\alpha(x)$. 
The above space $\cM$ includes only \emph{relevant and marginal} 
deformations. When we construct a miniversal deformation 
(see \S\ref{subsec:Frobenius}), we choose 
homogeneous polynomials $\{\phi_i\}_{i=1}^N$ such that 
$[\phi_1],\dots,[\phi_N]$ form a basis of the Jacobi ring 
$\Jac(f_0)$; in this case the deformation family may also 
contain irrelevant directions. 
\end{remark} 

\subsection{The twisted de Rham cohomology} 
\label{subsec:tw_deRham} 

We are interested in the hypercohomology of the twisted
de Rham complex: 
\[
\cF = \R^n \pi_*
\left(\Omega^\bullet_{X/\cM}[z],zd_{X/\cM}+df(x;t) \wedge\right)  
\]
where $f(x;t) \colon X\to \C$ is the universal polynomial in the 
diagram \eqref{eq:universal_polynomial}. 
Since $\pi$ is affine, this is: 
\[
\cF \cong \Omega^n_{X/\cM}[z]/(z d_{X/\cM} + df(x;t) \wedge) 
\Omega^{n-1}_{X/\cM}[z]. 
\]
The fiber of $\cF$ at a single polynomial $f$ is called the 
\emph{Brieskorn lattice} \cite{Brieskorn} of $f$. 
A presentation of the Brieskorn lattice as a twisted 
de Rham cohomology group was given in \cite{SaitoK:higher_residue}; 
this is also called the \emph{filtered de Rham cohomology} 
(see \cite{MS}). 
We introduce the grading on $\cF$ given by the grading 
on $\cO(X)$ together with $\deg (d x_i)= c_i$, $\deg z=1$. 
This is well-defined since the differential 
$z d_{X/\cM} + df(x;t) \wedge$ is of degree one. 
The module of global sections of $\cF$ is again non-negatively graded. 
\begin{proposition}
\label{H:vb}
The sheaf $\cF$ is a locally free 
$\mathcal{O}_\cM[z]$-module of rank $N$. 
\end{proposition}
\begin{proof} 
Let us fix an affine open subset $U\subset \cM_\mar^\circ$ 
such that $(\pi_*\cO_C)|_U$ is a free $\cO_U$-module.  
Choose quasi-homogeneous polynomials $\psi_i \in \cO(U)[x_1,\dots,x_n]$, 
$1\le i\le N$ which induce a basis of $(\pi_*\cO_C)|_U$. 
We claim that $\psi_i dx$, $1\le i\le N$ form a 
basis of $\cF$ over $\cM_\rel \times U$, i.e.~the map 
\[
\phi \colon 
\left( \cO_{\cM_\rel \times U}[z]\right)^{\oplus N} \to \cF|_{\cM_\rel\times U}  
\]
sending $(v_i)_{i=1}^N$ to the class of $\sum_{i=1}^N v_i \psi_i dx$ 
is an isomorphism, where we set 
$dx = dx_1 \wedge \cdots \wedge dx_n$. 

First we prove the surjectivity. 
Suppose by induction that the submodule $\cF(\cM_\rel \times U)^{\le k}$ of degree 
less than or equal to $k$ is contained in the image of $\phi$. 
Every homogeneous element $\omega \in \Omega^n_{X/\cM}[z]$ 
of degree $\le(k+1)$ can be written as  
$\omega = \sum_{i=1}^n v_i \psi_i  dx+ df \wedge \alpha$ 
for some $\alpha \in \Omega^{n-1}_{X/\cM}[z]$. 
By taking the homogeneous component if necessary, we may  
assume that $\alpha$ is homogeneous of degree $\le k$. 
Then $[\omega] = [\omega - (z d + df\wedge)\alpha] 
= \sum_{i=1}^N v_i [\psi_i dx]- z [d \alpha]$. 
By induction hypothesis, $[d\alpha]$ is in the image  
of $\phi$ and thus $\omega$ is also in the image. 

Let us prove the injectivity of $\phi$. 
Suppose that $\phi(v) = 0$ for some $v = ( v_1,\dots, v_N)$. 
By definition there exists $\alpha \in \Omega^{n-1}_{X/\cM}[z]$ 
such that $\sum_{i=1}^N v_i \psi_i dx = (zd + df \wedge) \alpha$. 
% By taking homogeneous components, we may assume that 
% $v_i$ and $\alpha$ are homogeneous. 
Expand $v_i$ and $\alpha$ in powers of $z$: 
\[
v_i = \sum_{k\ge 0} v_{i,k} z^k, \quad 
\alpha = \sum_{k\ge 0} \alpha_k z^k 
\]
where the sum is finite. 
Comparing with the coefficient of $z^0$, we have 
$\sum_{i=1}^N v_{i,0} \psi_i = df \wedge \alpha_0$. 
Since $\psi_i$, $1\le i\le N$ form a basis of the Jacobi ring, 
we have $v_{i,0}=0$ for all $i$. Therefore $df \wedge \alpha_0 = 0$. 
Because $\partial_{x_1} f(x;t),\dots,\partial_{x_n} f(x;t)$ 
form a regular sequence, 
there exists $\beta_0 \in \Omega^{n-2}_{X/\cM}$ such that 
$\alpha_0 = df \wedge \beta_0$.  
Setting $\alpha' = \alpha - (zd + df \wedge) \beta_0 
= \sum_{k\ge 1} \alpha'_k z^k$, we have  
\[
\sum_{i=1}^N \sum_{k\ge 1} z^k v_{i,k} \psi_i 
= (zd + d F \wedge) \alpha'. 
\]
Comparing with the coefficient of $z^1$, we obtain $v_{i,1}= 0$ for all $i$. 
We can repeat this argument inductively to show that 
$v_{i,0}= v_{i,1} = \cdots = 0$. 
\end{proof} 

Since we can identify the restriction $\cF|_{z=0}$ 
with $(\pi_* \cO_C) dx$, we obtain:  

\begin{corollary} 
\label{cor:Jacobi_ring} 
The sheaf $\pi_*\cO_C$ is a locally free $\cO_\cM$-module 
of rank $N$. 
\end{corollary} 

\subsection{The Gauss-Manin connection and the higher residue pairing} 
Here we introduce two important structures on the twisted 
de Rham cohomology $\cF$: the Gauss-Manin connection 
$\nabla$ and the higher residue pairing $K$. 
The \emph{Gauss-Manin connection} $\nabla$ is a map 
\[
\nabla \colon \cF \to z^{-1} \Omega_\cM^1 \otimes_{\cO_{\cM}} 
\cF \oplus z^{-2} \cO_{\cM}[z]dz 
\]
defined by the formula: 
\begin{align*} 
\nabla_{\vec{v}} [\phi(x,t,z) dx] 
& = 
\left[ \vec{v} \phi(x,t,z)  + \frac{\vec{v}(f(x;t))}{z} \phi(x,t,z) dx\right] \\ 
\nabla_{\partial_z} [\phi(x,t,z) dx] 
& = \left[ \left(\parfrac{\phi(x,t,z)}{z} - \frac{f(x;t)}{z^2}  
\phi(x,t,z)  
- \frac{n}{2} \frac{\phi(x,t,z)}{z} 
\right) dx \right] 
\end{align*} 
where $\phi(x,t,z) \in \cO_{X}[z]$, $dx = dx_1\wedge 
\cdots \wedge dx_n$ and $\vec{v}$ is a vector field on $\cM$.  
One can easily check that this is well-defined; moreover 
it satisfies the Leibnitz rule: 
\[
\nabla( g(t,z) \omega) = d g(t,z) \omega + g(t,z) \nabla \omega, 
\qquad g(t,z) \in \cO_{\cM}[z], \ \omega\in \cF 
\]
and the flatness condition $\nabla^2=0$.  
The \emph{higher residue pairing} of K.~Saito \cite{SaitoK:higher_residue} 
is a map 
\[
K \colon \cF \otimes_{\cO_\cM} \cF \to \cO_\cM[z]
\]
which we expand in the form
\[
K(\omega_1,\omega_2) = \sum_{p=0}^\infty z^{p} 
K^{(p)} (\omega_1,\omega_2). 
\]
The higher residue pairing is uniquely 
characterized by the following properties: 
\begin{enumerate} 
\item
$zK(\omega_1,\omega_2) = K(z\omega_1,\omega_2)
=-K(\omega_1,z\omega_2)$; 
\item 
$K^{(0)}(\omega_1,\omega_2) $ is the residue pairing on the Jacobi 
algebra of $f$: 
\[
K^{(0)}(\omega_1,\omega_2) = \Res_{X/\cM} \left[ 
\frac{\phi_1(x,t,0) \phi_2(x,t,0) dx}{ 
\partial_{x_1}f(x;t), \dots, \partial_{x_n} f(x;t)}\right] 
\]
where $\omega_i = \phi_i(x,t,z) dx$; 
\item $K(\omega_1,\omega_2)(z) = K(\omega_2,\omega_1)(-z)$; i.e.,~
$K^{(p)}$ is skew symmetric for $p$ odd and symmetric for $p$ even; 
\item
$K$ is flat with respect to the Gauss-Manin connection:
\ben
\xi K(\omega_1,\omega_2) = 
K(\nabla_{\xi}\omega_1,\omega_2) - K(\omega_1,\nabla_{\xi}\omega_2),
\een
where $\xi=z\vec{v}$ (with $\vec{v}$ a vector field on $\cM$) 
or $z^2\partial/\partial z$. 
\end{enumerate} 

\begin{remark} 
\label{rem:oscillatory} 
The Gauss-Manin connection is defined in such a way that 
oscillatory integrals 
\begin{equation} 
\label{eq:oscillatory} 
\cF \ni 
[\phi dx] \longmapsto 
(-2\pi z)^{-n/2}\int_\Gamma \phi(x) e^{f(x;t)/z} dx
\end{equation} 
define solutions  
(i.e.~intertwine the Gauss-Manin connection with the 
standard differential), where $\Gamma$ is a cycle in 
$H_n(\C^n, \{x\in \C^n: \Re (f(x)/z) \ll 0\};\Z)$. 
The prefactor $(-2\pi z)^{-n/2}$ here should be viewed as 
a shift of weights by $n/2$; this is introduced in order to 
make the Gauss-Manin connection compatible with 
the Dubrovin connection on the A-side under mirror symmetry. 
This in turn results in the shift of the higher residue pairing $K_f$ 
by the factor of $z^n$. 
\end{remark} 

\begin{definition} 
We call the triple $(\cF,\nabla,K)$ consisting of the twisted de Rham 
cohomology, the Gauss-Manin connection and the higher residue 
pairing the \emph{Saito structure} of the family 
\eqref{eq:universal_polynomial} of polynomials. 
\end{definition} 

\begin{remark} 
The Saito structure gives a \emph{TEP structure} in the 
sense of Hertling \cite{He2}.  
\end{remark}

\section{Opposite subspaces} 
In this section, we introduce opposite subspaces  
for the Saito structure $(\cF,\nabla,K)$. 
For a marginal polynomial $f$, we obtain a one-to-one correspondence 
between homogeneous opposite subspaces and splittings  
(opposite filtration) of the Hodge filtration on the vanishing cohomology. 
We also observe that the complex-conjugate 
opposite subspace yields a positive-definite Hermitian bundle 
with connection, called the Cecotti-Vafa structure. 
The notion of opposite filtrations were originally used in 
the work of M.~Saito \cite{MS} to construct a flat structure 
(Frobenius structure) \cite{KS,Du} on the base of miniversal 
deformations (see also \S\ref{subsec:Frobenius}). 
Most of the materials in this section are again not new; 
similar (and in fact more general) results have been 
obtained by Saito \cite{MS} and Hertling \cite{He2}. 
Since we restrict ourselves to weighted homogeneous polynomials, 
our presentation has the advantage of being more explicit and elementary. 

% Let us point out that sometimes 
% it will be convenient to identify the points $t\in \cM$ and
% the corresponding deformations $f=f(x,t)$, so the points in
% $\cM$ are functions, while the deformation parameters are
% coordinates. 
\subsection{Symplectic vector space and semi-infinite VHS} 
Let us recall that we sometimes identify the points $t\in \cM$ 
with the corresponding polynomials $f = f(x;t)$, so the points 
in $\cM$ are functions. 
Recall the sheaf $\cF$ of twisted de Rham cohomology groups from 
\S \ref{subsec:tw_deRham}. 
Proposition \ref{H:vb} implies that $\mathcal{F}$ is the sheaf
of sections of a vector bundle $\HH_+ $ on $\cM$, 
whose fiber over a deformation $f\in \cM$ is given by 
the infinite-dimensional vector space 
\ben
\HH_+(f):=
H_{\rm twdR}(f):=\Omega^n_{\C^n}[z]/(zd+df\wedge)
\Omega^{n-1}_{\C^n}[z]\cong \Jac(f)[z].
\een
We introduce the free $\C[z,z^{-1}]$-module
\ben
\HH(f):=H_{\rm twdR}(f)\otimes_{\C[z]}\C[z,z^{-1}]
\een
and its completion 
\ben
\widehat{
\HH}(f):=H_{\rm twdR}(f)\otimes_{\C[z]}\C(\!(z)\!).
\een
The spaces $\HH(f)$, $\hHH(f)$ are 
equipped with the symplectic form 
\[
\Omega(\omega_1,\omega_2) = 
\Res_{z=0} K_f(\omega_1,\omega_2) dz 
\]
where $K_f$ is the restriction of the higher residue pairing to the
fiber of the sheaf $\mathcal{F}$ at $f$. 
Note that $\HH_+(f)$ is a Lagrangian 
(i.e.~maximally isotropic) with respect to $\Omega$. 

The spaces $\HH(f)$, $\hHH(f)$ are the B-model 
analogues of Givental's symplectic space \cite{Givental:symplectic}. 
The Gauss-Manin connection induces a flat connection $\nabla$ on 
the bundle $\HH = \bigcup_f \HH(f)$ 
and the symplectic form $\Omega$ is flat with respect to $\nabla$. 
Over a contractible subset $U$ of the marginal locus $\cM_\mar^\circ$,  
we can identify all fibers $\HH(f)$ via parallel transport\footnote
{The parallel transport is well-defined only over the marginal locus; the 
parallel transport along relevant deformations involves infinitely many 
negative powers of $z$, and only makes sense after 
tensoring $\HH$ with the ring of holomorphic functions 
on $\{z\in \C^\times\}$ over $\C[z,z^{-1}]$.}
with a single symplectic space $\cH$; then we can regard 
$f\mapsto \HH_+(f)$ as a family of Lagrangian subspaces in $\cH$ 
parametrized by $f\in U$. 
This is an example of the \emph{semi-infinite variation of Hodge structure} 
(semi-infinite VHS) in the sense of Barannikov \cite{Barannikov:quantum} 
(see also \cite{CIT}). 
The main property of the semi-infinite VHS is the 
Griffiths transversality: 
\[
\nabla_{\vec{v}} \HH_+(f) \subset z^{-1} \HH_+(f) \qquad 
\text{for $\vec{v} \in T\cM$} 
\]
for the semi-infinite flag $\cdots \subset z \HH_+(f) \subset \HH_+(f) \subset 
z^{-1} \HH_+(f) \subset \cdots$. 
In the $z$-direction, we also have $\nabla_{z\partial_z} \HH_+(f) 
\subset z^{-1} \HH_+(f)$.

\subsection{Definition and first properties}\label{opposite-section}
\begin{definition}[\cite{Barannikov:quantum, CIT}] 
We say that a Lagrangian subspace $P\subset \HH(f)$ is 
{\em opposite} if $\HH(f) = \HH_+(f)\oplus P$ and 
$z^{-1}P\subset P$. 
\end{definition} 

The vector space $\HH(f)$ can be identified with the space of 
sections of a vector bundle over $\{z\in \C^\times\}$ 
and the subspace $\HH_+(f)$ corresponds to the extension 
of the vector bundle across $0$. 
In this viewpoint, the data of an opposite subspace $P$ 
corresponds to an extension of the bundle 
across $\infty$ such that the resulting bundle over $\PP^1$ is trivial.

\begin{proposition}\label{good-basis}
If $P$ is an opposite subspace, then the following properties
hold: 
\begin{enumerate}
\item[(1)]
The vector space $\HH_+(f)\cap zP$ has dimension $N$. 
\item[(2)]
If $\{\omega_i\}_{i=1}^N$ is a basis of $\HH_+(f)\cap
zP$, then $K(\omega_i,\omega_j)\in \C$.
\item[(3)] Let $\{\omega_i\}$ and $\{\omega^i\}$ be dual bases  of
$\HH_+(f)\cap zP$ with respect to the residue pairing
$K_f^{(0)}$. Then 
\ben
\{\omega_i z^k\}_{i=1,\dots,N}^{k=0,1,\dots}
\quad\mbox{and}\quad 
\{\omega^i (-z)^{-k-1}\}_{i=1,\dots,N}^{k=0,1,\dots} 
\een
are bases of respectively $\HH_+(f)$ and $P$ dual with
respect to the symplectic pairing. 
\end{enumerate}
\end{proposition}

The proof of the above proposition is straightforward, so it will be omitted.
Motivated by Proposition \ref{good-basis}, for a given opposite
subspace $P$ we will refer to a basis of
$\HH_+(f)\cap zP$ as a {\em good basis}. Note that a
$\C[z]$-basis $\{\omega_i\}$ of 
$\HH_+(f)$ is good if and only if $K(\omega_i,\omega_j)\in
\C$. 
Similarly, one can define the notion of an opposite subspace and a good
basis for the completion $\hHH(f)$ and its Lagrangian
subspace $\hHH_+(f):=\HH_+(f)\otimes_{\C[z]} \C[\![z]\!]$. 
Proposition \ref{good-basis} still holds, except for property (3), 
which takes the following form. Put $H:=\hHH_+(f)\cap zP$,
then 
\begin{equation}\label{opposite-good}
\hHH_+(f)= H[\![z]\!],\quad P=H[z^{-1}]z^{-1}.
\end{equation}

An opposite subspace $P\subset \HH(f)$ at $f\in \cM$ can be extended 
to a family of opposite subspaces in a neighbourhood $U$ 
of $f\in \cM$ by parallel transport 
(see the discussion in \S\ref{subsec:Frobenius} and 
\cite[\S 2.2]{CIT}). 
We regard this family of opposite subspaces 
as a subbundle of $\HH= \bigcup_f \HH(f)$ and 
denote it again by $P$. The Gauss-Manin connection induces a flat 
connection on the finite-dimensional bundle $zP/P$ and 
the identification 
\[
\HH_+ \cap z P \cong z P/P 
\]
induces a trivialization $\HH_+ \cong (zP/P)[z]$ over $U$ 
by a flat bundle $zP/P$. 
With respect to this trivialization, the Gauss-Manin connection 
is of the form: 
\[
\nabla = d + \frac{1}{z} C 
\]
with $C \in \End(zP/P)\otimes \Omega^1_\cM$ independent 
of $z$. This fact is crucial in the construction of a Frobenius 
(flat) structure. See \S \ref{subsec:Frobenius} for more details. 

\subsection{Homogeneous opposite subspaces over the marginal moduli} 
In this subsection we assume that $f$ lies in the marginal moduli 
$\cM_\mar^\circ$, i.e.~$f$ is a weighted homogeneous polynomial 
of degree 1. 
The operator 
\beq
\label{grading-op}
z\partial_z+\operatorname{Lie}_\xi,\quad \xi:=
\sum_{i=1}^n c_i x_i\partial_{x_i},
\eeq
where $\operatorname{Lie}$ denotes the Lie derivative, 
defines a grading on $\Omega^n_{\C^n}[z,z^{-1}]$. 
Since the twisted de Rham differential $zd + df(x;t)\wedge$ 
is homogeneous (of degree 1), the twisted de Rham cohomology 
$\HH(f)$ inherits the grading. 
We say that an opposite subspace $P\subset \HH(f)$ is 
\emph{homogeneous} 
if $(z\partial_z+\operatorname{Lie}_\xi)P\subset P$. 
We would like to establish one-to-one correspondence between homogeneous 
opposite subspaces and splittings of the Steenbrink's Hodge filtration 
of the vanishing cohomology $\frh:=H^{n-1}(f^{-1}(1);\C)$.  

\begin{remark} 
An opposite subspace $P$ is homogeneous if and only if 
$P$ is preserved by the Gauss-Manin connection in the $z$-direction,  
i.e.~$\nabla_{z\partial_z} P \subset P$. 
This implies that the Gauss-Manin connection has a logarithmic 
singularity at $z=\infty$ with respect to the extension of the bundle $\HH_+(f)$ 
across $\infty$ defined by $P$; this corresponds to the notion of 
TLEP structure \cite{He}. 
\end{remark} 

\subsubsection{Steenbrink's Hodge structure for 
weighted-homogeneous singularities}\label{S-PHS}
Given a holomorphic form $\omega\in 
\Omega_{\C^n }(\C^n)$ we recall
the so-called {\em geometric section} (see \cite{AGV}) 
\ben
s(\omega,\lambda):=\int\frac{\omega}{df} \quad \in \quad H^{n-1}(f^{-1}(\lambda);\C),
\een
where $\omega/df$ denotes a holomorphic $(n-1)$-form $\eta$ defined in
a tubular neighborhood of $f^{-1}(\lambda)$ such that 
$\omega=df\wedge\eta$; the restriction of $\eta$ to $f^{-1}(\lambda)$ 
is well-defined. 
By definition, Steenbrink's Hodge filtration 
\cite{Steenbrink:quasihomogeneous, Stn} on $\frh$ is given by 
\ben
F^p\frh:=\{A\in \frh\ |\ A=s(\omega,1) \ \mbox{for
  some}\ \omega\ \mbox{such that}\     
\operatorname{deg}(\omega)\leq n-p\},
\een
where $\operatorname{deg}(\omega)$ denotes the maximal degree of a
homogeneous component of $\omega$. 
This is an exhaustive filtration; in particular every cohomology class 
of $f^{-1}(1)$ can be represented by a geometric section. 

The vector space $\frh=H^{n-1}(f^{-1}(1);\C)$ is equipped with a
linear transformation $M\in \End(\frh)$, called the {\em classical monodromy}, which
corresponds to the monodromy of the 
Gauss--Manin connection around $\lambda=0$. 
Using the fact that $f$ is
weighted-homogeneous, it is easy to see that if $A=s(\omega,1)$ for
some homogeneous form $\omega$, then $M(A) =
e^{-2\pi\sqrt{-1}\, \operatorname{deg}(\omega)} A.$ Let us decompose
$\frh=\frh_1\oplus \frh_{\neq 1}$, where
$\frh_1$ is the invariant subspace of $M$ and
$\frh_{\neq 1}$ is the remaining part of the spectral
decomposition of $\frh$ with respect to $M$. 
Following Hertling (see \cite[Ch.~10]{He}) we introduce the 
non-degenerate bilinear form
\ben
S(A,B) = (-1)^{(n-1)(n-2)/2}\langle A, \operatorname{Var}\circ \nu(B)\rangle ,\quad
A,B\in \frh,
\een
where $\nu$ is a linear operator such that $\nu=(M-1)^{-1}$ on
$\frh_{\neq 1}$ and $\nu=-1$ on $\frh_1$, and 
the variation operator 
\ben
\operatorname{Var}\colon H^n(f^{-1}(1);\C)\to H_n(f^{-1}(1);\C)
\een
is an isomorphism constructed via the composition of 
the Lefschetz duality
\ben 
H^n(f^{-1}(1);\C)\cong  
H_n(f^{-1}(1),\partial f^{-1}(1);\C)
\een 
and the isomorphism 
$H_n(f^{-1}(1),\partial f^{-1}(1)) \cong 
H_n(f^{-1}(1))$ 
mapping a relative cycle 
$\gamma$ to an absolute cycle 
$\widetilde{M}(\gamma) - \gamma$, where $\widetilde{M}:
f^{-1}(1)\to f^{-1}(1)$ is the geometric monodromy 
fixing the boundary (see \cite[Ch.~1.1, 2.3]{AGV}). 

Combining the results of Hertling (see \cite[Ch.~10]{He}) and 
Steenbrink (see \cite{Stn}) we get the following: 
the filtration $\{F^p\frh\}_{p=0}^{n-1}$, 
the form $S$, and the real subspace
$\frh_\R:=H^{n-1}(f^{-1}(1),\R)$ give rise to
a (pure) Polarized Hodge Structures on $\frh_{\neq 1} $ and
$\frh_1$ of weights respectively $n-1$ and $n$. 
More precisely, put
$\frh_s=\operatorname{Ker}(M-s\operatorname{Id})$; 
then $F^p \frh = \bigoplus_{s\in S^1} F^p \frh_s$ with 
$F^p \frh_s = F^p \frh \cap \frh_s$ and 
\ben
&
\text{(a)}
&
\frh_{s}=F^p\frh_{s}\oplus
\overline{F^{m+1-p}\frh}_{\bar s},\quad \forall p\in \Z,\\
&
\text{(b)}
&
S(u,v) = (-1)^m S(v,u), \\ 
&
\text{(c)} 
&
S(F^p\frh, F^{m+1-p}\frh) = 0,\\
&
\text{(d)} 
&
\sqrt{-1}^{2p-m}S(u,\overline{u})>0\quad \mbox{for}\quad 
u\in F^p\frh_{s}\cap
\overline{F^{m-p}\frh}_{\bar s}\setminus{0},
\een
where $m=n-1$ for $s\neq 1$ and $m=n$ for $s=1$. 
Note that $S(\frh_s,\frh_t) =0$ unless $t =\bar{s}$ 
and that $\overline{\frh_s} = \frh_{\bar{s}}$. 

\subsubsection{The polarizing form and the higher-residue pairing}
\label{pf-hrp}

We will identify the vector space $\frh=H^{n-1}(f^{-1}(1);\C)$ 
with a fiber of 
the local system underlying the Gauss-Manin connection 
$(\HH_+(f) = \cF|_{f},\nabla_{z\partial_z})$. 
Then we describe the higher residue pairing in terms of 
the polarizing form $S$ on $\frh$. 

Recall that oscillatory integrals \eqref{eq:oscillatory} 
give solutions of the Gauss-Manin connection, and therefore 
the local system underlying the Gauss-Manin connection 
is dual to the space 
\begin{equation} 
\label{eq:space_Lefschetz} 
V_{f,z} := \varprojlim_{M} H_n(\C^n, \{ x\in \C^n : \Re(f(x)/z)\le -M\}) 
\end{equation} 
of Lefschetz thimbles (twisted by $(-2\pi z)^{-n/2}$). 
By the relative homology exact sequence, 
we can easily see that this is isomorphic 
to $H_{n-1}(f^{-1}(1))$; hence fibers of the Gauss-Manin 
local system should be identified with $\frh$. 
To make this identification explicit, we use the Laplace transformation. 
When $z<0$ and the integration cycle $\Gamma$ 
in \eqref{eq:oscillatory} is a Lefschetz thimble of $f$ lying over 
the straight ray $[0,\infty)$,  
we may rewrite the oscillatory integral 
\eqref{eq:oscillatory} 
as the Laplace transform of a period
\[
(-2\pi z)^{-n/2}\int_0^\infty e^{\lambda/z} 
\int_{\Gamma_\lambda} s(\omega,\lambda)  
\]
where $\Gamma_\lambda$ is a vanishing cycle in $f^{-1}(\lambda)$ 
such that $\Gamma = \bigcup_{\lambda \in [0,\infty)} \Gamma_\lambda$. 
This can be viewed as the pairing of the vanishing cycle 
$\Gamma_1\subset f^{-1}(1)$ 
and the cohomology class $\widehat{s}(\omega,z)$ of $f^{-1}(1)$ 
given by: 
\[
\widehat{s}(\omega,z) := (-2\pi z)^{-n/2} \int_0^\infty e^{\lambda/z} 
s(\omega,\lambda) d\lambda 
\]
where we identified $H^{n-1}(f^{-1}(\lambda);\C)\cong \frh$ 
via the parallel transport along the integration path 
(with respect to the Gauss-Manin connection), so that $s(\omega,\lambda)$
takes values in $\frh$. 
Thus the map $[\omega]\mapsto \widehat{s}(\omega,z)$ defines 
a flat identification between $\HH_+(f)|_z$ and $\frh$. 
For a homogeneous form $\omega\in \Omega^n_{\C^n}(\C^n)$, 
the geometric section $s(\omega,\lambda)$ satisfies 
the homogeneity 
$s(\omega,\lambda) = \lambda^{\deg(\omega)-1} s(\omega,1)$, and 
therefore we find 
\begin{equation}
\label{eq:s_hat}
\widehat{s}(\omega,z) = (-2\pi z)^{-n/2} (-z)^{\deg(\omega)} 
\Gamma(\deg \omega) s(\omega,1).  
\end{equation} 
Thus $\widehat{s}(\omega,z)$ makes sense as a Laurent polynomial 
of $z$ (with fractional exponents) taking values in $\frh$. 
We verify the following lemma directly. 
% \begin{lemma}\label{van-cohom}
% The map 
% \ben
% \omega\mapsto s(\Gamma(\operatorname{Lie}_\xi)\omega,1),
% \een
% where $\Gamma(\ )$ is the Euler Gamma-function, induces an isomorphism
% \ben
% \Omega_{\C^n}^n(\C^n)/(d-df\wedge)
% \Omega_{\C^n}^{n-1}(\C^n) \cong
% H^{n-1}(f^{-1}(1);\C).
% \een
% \end{lemma}
\begin{lemma} 
\label{van-cohom} 
The map 
\[
\HH_+(f) = \cF|_f \longrightarrow \frh[z^{\pm 1/d}]  , \quad 
[\omega] \longmapsto \widehat{s}(\omega,z) 
\]
is well-defined and intertwines the Gauss-Manin connection $\nabla_{z\partial_z}$ 
with the standard differential $z\partial_z$, where $d$ is a common denominator 
of $c_1,\dots,c_n$ and $n/2$. This induces an isomorphism 
$\HH_+(f)|_{z} \cong \frh$ between fibers for every $z\in \C^\times$. 
\end{lemma} 
\begin{proof} 
Let us first check that the map passes to the quotient 
$\HH_+(f) = \Omega^n_{\C^n}(\C^n)[z]/ 
(z d + df\wedge) \Omega^{n-1}_{\C^n}(\C^n)[z]$. 
If $\omega$ is a homogeneous $(n-1)$-form of degree $m$, then the image of
$z d\omega + df\wedge\omega$ is 
\beq\label{kernel-image}
\Gamma(m) s(d\omega,1) - \Gamma(m+1) s(df\wedge\omega,1).
\eeq
multiplied by $(-2\pi z)^{-n/2} (-z)^{m} z$. 
On the other hand
\ben
s(d\omega,\lambda) = \int \frac{d\omega}{df} = \partial_\lambda
\int\omega =\partial_\lambda s(df\wedge \omega,\lambda).
\een
Using homogeneity, $s(df\wedge \omega,\lambda) = \lambda^m
s(df\wedge\omega,1)$. Hence
$
s(d\omega,1) =  m\, s(df\wedge \omega,1),
$
so the expression \eqref{kernel-image} vanishes. This proves that the
map in the Lemma passes to the quotient. 

Next we show that the map intertwines the Gauss-Manin connection with 
$z\partial_z$. For a homogeneous form $\omega$ of degree $m$, 
the image of $\nabla_{z\partial_z} [\omega] = [-(f/z + n/2) \omega]$ 
is 
\[
(-2\pi z)^{-n/2} (-z)^{m} 
\left(\Gamma(m+1) s(f\omega,1) - \frac{n}{2} \Gamma(m) s(\omega,1)
\right) 
\]
which equals $z\partial_z \widehat{s}(\omega,z) = (m-n/2) \widehat{s}(\omega,z)$ 
since $s(f\omega,1) = s(\omega,1)$.  

The last statement follows by comparing the ranks: the map is 
surjective since every class on $f^{-1}(1)$ is represented by a 
geometric section, and Proposition \ref{H:vb} shows that 
the rank of $\HH_+(f)$ equals the Milnor number $N = \dim \frh$. 
% Note that the quotient on the LHS of the isomorphism that we wish to
% prove is the specialization of the twisted de Rham cohomology
% $\HH_+(f)$ to $z=-1$. Recalling Proposition \ref{H:vb} we get
% that the quotient is a finite dimensional vector space of dimension
% the Milnor number $N$. Therefore, we need just to prove that our map
% is surjective. This however follows from the fact that $f^{-1}(1)$ is
% a Stein manifold, so the cohomology $H^{n-1}(f^{-1}(1);\C)$
% can be computed with the holomorphic de Rham complex. 
\end{proof} 

% \begin{remark} 
% \label{rem:induced_by_topology} 
% By construction, the map $[\omega]\mapsto \widehat{s}(\omega,z)$ in 
% Lemma \ref{van-cohom} is dual to 
% the isomorphism $H_{n-1}(f^{-1}(1)) \cong H_n(\C^n, 
% \{x: \Re(f(x)/z)\ll 0\}) =V_{f,z}$.  
% \end{remark} 

% If $\omega\in \Omega^n_{\C^n}(\C^n)$, then let us
% define the Laplace transform of $s(\omega,\lambda)$ by 
% \ben
% \widehat{s}(\omega,z) = (-2\pi z)^{-n/2} \int_0^\infty e^{\lambda/z}
% s(\omega,\lambda) d\lambda\quad \in \quad \frh,
% \een
% where the integration path is $\lambda = -zt$, $t\in [0,+\infty)$,
% $z$ takes only negative real values, and we identified
% $H^{n-1}(f^{-1}(\lambda);\C)\cong \frh$ via the
% parallel transport  along the integration path (with respect to the
% Gauss--Manin connection), so that $s(\omega,\lambda)$ takes values in
% $\frh$. The function 
% $\widehat{s}(\omega,z) $ can be extended analytically in $z$ along any
% path avoiding $z=0$. 
Let us denote by $\widehat{s}(\omega,z)^*:=
\widehat{s}(\omega, e^{-\pi\sqrt{-1}} z)$ 
the analytic continuation along the semi-circle
$\theta \mapsto e^{-\sqrt{-1}\theta}z $, $0\leq \theta\leq \pi$. 
The relation between the polarizing form $S$ and the higher residue
pairing $K$ has been determined by Hertling 
(see \cite[Ch.~10]{He}, 
\cite[\S 7.2(f); \S 8, Step 2]{He2}). 
We follow the presentation in \cite{Mi2}. 

\begin{theorem}[{\cite{He,He2}, \cite[Lemma 3.3]{Mi2}}] 
\label{thm:SandK} 
The polarizing form $S$ and the higher residue pairing $K_f$ are 
related by the formula: 
\begin{align} \label{SandK}
K_f(\omega_2,\omega_1) & = 
-S(\widehat{s}(\omega_1,z)^*, \nu^{-1} 
\widehat{s}(\omega_2,z) ) 
\end{align} 
where in the right-hand side we use the determination 
of $\widehat{s}(\omega_i,z)$ given canonically for 
$z\in \R_{<0}$ via formula \eqref{eq:s_hat}.  
\end{theorem} 

\begin{remark}
We give a brief explanation for the above formula \eqref{SandK}. 
Pham \cite[2\`eme, \S 4]{Pham} identified the higher residue pairing with 
the dual of the intersection pairing between 
the relative homologies $V_{f,z}$ and $V_{f,-z}$ 
from \eqref{eq:space_Lefschetz}  
(see also \cite[\S 8]{He2}, \cite[Definition 2.18]{CIT}). 
This intersection pairing can then be identified with the Seifert form 
$\langle A, \operatorname{Var}(B)\rangle$  
on $\frh$ by a topological argument 
in \cite[Ch.~2.3]{AGV}, \cite[2\`eme, \S 3.2]{Pham}. 
Note that the right-hand side of \eqref{SandK} 
is induced by the Seifert form. 
\end{remark}

\subsubsection{Splitting of the Hodge structure}
\label{si-section}
Following M.~Saito \cite{MS}, we use opposite filtrations on $\frh$ 
to construct homogeneous opposite subspaces for $\HH_+(f)$ 
(see also \cite[Theorem 7.16]{He}). 
\begin{definition} 
An \emph{opposite filtration} on $\frh = H^{n-1}(f^{-1}(1);\C)$ is an 
increasing $M$-invariant filtration $\{U_p\frh\}_{p\in \Z}$ 
such that 
\ben
&
\text{(a)}
&
U_p\frh=0
\quad \mbox{for}\quad p\ll 0
\quad
\mbox{and}
\quad
U_p\frh=\frh
\quad \mbox{for}\quad p\gg 0,\\
&
\text{(b)}
&
\frh = \bigoplus_{p\in \Z} F^p\frh\cap
U_p\frh,\\
&
\text{(c)}
&
S(U_p\frh , U_{m-1-p}\frh) = 0, \quad \forall p\in \Z,
\een
where $m=n-1$ for $s\neq 1$ and $m=n$ for $s=1$.
\end{definition} 

Let $\psi \colon \frh \to \HH(f)[z^{\pm 1/d}]$ denote 
the map inverse to the map $[\omega] \mapsto \widehat{s}(\omega,z)$ 
in Lemma \ref{van-cohom}. 
This is given by 
\begin{equation} 
\label{eq:psi} 
\psi(A) = \frac{(2\pi)^{n/2}}{\Gamma(\deg \omega)}  
(-z)^{-\deg \omega + n/2} [\omega]  
\end{equation} 
when $A = s(\omega,1) \in \frh$ for some 
homogeneous form $\omega \in \Omega^n_{\C^n}(\C^n)$. 
The image $\psi(A)$ is homogeneous of degree $n/2$ and 
is flat with respect to the Gauss-Manin connection. 
Take $s = e^{2\pi\sqrt{-1} \alpha}$ with 
$0\le \alpha <1$. 
Note that $\psi(A) \in z^{\alpha -n/2} \HH(f)$ for $A \in \frh_s$. 
The Hodge filtration $F^p \frh_s$ 
can be described in terms of the Lagrangian subspace $\HH_+(f)$ as 
\[
F^p \frh_s = \{ A \in \frh_s :  \psi(A) \in z^{p + \alpha - n/2} \HH_+(f) \}   
\]
since, for $A = s(\omega,1)\in \frh_s$, every homogeneous component 
$\omega_j$ of $\omega$ satisfies $\fracp{-\deg \omega_j} = \alpha$ 
and we have 
\[
-\deg \omega_j + \frac{n}{2} \ge p + \alpha - \frac{n}{2} 
\Longleftrightarrow \ceil{\deg \omega_j} \le n-p.  
\] 
Conversely, $\HH_+(f)$ can be reconstructed from $F^p\frh$ as 
\[
\HH_+(f) = \sum_{0\le \alpha<1} 
\sum_{p\in \Z} z^{-p -\alpha +n/2} \psi( F^p \frh_s)  [z] 
\qquad \text{with $s= e^{2\pi\sqrt{-1} \alpha}$.}
\]
The correspondence between homogeneous opposite subspaces $P 
\subset \HH(f)$ 
and opposite filtrations $U_p \frh$ is given similarly. We require 
that $P$ and $U_p\frh$ are related by 
\begin{align} 
\label{eq:opposite_correspondence}
\begin{split} 
U_p \frh_s &= \{ A \in \frh_s : \psi(A) \in z^{p+\alpha - n/2} z P \} \\
P & = \sum_{0\le \alpha<1} 
\sum_{p\in \Z} z^{-p-\alpha +n/2} z^{-1} \psi( U_p \frh_s)  [z^{-1}] 
\qquad \text{with $s= e^{2\pi\sqrt{-1} \alpha}$.}
\end{split} 
\end{align} 
\begin{proposition}[{\cite{MS}, \cite[Ch~7.4]{He}}] 
The formulas \eqref{eq:opposite_correspondence} establish 
one-to-one correspondence between homogeneous opposite subspaces $P$ and 
opposite filtrations $\{U_p \frh\}_{p\in \Z}$. 
\end{proposition} 
\begin{proof} 
Given an opposite filtration $\{U_p\frh\}$, we prove that 
the subspace $P$ defined by the second formula 
of \eqref{eq:opposite_correspondence} is an opposite subspace. 
It is obvious that $P$ is homogeneous. 
Using the decomposition $\frh_s = \bigoplus_{p\in \Z} 
(F^p\frh_s \cap U_p \frh_s)$, we can rewrite 
\begin{align}
\label{eq:P_H+_decomp} 
\begin{split}  
\HH_+(f) &= \bigoplus_{0\le \alpha<1} \bigoplus_{p\in \Z} 
z^{-p -\alpha + n/2} \psi(F^p\frh_s \cap U_p \frh_s) [z] \\
P & = \bigoplus_{0\le \alpha <1} \bigoplus_{p\in \Z} 
z^{-p-\alpha + n/2} z^{-1} \psi(F^p \frh_s \cap U_p\frh_s)[z^{-1}] 
\end{split} 
\end{align} 
where the summand indexed by $\alpha$ is generated by elements 
$\omega \in \HH(f)$ with $\alpha = \fracp{-\deg \omega}$. 
This decomposition clearly shows $\HH(f) = \HH_+(f) \oplus P$. 
The Lagrangian property of $P$ follows from the property (c) 
of the opposite filtration: using 
the fact that $\psi$ is inverse to $[\omega]\mapsto \widehat{s}(\omega,z)$ 
and equation \ref{SandK},  
we have 
\begin{align*} 
& K(z^{-p-\alpha+\frac{n}{2}} z^{-1} \psi(U_p\frh_s), 
z^{-q-\beta + \frac{n}{2}} z^{-1} \psi(U_q \frh_t))  \\ 
& =  z^{-p-q-\alpha-\beta +n-2} S(U_q\frh_t, \nu^{-1} 
U_p \frh_s) 
= z^{-p-q-\alpha-\beta+n-2} S(U_q\frh_t, U_p\frh_s) 
\end{align*} 
when $s =e^{2\pi\sqrt{-1} \alpha}$ and $t = e^{2\pi\sqrt{-1} \beta}$ 
with $\alpha,\beta \in [0,1)$, 
and this is nonzero only if $-p-q-\alpha-\beta+n\le 0$ by the property (c). 
This means $K(P,P) \subset z^{-2} \C[z^{-1}]$ and thus $P$ is Lagrangian. 

In the opposite direction, we start from a homogeneous opposite subspace 
$P \subset \HH(f)$. 
The filtration $U_p\frh_s$ defined by the first formula 
of \eqref{eq:opposite_correspondence} is 
an increasing filtration since $P \subset zP$. 
The homogeneity of $P$ implies that the finite-dimensional 
space $\HH_+(f) \cap zP$ is spanned by homogeneous elements. 
The map $A \mapsto (-z)^{-p-\alpha+n/2} \psi(A)$ 
identifies 
$U_p \frh_s \cap F^p\frh_s$ 
with the homogeneous component of 
$zP \cap \HH_+(f)$ of degree $n-p-\alpha$ 
(when $s = e^{2\pi\sqrt{-1} \alpha}$, $\alpha\in [0,1)$). 
Setting $z=-1$ and using the isomorphy of the map 
$\frh \cong \HH_+(f) \cap zP \cong \HH(f)|_{z=-1}$, 
$A\mapsto \psi(A)|_{z=-1}$ 
in Lemma \ref{van-cohom}, we conclude the decomposition 
$\frh_s = \bigoplus_{p\in \Z} F^p\frh_s \cap U_p\frh_s$, 
i.e.~property (b) for opposite filtrations holds. 
The property (a) follows from (b), and the property (c) follows from 
the Lagrangian property of $P$ by reversing the above argument. 
\end{proof} 

An opposite filtration $\{U_p \frh\}_{p\in \Z}$ induces 
a homogeneous opposite subspace $P$ \eqref{eq:opposite_correspondence} 
and an isomorphism 
$\sigma = \sigma_{U_\bullet} \colon \frh \xrightarrow{\cong} 
\HH_+(f) \cap zP$ 
by the formula (cf.~\eqref{eq:P_H+_decomp}) 
\begin{equation} 
\label{eq:sigma_splitting} 
\sigma(A) = (-z)^{-p-\alpha+\frac{n}{2}} \psi(A) \qquad 
\text{for $A\in F^p \frh_s \cap U_p \frh_s$} 
\end{equation} 
where $s = e^{2\pi \sqrt{-1} \alpha}$, 
$\alpha \in [0,1)$ and $\psi$ is given in \eqref{eq:psi}. 
The map $\sigma_{U_\bullet}$ gives a splitting of the 
projection $\HH_+(f) \to \HH_+(f)|_{z=-1} \cong \frh$. 

The higher residue pairing takes values in $\C$ on the image of 
$\sigma$. We compute the precise values for later purposes. 
\begin{lemma}
\label{sigma-K}
Let $\sigma$ be the splitting \eqref{eq:sigma_splitting} 
associated to an opposite filtration $U_\bullet$. 
If $A \in F^p\frh_s \cap U_p \frh_s$ with $s=e^{2\pi\sqrt{-1}\alpha}$, 
$\alpha\in [0,1)$ and $B \in \frh$ is arbitrary, 
then  
\[
K_f(\sigma(A), \sigma(B)) = C(s) \sqrt{-1}^{2p-m} S(A,B) 
\]
where we set $m=n-1$ if $s\neq 1$ and $m=n$ if $s=1$, and 
\[  
C(s) = \begin{cases} 
2 \sin(\pi \alpha) & \text{if $\alpha \neq 0$;} \\ 
1 & \text{otherwise.}
\end{cases}  
\]  
\end{lemma} 
\begin{proof} 
We may assume that $B\in F^q\frh_t \cap U_q\frh_t$ 
with $t=e^{2\pi\sqrt{-1}\beta}$ and $\beta\in [0,1)$. 
Combining the fact that $\psi$ is inverse to $[\omega]\mapsto \widehat{s}(\omega,z)$ 
and Theorem \ref{thm:SandK}, we find 
\[
K_f(\sigma(A),\sigma(B)) 
= - S(e^{-\pi\sqrt{-1}(-q-\beta+\frac{n}{2})} 
(-z)^{-q-\beta + \frac{n}{2}} B, \nu^{-1} (-z)^{-p-\alpha+\frac{n}{2}} A).  
\] 
This pairing vanishes unless $\alpha+\beta\equiv 0 \mod \Z$ 
and $p+q +\alpha+\beta =n$. 
If $\alpha \neq 0$, this equals 
\[
-e^{\pi\sqrt{-1}(m-p +1-\alpha)} (\sqrt{-1})^{-m-1} 
(e^{2\pi\sqrt{-1}\alpha}-1) 
(-1)^m S(A,B)
\]
by $\nu^{-1} A= (e^{2\pi\sqrt{-1} \alpha} -1)A$, $\beta=1-\alpha$ and 
$p+q=n-1=m$; if $\alpha=0$, this equals 
\[
- e^{\pi\sqrt{-1}(m-p)} (\sqrt{-1})^{-m} (-1)(-1)^m S(A,B) 
\]
by $\nu^{-1}A= -A$, $\alpha=\beta=0$, $p+q=n=m$. The conclusion 
follows easily.  
\end{proof} 

\begin{remark}
When we regard $\HH_+(f)$ as a vector bundle over $\C$ and 
$zP$ as the extension data across $\infty$, the filtration 
$U_p\frh$ \eqref{eq:opposite_correspondence} 
on the space $\frh$ of flat sections is determined by pole orders at $z=\infty$. 
\end{remark}

\subsection{The complex conjugate opposite subspace} 
Over the marginal locus, 
the vector bundle $\HH\to \cM_\mar^\circ$ has a natural real
structure coming from the space of real semi-infinite cycles
\[ 
V_{f,z} = \varprojlim_{M\in \R_+}  
H_n(\C^n, \{x: \Re(f(x)/z)< -M\};\R) 
\cong \R^N,
\]
where the homology groups form an inverse system with respect to the
natural order on $\R_+$ and the limit is the projective (or
inverse) limit of vector spaces. The vector spaces $V_{f,z}$ form a
real vector bundle on $\cM_\mar^\circ \times \C^\times$ equipped with a
flat Gauss--Manin connection. For each $f\in \cM_\mar$, let us
denote by $\HH(f;\R)\subset \HH(f)$ the real
vector subspace consisting of $\omega$ such that 
\begin{equation} 
\label{eq:reality_condition} 
(-2\pi z)^{-n/2} \int_{\alpha} e^{f/z} \omega \in
\R \qquad \forall \alpha\in V_{f,z}, \quad \forall z \in S^1 
\end{equation} 
or equivalently, 
\[
\widehat{s}(\omega,z) \in \frh_\R \qquad \forall z\in S^1 
\]
where $S^1=\{|z|=1\}$ is the unit circle. 
Let $\kappa\colon \HH(f)\to \HH(f)$ be the complex conjugation 
corresponding to the real subspace $\HH(f;\R)$. 
The main properties of the complex conjugation $\kappa$ can be
summarized as follows (see \cite{Iritani:ttstar} for generalities 
on the real structure in a semi-infinite VHS). 
Let us denote by 
\[
\gamma\colon \C[z,z^{-1}] 
\to \C[z,z^{-1}], \quad
\gamma(g)(z):=\overline{g(\overline{z}^{-1})} 
\]
the complex conjugation corresponding to the real subspace consisting 
of Laurent polynomials that take real values on $|z|=1$. 
By definition, we have 
\ben
\kappa(g\omega) = \gamma(g)\kappa(\omega). 
\een 

\begin{remark}
\label{rem:realstr_awayfrom_marginal}
The definition \eqref{eq:reality_condition} for $[\omega]$ to be real
extends to non-marginal polynomials $f\in \cM\setminus \cM^\circ_\mar$; 
however our algebraic model $\HH(f)$ is not necessarily closed under the 
real involution $\kappa$. In general, $\kappa$ is defined on the 
analytification $\HH(f)^{\rm an} = \HH(f) \otimes_{\C[z,z^{-1}]} 
\cO^{\rm an}(\C^\times)$ and $\HH(f;\R)$ can be only 
defined as a subspace of 
$\HH(f)^{\rm an}$, where $\cO^{\rm an}(\C^\times)$ denotes 
the ring of holomorphic functions on $\C^\times$. 
\end{remark}

Complex conjugation in the vanishing cohomology $\frh$ gives a
natural splitting of the Steenbrink's Hodge filtration:
\beq\label{U-conj}
U_p\frh_s := \overline{F^{m-p}\frh}_{\bar s} \qquad 
\text{for $p\in \Z$, $|s|=1$}, 
\eeq
where $m=n-1$ for $s\neq 1$ and $m=n$ for $s=1$. 
\begin{proposition}\label{kappa-h}
Let $f\in \cM_\mar^\circ$ be a marginal deformation. 
The subspace $z^{-1} \kappa(\HH_+(f))$ is opposite and 
corresponds to the complex conjugate opposite filtration \eqref{U-conj} 
under \eqref{eq:opposite_correspondence}. 
\end{proposition}
\begin{proof} 
Let $\sigma \colon \frh \to \HH(f)$ be the 
splitting \eqref{eq:sigma_splitting} 
defined by the complex conjugate filtration \eqref{U-conj}. 
It suffices to show that $\sigma(\frh)$ is $\kappa$-invariant; 
indeed this implies $\sigma(\frh)[z^{-1}] = \kappa(\sigma(\frh)[z]) 
= \kappa (\HH_+(f))$. 
We will prove that 
\begin{equation} 
\label{kappa-A}
\kappa(\sigma(A)) = \sigma(\overline{A}) 
\qquad 
\forall A \in \frh. 
\end{equation} 
Recall from the definition of $\widehat{s}(\omega,z)$ in 
\S\ref{pf-hrp} that 
we have 
\[
(-2\pi z)^{-n/2} \int_\Gamma e^{f/z} \omega = \int_{\Gamma_1} 
\widehat{s}(\omega,z)
\]
and thus the map $[\omega]\mapsto \widehat{s}(\omega,z)$ 
intertwines the complex conjugation $\kappa$ on $\HH(f)|_{S^1}$ 
with the standard complex conjugation on $\frh$.  
This implies $\psi(\overline{A}) = \kappa(\psi(A))$ since $\psi$ 
is inverse to the map $[\omega]\mapsto \widehat{s}(\omega,z)$. 
For $A \in F^p \frh_s \cap U_p\frh_s$ with $s=e^{2\pi\sqrt{-1}\alpha}$, 
$0\le \alpha <1$, we have 
\[
\kappa(\sigma(A)) = \kappa( z^{-p-\alpha + \frac{n}{2}} \psi(A))  
= z^{p+\alpha-\frac{n}{2}} \psi(\overline{A}). 
\]
Here $\overline{A} \in F^q \frh_{\bar{s}} \cap U_q\frh_{\bar{s}}$ 
for $q$ with  
$p+\alpha -\frac{n}{2} = -q-\fracp{-\alpha}+\frac{n}{2}$. 
Therefore, the above quantity equals 
$z^{-q-\fracp{-\alpha}+\frac{n}{2}} \psi(\overline{A})
=\sigma(\overline{A})$. 
The proposition is proved. 
\end{proof}

\subsection{Opposite subspaces for Fermat polynomials}
In this subsection we will assume that
\ben
f(x)=x_1^{N_1+1}+\cdots + x_n^{N_n+1}
\een
is a Fermat polynomial. The higher residue pairing $K_f$
factorizes into a tensor product of the higher residue pairings of the
summands $x_i^{N_i+1}$. Using a simple degree count it is easy to see
that the forms (see also \cite[Theorem 2.10]{HLSW})
\begin{equation} 
\label{eq:Fermat_goodbasis} 
x_1^{i_1}\cdots x_n^{i_n} dx_1\cdots dx_n,\quad 0\leq i_s\leq N_s-1
\end{equation} 
form a good basis, i.e., if we denote by $H\subset \HH_+(f)$
the subspace spanned by the above forms, then $P=z^{-1}H[z^{-1}]$ is
an opposite subspace.  
\begin{proposition}
The complex conjugate subspace $\kappa(\HH_+(f))$ equals the subspace $zP$ 
spanned by the forms \eqref{eq:Fermat_goodbasis} over $\C[z^{-1}]$.   
\end{proposition}
\begin{proof} 
Since $P$ is an opposite subspace, we have
$\HH_+(f)=H[z]$. Therefore, it is enough to prove that
$\kappa(H)\subset H$. On the other hand, note that if $f=f_1\oplus
f_2$ is a direct sum of the quasi-homogeneous functions
$f_i\colon \C^{n_i}\to \C$, then the direct product of
cycles defines an isomorphism
\[
H_{n_1}(\C^{n_1},\{\Re(f_1/z)\ll 0\})\otimes
H_{n_2}(\C^{n_2},\{\Re(f_2/z)\ll 0\})\cong
H_{n}(\C^n, \{\Re(f/z)\ll 0\}).
\]
Moreover, if $\omega_i\in \HH_+(f_i)$, $i=1,2$, then
$\omega_1\wedge\omega_2\in \HH_+(f)$ and 
\ben
\kappa(\omega_1\wedge \omega_2) = \kappa_{f_1}(\omega_1)\wedge
\kappa_{f_2}(\omega_2), 
\een
where on the RHS we use the index $f_i$ ($i=1,2$) to indicate that the
conjugation is in the corresponding twisted de Rham cohomology. This
observation reduces the proof of our Proposition to the case $n=1$,
i.e., we may assume that $f(x)=x^{N+1}$. The oscillatory integrals are
very easy to compute explicitly and we can verify directly that 
\ben
\kappa\left( \frac{x^i dx}{\Gamma\left(\tfrac{i+1}{N+1}\right)} \right) = 
 \frac{x^{N-1-i}dx}{\Gamma\left(\tfrac{N-i}{N+1}\right)}. 
\een
Alternatively, we could argue that the opposite subspace $P$
corresponds to a splitting of the Steenbrink's Hodge filtration of $f$. 
However, in this case $\frh_1=0$ and $F^p\frh=0$
for $p>0$ and $F^p\frh=\frh$ for $p\leq 0$. 
Note that there is a unique monodromy invariant filtration $U_\bullet$ which
gives a splitting: $U_p\frh = \frh$ for $p\geq 0$ and
$U_p\frh = 0$ for $p<0$. Using Proposition \ref{kappa-h}, we
get that $P=\kappa( \HH_+(f))z^{-1}$. 
\end{proof} 

\begin{remark}
The results of this section can be generalised to any invertible polynomial.
\end{remark}

\subsection{The Cecotti-Vafa structure}\label{CV}
Cecotti and Vafa introduced $tt^*$-geometry for 
$N=2$ supersymmetric quantum field theories  
\cite{Cecotti-Vafa:top_anti_top, Cecotti-Vafa:classification}. 
This structure has been studied in mathematics by Dubrovin 
\cite{Dubrovin:top_anti_top}, Hertling \cite{He2} and 
many others. 
The Cecotti-Vafa structure for isolated hypersurface singularities has been 
introduced in \cite{He2}. We describe the structure 
for weighted homogeneous polynomials 
using the complex conjugate opposite subspaces. 

\begin{proposition}
If $f\in \cM_{\rm mar}^\circ$, then the subspace $
z^{-1}\kappa(\HH_+(f))$ is an opposite subspace and 
\ben
h(\omega_1,\omega_2) = K^{(0)}(\kappa(\omega_1),\omega_2) 
\een
is a positive-definite Hermitian pairing on 
\ben
\K(f):=\HH_+(f)\cap \kappa(\HH_+(f)).
\een
\end{proposition}
\begin{proof} 
According to Proposition \ref{kappa-h}, $z^{-1} \kappa(\HH_+(f))$ 
is an opposite subspace. Thus we have the corresponding 
splitting (see \eqref{eq:sigma_splitting}) 
\ben
\sigma\colon \frh\to \HH_+(f)\cap \kappa(\HH_+(f)).
\een 
Moreover, using formula \eqref{kappa-A}, we have
\ben
h(\omega,\omega) = K^{(0)}(\sigma(\overline{A}),\sigma(A)),
\een
where $\omega=\sigma(A)$. Let us assume that $A\in
F^p\frh_s\cap U_p\frh_s$. 
Then using Lemma \ref{sigma-K}, we get that the above pairing is 
\ben
C(s) \sqrt{-1}^{2p-m}S(\overline{A},A).
\een
Recalling that $F^p\frh$ is a Polarized Hodge Structure (see 
property (d) in Section \ref{S-PHS}), 
we get that the above number is a positive real number. 
\end{proof} 

By Remark \ref{rem:realstr_awayfrom_marginal}, 
we can extend the real structure $\kappa$ over the whole space 
$\cM$ by extending scalars. 
Following the notation there, we write 
$\HH_+(f)^{\rm an} = \HH_+(f) \otimes_{\C[z]} \cO^{\rm an}(\C)$. 
Because the oppositeness is an open condition, 
the subspace $z^{-1} \kappa(\HH_+(f)^{\rm an})$ 
is opposite to $\HH_+(f)^{\rm an}$ for $f$ 
in a neighborhood of $\cM_\mar^\circ$. 
Moreover, the Hermitian form $h(\omega_1,\omega_2) 
= K(\kappa(\omega_1), \omega_2)$ on 
the vector space 
\[
\K(f) := \HH_+(f)^{\rm an} \cap z \kappa(\HH_+(f)^{\rm an})
\]
is positive definite for $f$ in a neighborhood of $\cM_\mar^\circ$. 
On the other hand, Sabbah \cite[\S 4]{Sabbah:FL} proved that the 
Brieskorn lattice of any cohomologically tame function on a smooth affine manifold 
with only isolated critical points  satisfies these properties, 
i.e.~the oppositeness and the positive-definiteness of $h$. 
Therefore we have a globally defined Hermitian 
$C^\infty$ vector bundle $\K\to \cM$ whose fiber at $f\in \cM$ 
is $\K(f)$. 
The Gauss--Manin connection $\nabla$ on $\HH$ induces a family of flat 
connections of $\K$ depending on a parameter $z\in \C^\times$, 
which will be called the {\em Cecotti--Vafa connection}. 
Namely, let us pick a $C^\infty$-frame $\{\omega_i\}_{i=1}^N$ 
for $\K$. 
The deformation space $\cM$, being a Zariski open subset of a standard complex
vector space, has a natural holomorphic coordinate system
$\sigma=(\sigma_1,\dots, \sigma_{N'})$ and the vector fields
\ben
\partial/\partial \sigma_1,\dots,\partial/\partial \sigma_{N'},
\partial/\partial \overline{\sigma}_1,\dots,\partial/\partial \overline{\sigma}_{N'}
\een
give a frame for the complexified tangent bundle
$T^\C\cM:=T\cM\otimes_\R\C$. 
The properties $\nabla_X \HH_+(f) \subset z^{-1} \HH_+(f)$, 
$\nabla_{z\partial_z} \HH_+(f) \subset z^{-1} \HH_+(f)$ 
imply that 
\begin{align*} 
\nabla_X (\kappa \HH_+(f)) & = 
\kappa(\nabla_{\overline{X}} \HH_+(f)) 
\subset z \kappa(\HH_+(f))  \\ 
\nabla_{z\partial_z} \kappa(\HH_+(f)) & = 
\kappa(\nabla_{-z\partial_z} \HH_+(f)) 
= z \kappa(\HH_+(f))  
\end{align*} 
for a complexified vector field $X \in T^\C\cM$, 
where we used $\nabla_X\kappa =
\kappa \nabla_{\overline{X}}$ 
and $\kappa (z^{-1}\omega) = z \kappa(\omega)$. 
From these properties we get that in the frame $\{\omega_i\}$ 
the Gauss--Manin connection takes the form
\begin{align*} 
\nabla_i \omega_a & =  \sum_{b=1}^N \Big(\Gamma_{ia}^b(\sigma) + z^{-1}
C_{ia}^b(\sigma)\Big) \omega_b \\
\nabla_{\bar\imath} \omega_{a} & =  \sum_{b=1}^N
\Big(\Gamma_{\bar\imath a}^{ b}(\sigma) + z
\widetilde{C}_{\bar\imath a}^{ b}(\sigma)\Big) \omega_{ b} \\
\nabla_{z\partial_z} \omega_a & =  \sum_{b=1}^N \Big(
-U_a^b(\sigma)z^{-1}+Q_a^b(\sigma)+\widetilde{U}_a^b(\sigma) z\Big)
\omega_b
\end{align*} 
where $\nabla_i:=\nabla_{\partial/\partial \sigma_i}$,
$\nabla_{\bar\imath}:=\nabla_{\partial/\partial \overline{\sigma}_i}$,
and the connection matrices are $C^\infty$ functions in
$\sigma$. It is easy to prove that 
\ben
D = d+\sum_{i=1}^{N'} \Big(\Gamma_i d\sigma_i 
+\Gamma_{\bar\imath}d\overline{\sigma}_i\Big)
\een
is compatible with the Hermitian metric $h$ 
and its $(0,1)$ part defines the holomorphic structure on 
$\K \cong \HH_+/z\HH_+$ 
and its $(1,0)$ part defines the anti-holomorphic structure 
on $\K \cong \kappa(\HH_+)/z^{-1} \kappa(\HH_+)$.  
Here $\Gamma_i,\Gamma_{\bar\imath}\in \End(\K)$ are defined by
$\Gamma_i\omega_a=\sum_b \Gamma_{ia}^b\omega_b$ and 
$\Gamma_{\bar\imath}\omega_a=\sum_b \Gamma_{\bar\imath a}^b\omega_b$.
Using the compatibility of the
Gauss--Manin connection with the complex conjugation
\[
\nabla_{\overline{X}} = \kappa \nabla_X \kappa,\quad 
\nabla_{z\partial_z} = -\kappa \nabla_{z\partial_z} \kappa
\]
we get the following relations between the connection matrices
\[
\widetilde{C}_{\bar\imath} = \kappa C_i \kappa, \quad
\widetilde{U} = \kappa U \kappa,\quad
\widetilde{Q} = -\kappa Q \kappa.
\]
\begin{remark} 
When we choose $\{\omega_i\}$ to be a frame that  
is holomorphic under the identification $\K \cong \HH_+/z\HH_+ 
\cong \bigcup_f \Jac(f) \cdot dx$, then we have 
$\Gamma_{\bar\imath} =0$. 
Moreover the operators $C_i$ and $U$ above correspond to 
the multiplication on $\Jac(f) \cdot dx$ by $\partial_{\sigma_i} f(x;\sigma)$ and $f$, respectively. In particular, 
$U=0$ along the marginal locus $\cM_\mar^\circ$. 
\end{remark}

\section{Quantization and Fock bundle}

Using Givental's quantization formalism \cite{G2}, we define a vector
bundle of Fock spaces on the moduli space $\mathcal{M}_{\rm
  mar}^\circ$. The main motivation for our definition is to provide a
convenient language to state mirror symmetry as well as to investigate
the transformation properties under analytic continuation of
Givental's total ancestor potential.

\subsection{Givental's quantization formalism}
Let $H$ be a complex vector space of dimension $N$ equipped with a non-degenerate bi-linear
pairing $(\ ,\ )$. Givental's quantization is based on the vector
space $\mathcal{H}=H(\!(z)\!)$ equipped with the following
symplectic structure:
\ben
\Omega(\mathbf{f}_1(z),\mathbf{f}_2(z)) = \Res_{z=0} 
(\mathbf{f}_1(-z),\mathbf{f}_2(z)) dz.
\een
The Lie algebra of infinitesimal symplectic transformations $A$ of
$\mathcal{H}$ is naturally identified with the Poisson Lie algebra of
quadratic Hamiltonians via
\ben
A\mapsto h_A(\mathbf{f}):=\frac{1}{2}\Omega(A\mathbf{f},\mathbf{f}).
\een
Note that $\mathbf{f}\mapsto A\mathbf{f}$ can be interpreted as a
vector field on $\mathcal{H}$. This vector field is Hamiltonian with
Hamiltonian $h_A$ if and only if $A$ is an infinitesimal symplectic
transformation. 

The symplectic vector space has a natural polarization
$\mathcal{H}=\mathcal{H}_+\oplus \mathcal{H}_-$, where
$\mathcal{H}_+:=H[\![z]\!]$ and $\mathcal{H}_-:=H[z^{-1}]z^{-1}$ are
Lagrangian subspaces. The polarization allows us to use the so-called
{\em canonical quantization} to represent quadratic Hamiltonians by
differential operators. In coordinates, the representation can be
constructed as follows. Let $\{\phi_i\}_{i=1}^N$ and
$\{\phi^i\}_{i=1}^N$ be bases of $H$ dual with respect to the pairing
$(\ ,\ )$. Then the linear functions on $\mathcal{H}$ defined by 
\ben
p_{k,i}(\mathbf{f})= \Omega(\mathbf{f},\phi_i z^k),\quad
q_{k,i}(\mathbf{f})=\Omega(\phi^i(-z)^{-k-1},\mathbf{f}),\quad 1\leq i\leq N,k\geq 0,
\een
form a Darboux coordinate system. We define
$\widehat{A}:=\widehat{h}_A$, where a function in $p_{k,i}$ and
$q_{k,i}$ is quantized by the rules
\ben
p_{k,i}\mapsto \hbar^{1/2}\frac{\partial}{\partial q_{k,i}},\quad
q_{k,i}\mapsto \hbar^{-1/2} q_{k,i}
\een
and normal ordering, i.e., all differentiation operations should
preceed the multiplication ones. 

If $R$ is a symplectic transformation of $\mathcal{H}$ of the form
$1+R_1 z+R_2 z^2+\cdots$, where $R_k\in {\rm End}(H)$, then we can
formally define $A=\log R$ and $\widehat{R}=e^{\widehat{A}}.$
Let us introduce the quadratic differential operator
\ben
V_R:=\sum_{k,\ell=0}^\infty\sum_{i,j=1}^N
(V_{k\ell}\phi^j,\phi^i)
\frac{\partial^2}{\partial q_{k,i}\partial q_{\ell,j}},
\een
where $V_{k\ell}\in {\rm End}(H)$ are defined by 
\ben
\sum_{k,\ell=0}^\infty V_{k\ell} z^kw^\ell = \frac{1-R(z)R^t(w)}{z+w}.
\een. 
\begin{proposition}\label{R-action}
Let $\cF=\cF(\hbar;\mathbf{q})$ be a formal power series
in $\mathbf{q}=(q_{k,i})$ with coefficients in
$\C_\hbar=\C(\!(\hbar)\!)$ such that
$\widehat{R}^{-1}\cF$ is well defined. Then 
\ben
\widehat{R}^{-1}\cF=
\left.
\Big(e^{\frac{\hbar}{2}\, V_R}\cF\Big)
\right|_{\mathbf{q}\mapsto R(z)\mathbf{q}},
\een
where $R(z)\mathbf{q}$ is defined by identifying $\mathbf{q}$ with a
vector $\sum_{k=0}^\infty \sum_{i=1}^N q_{k,i}\phi_i z^k\in H[\![z]\!]$.
\end{proposition}

\subsection{From an opposite subspace to a Frobenius structure}
\label{subsec:Frobenius} 

Let us denote by $\mathcal{L}\to \cM_\mar^\circ$ the line bundle whose
fiber over a point $f\in \cM_\mar^\circ$ is the space of elements in
$\HH_+(f)$ of minimal degree; i.e., 
$\mathcal{L}_f=\C\, dx_1\cdots dx_n$. 
We refer to $\mathcal{L}$ as the {\em vacuum line bundle}. 

Assume now that $f\in \cM_{\rm mar}$ is a given point,  $\omega\in
\mathcal{L}_f$ is a non-zero form, and $P$ is a homogeneous opposite
subspace of $\HH_+(f)$.  Let us choose  a good basis of homogeneous
forms $\{\omega\}_{i=1}^N\subset \HH_+(f)\cap Pz$ and define $\phi_i\in
\operatorname{Jac}(f)$ such that\footnote
{The differential form $\phi_i \omega$ depends on the choice of 
a representative of $\phi_i \in \Jac(f)$, but the class of 
$\phi_i \omega$ in $\HH_+(f)/z\HH_+(f)$ does not.}
\ben
\omega_i \equiv \phi_i \, \omega \mod z\HH_+(f),\quad
1\leq i\leq N.
\een
We construct a miniversal unfolding of $f$ by 
\ben
F(x,t)=f(x)+\sum_{i=1}^N t_i\phi_i(x),\quad t=(t_1,\dots,t_N)\in B_f,
\een 
where $B_f\subset \mathbb{C}^N $ is a sufficiently small ball
representing the holomorphic germ at $0$ of $\C^N$ and $\phi_i(x)$ is
a homogeneous polynomial representing $\phi_i\in
\operatorname{Jac}(f)$. Let us assign a degree
to  $t_i$ such that $F(x,t)$ is weighted homogeneous of degree 1, and let us
split the deformation parameters $t$ into 3 groups: relevant $t^{\rm
  rel}=(t_1,\dots,t_{N_{\rm rel}})$, marginal $t^{\rm mar}=(t_{N_{\rm
    rel}+1},\dots,t_{N_{\rm rel} + N_{\rm mar} }),$ and irrelevant $
t^{\rm irr}=(t_{N_{\rm rel} + N_{\rm mar}  +1},\dots,t_N)$, depennding
on whether their degrees are 
respecively $>0$, $=0$, or $<0$. 
There is a natural way to construct a Frobenius structure on 
$B_f$. Let us outline
the construction referring for more details to \cite{He, KS, ST}. To
begin with, we choose an appropriately small Stein domain
$X_f\subset \C^n\times B_f$ around $0$ (see \cite{AGV}). Let
us denote by $F\colon X_f\to \C$ the miniversal unfolding
of $f$ and put 
$\widehat{\mathcal{F}}:=\mathbb{R}^n\varphi_*(\widehat{\Omega}_{\rm 
  twdR},\widehat{d}_{\rm twdR})$  for the hypercohomology of the
twisted de Rham complex 
\ben
(\widehat{\Omega}_{\rm twdR}^\bullet,\widehat{d}_{\rm twdR}):=
(\Omega^\bullet_{X_f/B_f}[\![z]\!],zd_{X_f/B_f}+dF\wedge),
\een
where $\varphi\colon X_f\subset \C^n\times B_f \to B_f$ is the natural
projection. Following the argument in Section \ref{sec:twdR} it is
easy to prove that  after decreasing $B_f$ if necessary,
$\widehat{\mathcal{F}}$ is a
trivial vector bundle on $B_f$, whose fiber over a point $s\in B_f$ 
\ben
\hHH_+(F):=
\Omega^n_{X_f}[\![z]\!]/(zd+dF\wedge)
\Omega^{n-1}_{X_f}[\![z]\!],\quad F=F(x,s),
\een
is a free $\C[\![z]\!]$-module of rank the Milnor number $N$. 
Put 
\ben
\hHH(F):=\hHH_+(F) \otimes_{\C[\![z]\!]} 
\C(\!(z)\!).
\een 
This way we obtain vector bundles
$\hHH_+\subset\hHH$ on $B_f$. For a given holomorphic
function $g(x,z)\in \mathcal{O}_{X_f}(X_f)\otimes \C(\!(z)\!)$ let us denote by 
\ben
[g(x,z)dx]_F=:\int e^{F(x)/z}g(x,z) dx
\een
the equivalence class of the form $g(x,z)dx_1\cdots dx_n$ in the de Rham
cohomology group $\hHH(F)$.  

\subsubsection{Extension in the relevant and marginal directions}
There is a unique way to extend $\omega_i=[g_i(x,z)dx]_f$ to sections
$\widetilde{\omega}_i$ of $\widehat{\HH}_+|_{\{t^{\rm irr}=0\}}$ so that they give
a good basis in each fiber $\widehat{\HH}_+(F)$ for $F\in \{t^{\rm
  irr}=0\}\subset B_f$. 
The extension can be constructed by Birkhoff factorization as
follows. Let us denote by $G_i$ the section of $\widehat{\HH}_+$ obtained
by flat extension of $\omega_i$ with respect to the Gauss--Manin
connection. The Gauss--Manin connection $\nabla$ gives rise to a
system of differential equations
\ben
z\nabla_{\partial/\partial t_i} [g_j(x,z)dx]_F=\sum_{k=1}^N
\Gamma_{ij}^k(t,z) [g_k(x,z)dx]_F,\quad 1\leq i\leq N_{\rm rel}+N_{\rm mar}.
\een
Since $f$ is weighted-homogeneous, the
functions $g_i(x,z)$ are polynomials in $z$. In particular, the connection
matrix $\Gamma$ is holomorphic at $(t,z)=(0,0)$. Let us pick a
fundamental solution $\Phi(t,z)$; i.e., a $N\times N$ non-degenerate
matrix solving the differential equations
\beq\label{GM:system}
z\partial_{t_i} \Phi(t,z) = \Gamma_i(t,z)\Phi(t,z),\quad 1\leq i\leq
N_{\rm rel}+N_{\rm mar},
\eeq
where $\Gamma_i(t,z)$ is the matrix whose $(j,k)$-entry is
$\Gamma_{ij}^k(t,z)$ and satisfying $\Phi(0,z)=1$. We have
\ben
[g_i(x,z)dx]_F = \sum_{j=1}^N \Phi_{ij}(t,z) G_j.
\een
Note that $\Phi(t,z)$ is a holomorphic matrix for $z\in
\C^*:=\mathbb{P}^1\setminus{\{0,\infty\}}$ that has a Birkhoff
factorization at $t=0$, so $\Phi(t,z)$ must have a Birkhoff
factorization for all $t\in B_f$ provided we choose $B_f$ sufficiently
small. Put $\Phi(t,z)=\Phi_+(t,z)^{-1}\Phi_-(t,z)$, where $\Phi_-(t,z)$ is
holomorphic for $z\in \mathbb{P}^1\setminus{\{0\}}$ (with
$\Phi_-(t,\infty)=1$) and $\Phi_+(t,z)$ is holomorphic for
$z\in \mathbb{P}^1\setminus{\{\infty\}}$. One can check that the forms
\ben
\widetilde{\omega}_i = \sum_{j=1}^N (\Phi_+(t,z))_{ij}[g_j(x,z)dx]_F
,\quad 1\leq i\leq N,
\een 
give rise to a good basis. Moreover,  the good basis
$\widetilde{\omega}_i=[\widetilde{g}_i(x,t^{\rm rel},t^{\rm
  mar};z)dx]_F$ depends polynomially on $x,t^{\rm rel}$ and $z$
(because these variables have positive degrees) and
analytically in $t^{\rm mar}$.
\subsubsection{Extension in the irrelevant directions}
If we want to extend in the irrelevant directions, then the above
argument becomes much more involved, because the system
\eqref{GM:system} might fail to be convergent and holomorphic in $z$. To offset this
difficulty one can take the formal Laplace transform, solve the
resulting system, and then obtain $\widetilde{\omega}_i$ via the
inverse Laplace transform. The details are quite delicate, so we refer
to \cite{He, MS}. An alternative way to proceed is the perturbative
approach of \cite{LLS}. The main idea is to look for a good basis that
depends formally on the irrelevant parameters, i.e., we are looking
for a good basis of the form 
\ben
\widetilde{\omega}_i =[\widetilde{g}_i(x,t,z)dx]_F,\quad 
\widetilde{g}_i\in \C\{t^{\rm mar}\}[x,t^{\rm rel}][\![t^{\rm \irr},z]\!],
\een  
where $\C\{a\}$ is the ring of convergent power series in
$a$. According to \cite{LLS}, first we have to find the  extension
$\widetilde{\omega}=\widetilde{g}(x,t,z) dx$ of the volume 
form $\omega\in \mathcal{L}_f-\{0\}$ by solving the following
equation in $\hHH_+(f)$ :
\ben
J(t,z):=e^{(F(x,t)-f(x,t_{\rm rel},t_{\rm
    mar}))/z}\widetilde{g}(x,t,z) dx \in \omega+H[\![z^{-1}]\!]z^{-1},
\een
where $H=\HH_+(f)\cap zP$ is the vector space spanned by the
good basis $\{\omega_i\}_{i=1}^N$. 
The extension of the remaining forms is obtained from the period map 
\beq\label{period-map}
TB_f[\![z]\!] \to
\hHH_+,\quad \partial_{t_i}\mapsto z\nabla_{\partial_{t_i}} [\widetilde{\omega}]
\eeq
as the image of the {\em flat} vector fields. The latter are the vector
fields corresponding to the coordinate system on $B_f$ given by the
coefficients in front of $z^{-1}$ of $J(t,z)$. 
We define a Frobenius structure on $B_f$ for which a basis of flat vector
fields corresponds via the period map \eqref{period-map} to the good
basis $\{\widetilde{\omega}_i\}_{i=1}^N$ and the flat pairing
corresponds to $K_F^{(0)}$. Let us point out that the extension
$\widetilde{\omega}$ of the volume form $\omega$ is a {\em primitive form}
in the sense of K. Saito. Slightly abusing the terminology we will
sometimes refer to the sections of $\mathcal{L}$ as primitive forms,
keeping in mind that they do become primitive only after an
appropriate extension.

\subsection{The total ancestor potential}

Given $f\in \cM_{\rm mar}$, $\omega\in
\mathcal{L}_f\setminus{\{0\}}$, an opposite subspace
$P\subset \HH(f)$, and a good basis $\{\omega_i\}_{i=1}^N\subset
\HH_+(f)\cap Pz$, let us construct a miniversal unfolding space $B_f$
equipped with a Frobenius structure as explained above. Using the flat
structure we trivialize the tangent and the co-tangent bundle 
\ben
T^*B_f\cong T B_f \cong B_f\times T_0 B_f \cong B_f\times {\rm Jac}(f),
\een
where the first isomorphism uses the non-degenerate pairing, the
2nd one uses the flat Levi-Civita connection, and the last one is
induced from the period isomorphism $T_0B_f\cong \HH_+(f)/z
\HH_+(f)$ and the isomorphism 
\beq\label{Jac-triv}
 {\rm Jac}(f)\cong
\HH_+(f)/z
\HH_+(f),\quad \phi(x)\longmapsto \phi(x) \omega \mod z\HH_+(f).
\eeq
Note that $\phi_i\in  {\rm Jac}(f)$ are the elements corresponding
to the good basis $\omega_i$ via the isomorphism  \eqref{Jac-triv}. 
Let us introduce the Fock space
\ben
\C_\hbar[\![q_0,q_1+\mathbf{1},q_2,\dots]\!] 
= \C(\!(\hbar)\!)[\![q_0,q_1+\mathbf{1},q_2,\dots]\!] 
\een 
of formal power series in
$\mathbf{q}=(q_{k,i})_{i=1,\dots,N}^{k=0,1,\dots}$. We denote
$q_k=\sum_{i=1}^N q_{k,i}\phi_i$. The shift $q_1+\mathbf{1}$ means
that the element $\mathbf{1}\in {\rm Jac}(f)$ should be written as $\mathbf{1}
= \sum_i a_i \phi_i$ and in the formal power series the variables
$q_{1,i}$ are shifted to $q_{1,i}+a_i$.
 
On the other hand, if $F$ is a generic deformation of $f$, then the
critical values of $F$ give rise to the so-called {\em canonical} coordinate system
$u=(u_1,\dots,u_N)$, defined locally near $F$, in which the Frobenius
multiplication and the pairing are diagonal
\ben
(\partial_{u_i},\partial_{u_j}) =
\delta_{i,j}/\Delta_i,\quad \partial_{u_i}\bullet \partial_{u_j}=\delta_{i,j}\partial_{u_j} . 
\een
Let us denote by 
\ben
\Psi_F\colon \C^N \to 
T_F B \cong {\rm Jac}(f),\quad \Psi_F(e_i) =\sqrt{\Delta_i}\partial_{u_i}
\een
the trivialization of the tangent bundle at a generic $F$. 
The total ancestor potential
is an element of the Fock space defined by 
\ben
\cA_F(\hbar;\mathbf{q}) := \widehat{\Psi}_F \widehat{R}_F
\prod_{i=1}^N \mathcal{D}_{\rm pt} (\hbar\Delta_i;{}^i\mathbf{q}\sqrt{\Delta_i}),
\een
where $R_F$ is a symplectic transformation of
$\C^N(\!(z)\!)$ of the type
$1+R_1z+\cdots$, which will be defined below. We have a different set
of formal variables ${}^i\mathbf{q}=({}^iq_0,{}^iq_1,\dots)$ which is
related to the previous one by 
\ben
\sum_{i=1}^N {}^iq_k \Psi_F(e_i) = q_k,\quad k\geq 0.
\een
By definition the quantization $\widehat{\Psi}_F$ acts by the above
substitution. Finally, $\mathcal{D}_{\rm pt}$ is the
Witten--Kontsevich tau function:
\ben
\mathcal{D}_{\rm pt} (\hbar;\mathbf{q})= 
\exp \Big( \sum_{g,n=0}^\infty \frac{\hbar^{g-1}}{n!}
\int_{\overline{\cM}_{g,k}} \prod_{i=1}^n (\mathbf{q}(\psi_i)+\psi_i)\Big),
\een
where $\mathbf{q}=(q_k)_{k\geq 0}$ is a sequence of formal variables
and $\mathbf{q}(\psi_i)=\sum_{k=0}^\infty q_k \psi_i^k$  (with $\psi_i$
the 1st Chern class of the line bundle of $i$-th marked point cotangent lines) is a
cohomology class on $\overline{\cM}_{g,k}$. Note that the dilaton
shift is incorporated here, so the function is an element of 
$\C_\hbar[\![q_0,q_1+1,q_2,\dots]\!]$.

The operator $R_F$ is in general defined as a formal solution of the
Gauss--Manin connection near the irregular singular point $z=0$. In
the case of singularity theory, however, we have an alternative
description in terms of a stationary phase asymptotic. Namely, let
$\beta_i\subset \C^n$ be the cycle swept by the vanishing
cycle vanishing at the critical point of $F$ corresponding to the
critical value $u_i$, then the stationary phase asymptotic 
\beq\label{asymptotic}
(-2\pi z)^{-1/2}\int_{\beta_i} e^{F(x)/z} \widetilde{\omega}_a \sim
(\Psi_F R_F(z) e_i,\phi_a)e^{u_i/z},\quad z\to 0,
\eeq
where $\phi_a\in {\rm Jac}(f)$ corresponds to the flat vector field
determined by $\widetilde{\omega}_a$. 

According to Milanov \cite{Mi}, the total ancestor potential $\cA_F$
extends analytically for all $F\in B_f$. In particular, we have a well
defined limit
\ben
\cA_{f,\omega}^{\omega_1,\dots, \omega_N}(\hbar;\mathbf{q})
:=\lim_{F\to f} \cA_F (\hbar;\mathbf{q}).
\een
Let us describe the dependence of the total ancestor potential on the
choices of $\omega$ and $\omega_1,\dots,\omega_N$. Assume that
$\omega'\in \mathcal{L}_f-\{0\}$, $P'\subset \HH(f)$ is
an opposite subspace, and $\{\omega_i'\}_{i=1}^N\subset 
\HH_+(f)\cap P'z $ is a good basis. It is convenient to split the general
formula into two cases. The first case is the following: 
if $P'=P$, then 
\beq\label{fr-change}
\omega_j' = \sum_{i=1}^N \omega_i B_{ij},\quad \omega'=c\omega
\eeq
for some invertible matrix $B=(B_{ij})_{i,j=1}^N$ and some non-zero constant $c$. 
The second case is the following: if $\omega'=\omega$ and 
\ben
\omega_i'\equiv \omega_i\equiv \phi_i \omega 
\mod z\HH_+(f),\quad 1\leq i\leq N,
\een
where $\{\phi_i\}_{i=1}^N\subset \operatorname{Jac}(f)$. 
Let us denote by $R(f,z)$ the linear operator
in $\Jac(f)(\!(z)\!)$ whose matrix $(R_{ij}(f,z))_{i,j=1}^N$
with respect to the basis $\{\phi_i\}_{i=1}^N$ is defined by 
\beq\label{gb-change}
\omega'_j(f)=\sum_{i=1}^N \omega_i (f) R_{ij}(f,z),\quad 1\leq j\leq N.
\eeq
Let us recall Givental's quantization formalism for 
\ben
H=\operatorname{Jac}(f),\quad
(\psi_1,\psi_2):=K^{(0)}(\psi_1\omega,\psi_2\omega),\quad
\psi_1,\psi_2\in \operatorname{Jac}(f).
\een
Note that $R(f,z)$ is a symplectic transformation of
$\mathcal{H}=H(\!(z^{-1})\!)$ of the type $1+R_1(f)z+R_2(f)z^2+\cdots$. 

\begin{proposition}\label{anc-change}
a) The transformation of the total ancestor potential under the change
\eqref{fr-change} is given by
\ben
\cA_{f,\omega'}^{\omega_1',\dots,\omega_N'}(\hbar;\mathbf{q})  =
\cA_{f,\omega}^{\omega_1,\dots,\omega_N}(\hbar c^{-2}; c^{-1} B \mathbf{q}) ,
\een 
where $(c^{-1}B\mathbf{q})_{k,i}=\sum_{j=1}^N c^{-1}B_{ij}q_{k,j}$. 

b) The transformation of the total ancestor potential under the change
\eqref{gb-change} is given by 
\ben
\cA_{f,\omega}^{\omega'_1,\dots,\omega'_N}(\hbar ;
\mathbf{q}) =(R(f,z)^t)^\wedge 
\cA_{f,\omega}^{\omega_1,\dots,\omega_N}(\hbar; \mathbf{q}),
\een
where $R(f,z)^t$ is the transpose of $R(f,z)$ with respect to the residue pairing.
\end{proposition}
\proof 
Let us denote by
\ben
\mathcal{A}_F'(\hbar;\mathbf{q}) = \widehat{\Psi}_F'\widehat{R}_F'\,
\prod_{i=1}^N \mathcal{D}_{\rm pt}(\hbar\Delta_i';{}^i\mathbf{q} \sqrt{\Delta'_i})
\een
the total ancestor potential corresponding to the Frobenius structure
determined by the primitive form $\omega'$ and the good basis
$\omega_1',\dots,\omega_N'$. 
 Let $\tau=(\tau_1,\dots,\tau_N)$ and
$\tau'=(\tau_1',\dots,\tau_N')$ be the flat coordinates on $B_f$ corresponding
respectively to the good bases $\{\omega_i\}_{i=1}^N$ and
$\{\omega_i'\}_{i=1}^N.$ By definition 
\ben
z\nabla_{\partial/\partial \tau_i} \omega =\omega_i,\quad 
z\nabla_{\partial/\partial \tau'_i} \omega' =\omega'_i,\quad 1\leq
i\leq N.
\een 

\medskip
a)
Recalling the change \eqref{fr-change} we get the following relations
\ben
\tau_i=\sum_{j=1}^N c^{-1}B_{ij} \tau_j',\quad 
\frac{\partial}{\partial \tau_j'} = \sum_{i=1}^N
c^{-1}B_{ij}\frac{\partial}{\partial \tau_i}.
\een
The matrix of $\Psi'_F\colon \mathbb{C}^N\to
\operatorname{Jac}(f)$ with respect to the bases
$\{e_i\}_{i=1}^N\subset \mathbb{C}^N$ and 
$\{\phi_i'\}_{i=1}^N\subset \operatorname{Jac}(f)$,
$\phi_i'=\partial/\partial \tau_i'$ has entries $(\Psi'_F)_{ki}$  defined by 
\ben
\Psi'_F e_i = \sum_{k=1}^N \phi_k'\, (\Psi'_F)_{ki},\quad
(\Psi'_F)_{ki} = \sqrt{\Delta_i} \, \partial \tau'_k/\partial u_i. 
\een
Similarly, $\Psi_F$ is represented by a matrix with entries
$(\Psi_F)_{ki}=\sqrt{\Delta_i} \,\partial \tau_k/\partial u_i.$
Note that $\sqrt{\Delta_i'} = c^{-1}\sqrt{\Delta_i}$, $R'_F=R_F$, and
$\Psi_F'=B^{-1}\,\Psi_F$. Using also that the quantized action
$\widehat{R}_F$ commutes with the rescaling 
\ben
\hbar\mapsto \hbar c^{-2},\quad {}^iq_k\mapsto {}^iq_k c^{-1},
\een
we get 
\ben
\mathcal{A}'_F(\hbar;\mathbf{q} ) = \mathcal{A}_F(\hbar c^{-2}, c^{-1} B\mathbf{q}).
\een
Taking the limit $F\to f$ completes the proof of part a). 

\medskip
b) The entries of the matrix of the linear operator $\Psi'_{F}R_F'$ with respect to
the bases $\{e_i\}_{i=1}^N\subset \mathbb{C}^N$ and
$\{\phi_a\}_{a=1}^N$ are by definition the stationary phase
asymptotics
\ben
(-2\pi z)^{-n/2} \int_{\beta_i} e^{(F-u_i)/z} \widetilde{\omega}'_b \,
\eta^{ab}\ \sim\
(\Psi'_FR'_F)_{ai},\quad z\to 0,
\een
where $\{\widetilde{\omega}_a'\}$ is the extension of the good basis
$\{\omega_a'\}_{a=1}^N$ to a good basis on $B_f$ and
$(\eta^{ab})_{a,b=1}^N$ is the inverse matrix of the matrix of the
residue pairing $(\eta_{ab})_{a,b=1}^N$, $\eta_{ab} = (\phi_a,\phi_b)$. Similarly, 
\ben
(-2\pi z)^{-n/2} \int_{\beta_i} e^{(F-u_i)/z} \widetilde{\omega}_b\, \eta^{ab} \ \sim\
(\Psi_FR_F)_{ai},\quad z\to 0.
\een
Let us denote by $\widetilde{R}(F,z)$ the symplectic transformation of
$\operatorname{Jac}(f)(\!(z)\!)$ whose entries
$\widetilde{R}_{ab}(F,z)$  with respect to the basis
$\{\phi_a\}_{a=1}^N$  are given by 
\ben
\widetilde{\omega}'_b = \sum_{a=1}^N \widetilde{\omega}_a \, \widetilde{R}_{ab}(F,z).
\een
By definition $\lim_{F\to f} \widetilde{R}_{ab}(F,z) = R_{ab}(f,z).$
Comparing the above asymptotic expansions, we get 
\ben
(\Psi'_FR'_F)_{ai} = \sum_{\mu,\nu=1}^N \eta^{a\mu}
\widetilde{R}_{\nu\mu}(F,z) \eta_{\nu b} \, (\Psi_F R_F)_{bi}.
\een
Note that the matrix of the transpose $\widetilde{R}(F,z)^t$ with
respect to the residue pairing has entries
\ben
(\widetilde{R}(F,z)^t)_{ab} = \sum_{\mu,\nu=1}^N \eta^{a\mu}
\widetilde{R}_{\nu\mu}(F,z) \eta_{\nu b}.
\een
We get that $\Psi'_FR'_F = \widetilde{R}(F,z)^t \Psi_F R_F$, so 
\ben
\mathcal{A}'_F(\hbar;\mathbf{q}) = (\widetilde{R}(F,z)^t)^\wedge\,
\mathcal{A}_F(\hbar;\mathbf{q}). 
\een
Taking the limit $F\to f$ completes the proof. \qed

\subsection{The abstract Fock bundle}

Recall that a series of the form 
\[
\sum_{g\in \mathbb{Z}}\sum_{\kappa=(k_1,i_1),\dots,(k_r,i_r)}\
c^{(g)}_\kappa \hbar ^{g-1} t_{k_1,i_1}\cdots t_{k_r,i_r},\quad
t_{k,i}=q_{k,i}+\delta_{k,1} a_i,
\] 
is called {\em tame} if $c^{(g)}_\kappa\neq 0$ only for $\kappa$ satisfying the
$(3g-3+r)$-jet constraint
\ben
k_1+\cdots + k_r \leq 3g-3+r.
\een
It is known that Givental's total ancestor potential
$\mathcal{A}_{f,\omega}^{\omega_1,\dots,\omega_N}(\hbar;\mathbf{q})$
is tame (see \cite{G4}). Motivated by Proposition \ref{anc-change} we
define a vector bundle 
$\widehat{\mathbb{V}}_{\rm tame}$ on $\mathcal{M}_{\rm
  mar}^\circ$ whose fibers are the Fock spaces
$\C_\hbar[\![q_0,q_1+\mathbf{1},q_2,\dots]\!]_{\rm
  tame}$ of tame series\footnote
{Note that elements of $\C[\![\hbar]\!]$ are tame, but $\hbar^{-1}$ is 
not tame; hence $\C_\hbar[\![q_0,q_1+\mathbf{1},q_2,\dots]\!]_{\rm tame}$ 
is not a $\C_\hbar$-algebra (only a $\C[\![\hbar]\!]$-algebra).} 
and the transition functions are given
by the transformation laws of Proposition \ref{anc-change} (with
$c=1$). Following Costello-Li's interpretation \cite{CS} 
of Givental's quantization formalism, we will
identify each fiber of $\widehat{\mathbb{V}}_{\rm tame}$ with a highest
weight module of a certain Weyl algebra, which in particular yields
an intrinsic definition of $\widehat{\mathbb{V}}_{\rm tame}$.

\subsubsection{The Weyl algebra and the Fock space}

Let us fix $f\in \cM_{\rm mar}^\circ$. 
The Weyl algebra of $\HH(f)$ is defined by 
\ben
\mathcal{W}(f)=\bigoplus_{n=0}^\infty 
\Big( \HH(f)^{\otimes n}\otimes \C_\hbar \Big)/ I,
\een
where $\C_\hbar = \C(\!(\hbar)\!)$ and 
$I$ is the two sided ideal generated by the elements 
\ben
a\otimes b-b\otimes a -\hbar \, \Omega(a,b),\quad a,b\in \HH(f).
\een
The Lagrangian subspace $\HH_+(f)$ determines the following Fock space 
\ben
\mathbb{V}(f):=\mathcal{W}(f)/\mathcal{W}(f) \HH_+(f).
\een
\begin{lemma}\label{Fock-coords}
If $P$ is an opposite subspace, then 
the natural map
\ben
\phi_{P}\colon \bigoplus_{n=0}^\infty
\Big(\operatorname{Sym}^n(P)\otimes
\C_\hbar \Big) \to \mathbb{V}(f),\quad a_1\cdots
a_n\mapsto a_1\otimes\cdots \otimes a_n,
\een
induces an isomorphism. 
\end{lemma}
\proof
The map is well defined and injective because $P$ is
Lagrangian. The surjectivity follows from the Wick's formula (see \cite{K2}). Namely,
given $a_1,\dots, a_n\in \HH(f)$ we have the following identity
in $\mathbb{V}(f)$:
\ben
a_1\otimes \cdots \otimes a_n = \sum \Big(\prod_{s=1}^{n'}
\Omega(a_{i_s'}^+,a_{i_s''}^-)\Big) a_{j_1}^-\otimes \cdots \otimes a_{j_{n''}}^-,
\een 
where the sum is over all possible ways to select pairs
$(i_1',i_1''),\dots, (i_{n'}',i_{n'}'')\subset \{1,2,\dots,n\}$ such that $i_s'<i_s''$ and
$i_1'<\cdots <i_{n'}'$ and $\{j_1,\dots,j_{n''}\} =
\{1,2,\dots,n\}\setminus \{ i_1',\dots ,i_{n'}',i_1'',\dots
,i_{n'}''\}$, and where $a^-\in P$ (resp. $a^+\in \HH_+(f)$) denotes the
projection of $a$ on $P$ (resp. $\HH_+(f)$) along
$\HH_+(f)$ (resp. $P$).
\qed

\subsubsection{The tame Weyl algebra and the tame Fock space}
If $P\subset \HH(f)$ is an opposite subspace, then using the
vector spaces isomorphism 
\ben
\mathcal{W}(f)=\bigoplus_{r,s=0}^\infty P^{\otimes r}\otimes
\HH_+(f)^{\otimes s} \otimes \C_\hbar 
\een
we can write an element of $\mathcal{W}(f)$ as a finite sum of terms
of the form
\beq\label{weyl:mon}
c^{(g)}_{I,J} \hbar^{g-1}\omega^{i_1}(-z)^{-k_1-1}\cdots
\omega^{i_r}(-z)^{-k_r-1}\otimes 
\omega_{j_1}z^{\ell_1}\cdots \omega_{j_s}z^{\ell_s},
\eeq
where $\{\omega^i\}$ and $\{\omega_j\}$ are dual bases of
$\HH_+(f)\cap zP$, and the coefficient $c^{(g)}_{I,J}$ is a
constant depending on $g$ and the
multi-indexes $I=\{(i_1,j_1),\dots,(i_r,k_r)\}$ and
$J=\{(j_1,\ell_1),\dots,(j_s,\ell_s)\}$. The {\em tame} Weyl algebra
$\mathcal{W}_{\rm tame} (f)$ 
is defined as the vector subspace of  $\mathcal{W}(f)$ spanned by
monomials of the type \eqref{weyl:mon} such that 
\beq\label{def:tame}
k_1+\dots+k_r - r\leq 3(g-1+s/2).
\eeq
Finally, we need to introduce the completion of $\mathcal{W}_{\rm
  tame} (f)$
\ben
\widehat{\mathcal{W}}_{\rm tame} (f):=\varprojlim_m \mathcal{W}_{\rm
  tame} (f)/\mathcal{W}^m_{\rm tame} (f), 
\een
where the decresing filtration $\{\mathcal{W}^m_{\rm tame} (f)\}_{k=0}^\infty$ is
defined as the span of all terms of the type \eqref{weyl:mon} such that
\ben
k_1+\cdots+k_r+\ell_1+\cdots+\ell_s+r+s\geq m.
\een
Equivalently, the elements of $\widehat{\mathcal{W}}_{\rm tame}
(f)$ are arbitrary infinite sums of terms of the type \eqref{weyl:mon}
satisfying the tameness condition \eqref{def:tame}. 
We can prove the following proposition by an argument similar to 
the proof of the fact that tame functions are preserved 
by the upper-triangular Givental group action \cite{G4}. 

\begin{proposition}\label{w:tame}
The tame Weyl algebra $\mathcal{W}_{\rm tame} (f)$ and its completion
$\widehat{\mathcal{W}}_{\rm tame} (f)$ are independent of the choices
of an opposite subspace and a good basis. Moreover, the multiplication
induced from $\mathcal{W}(f)$ is well defined, so both
$\mathcal{W}_{\rm tame} (f)$ and 
$\widehat{\mathcal{W}}_{\rm tame} (f)$ are
associative algebras. 
\end{proposition}

Let $\{\phi_i\}_{i=1}^N\subset \operatorname{Jac}(f)$ be a fixed
basis, let $P\subset \HH(f)$ be an opposite subspace, and let $\omega\in
\mathcal{L}_f-\{0\}$. Given these data, we can uniquely construct a good
basis $\{\omega_i\}_{i=1}^N\subset \HH_+(f)\cap Pz$ such that
\ben
\omega_i\equiv \phi_i \omega \mod z\HH_+(f)
\een
and so that there is an isomorphism 
\ben
\Phi_{\omega_1,\dots,\omega_N}\colon 
\C_\hbar[q_0,q_1,\dots]\to \mathbb{V}(f)
\een
defined by
\ben
\Phi_{\omega_1,\dots,\omega_N}(q_{k_1,i_1}\cdots q_{k_1,i_1}):=
(\omega^{i_1}z^{-k_1-1})\otimes\cdots\otimes (\omega^{i_n}z^{-k_n-1}),
\een
where $\{\omega^i\}_{i=1}^N$ is a basis of $\HH_+(f)\cap
zP$ dual to  $\{\omega_i\}_{i=1}^N$ with respect to the
residue pairing $K_f^{(0)}$.
Let us define
\ben
\sigma_{\omega,P}^{\phi_1,\dots,\phi_N} := e^{-\omega
  z/\hbar}\Phi_{\omega_1,\dots,\omega_N} \colon 
\C_\hbar[q_0,q_1,\dots]\to \mathbb{V}(f),
\een
where we use that $\mathbb{V}(f)$ is a $\mathcal{W}(f)$-module on
which $\omega z$ acts locally nilpotently. Note that 
\ben
[-\omega z , \omega^i (-z)^{-k-1}] = \hbar \Res_{z=0}
K_f(-\omega z,\omega^i(-z)^{-k-1}) dz = -\hbar \delta_{k,1} a_i,
\een
where $a_i$ are the coordinates of $\mathbf{1}$, i.e.,
$\mathbf{1}=\sum_{i=1}^N a_i \phi_i$. 
Therefore the operator $e^{-\omega z/\hbar}$ acts as the shift 
$q_1 \mapsto q_1- \mathbf{1}$, and we have 
\[
\sigma_{\omega,P}^{\phi_1,\dots,\phi_N} (g(q_0,q_1,q_2,\dots)) 
= \Phi_{\omega_1,\dots,\omega_N}(g(q_0,q_1-\mathbf{1},q_2,\dots)) 
\]
for any $g\in \C_\hbar[q_0,q_1,\dots]$. 
% The isomorphism $\Phi_{\omega_1,\dots,\omega_N}$ maps $q_{1,i}+a_i$ 
% to $\omega^i(-z)^{-2}$.  
It follows that 
$\sigma_{\omega,P}^{\phi_1,\dots,\phi_N}$ induces an isomorphism
between the completed tame Fock spaces  
\beq\label{Fock-triv}
\sigma_{\omega,P}^{\phi_1,\dots,\phi_N} \colon \C_\hbar
[\![q_0,q_1+\mathbf{1},q_2\dots]\!]_{\rm tame} 
\to \widehat{\mathbb{V}}_{\rm tame}(f), 
\eeq
where $\widehat{\mathbb{V}}_{\rm
  tame}(f):=\widehat{\mathcal{W}}_{\rm
  tame}/\widehat{\mathcal{W}}_{\rm tame}\HH_+(f) $.

\subsubsection{The total ancestor potential and the abstract Fock space}\label{ancestor-F}
It turns out that the dependence of the
isomorphism $\Phi_{\omega_1,\dots,\omega_N}$ on the choice of an opposite subspace and a good
basis is controlled by Givental's symplectic loop group
quantization. Let us assume that we have two opposite filtrations $P'$ and
$P$ and $\{\omega_i'\}_{i=1}^N\subset \HH_+(f)\cap P'z$ and
$\{\omega_i\}_{i=1}^N\subset \HH_+(f)\cap Pz$ are
corresponding good bases. 
\begin{lemma}\label{change-B}
If $P'=P$ and the transition between the good bases is given by
\eqref{fr-change}, then  
\ben
\Phi_{\omega_1',\dots,\omega_N'}^{-1}\circ 
\Phi_{\omega_1,\dots,\omega_N}\mathcal{F}(\hbar;\mathbf{q}) = 
\mathcal{F}(\hbar;B\mathbf{q}). 
\een
\end{lemma}

\noindent
It remains only to investigate the case when $P'$ and
$P''$ are different. Let us choose $\omega\in
\mathcal{L}_f-\{0\}$. Using Lemma \ref{change-B} we may reduce the
general case to the case when $\omega_i\equiv\omega_i'\equiv
\phi_i\omega \mod z \HH_+(f)$.
In order to compare with Givental's formalism put
$H:=\operatorname{Jac}(f)$, and denote by $(\ ,\ )$ the pairing
on $H$ induced by the residue pairing $K_f^{(0)}$. Let $R(f,z)$ be the
symplectic transformation of $H(\!(z)\!)$ defined by \eqref{gb-change}. 
\begin{lemma}\label{change-R}
The following formula holds
\ben
\Phi_{\omega_1',\dots,\omega_N'}^{-1}\circ 
\Phi_{\omega_1,\dots,\omega_N} (\cF) =
(R(f,z)^t)^\wedge\cF,
\een
where $R(f,z)^t$ is the transopse of $R(f,z)$ with respect to the residue pairing.
\end{lemma}
\proof
It enough to prove that if the Lemma is true for some $\cF$, then it is
true for $q_{k,i}\cF$ for all $k\geq 0$, $1\leq i\leq N$.
Recalling Proposition \ref{R-action} and that $R(f,z)^t =
R(f,-z)^{-1}$, we get
\ben
(R(f,z)^t)^\wedge\cF(\mathbf{q}) = 
\left.
\Big( e^{\frac{\hbar}{2} V(\partial_\mathbf{q},\partial_\mathbf{q})}
\, \cF\Big)\right|_{\mathbf{q}\mapsto R(f,-z)\mathbf{q}},
\een
where the 2nd order differential operator 
\ben
V(\partial_\mathbf{q},\partial_\mathbf{q}) = \sum_{k,\ell=0}^\infty
(V_{k\ell} \phi^j,\phi^i)\frac{\partial^2}{\partial q_{k,i}\partial q_{\ell,j}}
\een
is given by 
\ben
V_{k\ell} = \sum_{a=0}^\ell (-1)^{a+k+\ell+1} R_{k+1+a} R_{\ell-a}^t.
\een
By definition,
\ben
\Phi_{\omega_1,\dots,\omega_N} (q_{k,i}\cF) = \omega^i(-z)^{-k-1} \Phi_{\omega_1,\dots,\omega_N} (\cF) 
\een
and $\omega^i(-z)^{-k-1} $ is given by
\ben
\sum_{j=1}^N \Big( \sum_{a=0}^k R_{ij;a}(f)
(-1)^a \omega'^{j} (-z)^{-(k-a)-1} + \sum_{a=k+1}^\infty  R_{ij;a}(f)
(-1)^a \omega'^{j} (-z)^{a-k-1}\Big).
\een
By definiton, if $k-a\geq 0$, then
\ben
\Phi^{-1}_{\omega_1',\dots,\omega_N'} \circ \omega'^{j}
(-z)^{-(k-a)-1} = q_{k-a,j} \circ \Phi^{-1}_{\omega_1',\dots,\omega_N'} 
\een
and if $a\geq k+1$, then 
\ben
\Phi^{-1}_{\omega_1',\dots,\omega_N'} \circ \omega'^{j} (-z)^{a-k-1} =
(-1)^{a+k+1} \hbar \sum_{j'=1}^N (\phi^j,\phi^{j'})
\frac{\partial}{\partial q_{a-k-1,j'} }\circ 
\Phi^{-1}_{\omega_1',\dots,\omega_N'} .
\een
Therefore $\Phi^{-1}_{\omega_1',\dots,\omega_N'}
\Phi_{\omega_1,\dots,\omega_N} (q_{k,i}\cF)$ can be written as the sum
of 
\ben
\sum_{j=1}^N \sum_{a=0}^k R_{ij;a}(f)
(-1)^a q_{k-a,j} \left.
\Big( e^{\frac{\hbar}{2} V(\partial_\mathbf{q},\partial_\mathbf{q})}
\, \cF\Big)\right|_{\mathbf{q}\mapsto R(f,-z)\mathbf{q}} 
\een
and
\ben
\sum_{j=1}^N \sum_{a=k+1}^\infty  R_{ij;a}(f)
(-1)^{k+1}\hbar
\sum_{j'=1}^N (\phi^j,\phi^{j'})
\frac{\partial}{\partial q_{a-k-1,j'} }\
\left.
\Big( e^{\frac{\hbar}{2} V(\partial_\mathbf{q},\partial_\mathbf{q})}
\, \cF\Big)\right|_{\mathbf{q}\mapsto R(f,-z)\mathbf{q}} .
\een
Note that 
\ben
\sum_{j=1}^N \sum_{a=0}^k R_{ij;a}(f)
(-1)^a q_{k-a,j} = (R(f,-z)\mathbf{q})_{k,i}.
\een
A straightforward computation shows that 
\ben
\sum_{j=1}^N \sum_{a=k+1}^\infty  R_{ij;a}(f)
(-1)^{k+1}\hbar
\sum_{j'=1}^N (\phi^j,\phi^{j'})
\frac{\partial}{\partial q_{a-k-1,j'} }\
\left.
\Big(\mathcal{G}(\mathbf{q})\right|_{\mathbf{q}\mapsto
R(f,-z)\mathbf{q}} \Big) 
\een
equals
\ben
\left.
\Big(
\frac{\hbar}{2}[V(\partial_\mathbf{q},\partial_\mathbf{q}),q_{k,i}]
\mathcal{G}\Big)\right|_{\mathbf{q}\mapsto R(f,-z)\mathbf{q}}. 
\een
Finally, for $\Phi^{-1}_{\omega_1',\dots,\omega_N'}
\Phi_{\omega_1,\dots,\omega_N} (q_{k,i}\cF)$ we get
\ben
\left.
\Big( \Big(q_{k,i}+
\frac{\hbar}{2}[V(\partial_\mathbf{q},\partial_\mathbf{q}),q_{k,i}]\Big)
e^{\frac{\hbar}{2} V(\partial_\mathbf{q},\partial_\mathbf{q})} 
\, \cF\Big)\right|_{\mathbf{q}\mapsto R(f,-z)\mathbf{q}} = 
\left.
\Big(
e^{\frac{\hbar}{2} V(\partial_\mathbf{q},\partial_\mathbf{q})}
\, (q_{k,i}\cF)\Big)\right|_{\mathbf{q}\mapsto R(f,-z)\mathbf{q}}.
\een
The above expression is precisely $(R(f,z)^t)^\wedge (q_{k,i}\cF)$.
\qed

\medskip
Comparing the transformation laws for the total ancestor potential
(see Proposition \ref{anc-change}) and the transformation laws in
Lemma \ref{change-B} and Lemma \ref{change-R}, we get that  
\ben
\cA(\hbar,f,\omega):=\sigma^{\phi_1,\dots,\phi_N}_{\omega,P}
(\cA_{f,\omega}^{\omega_1,\dots,\omega_N}(\hbar;\mathbf{q}))
\een 
is a vector in $\widehat{\mathbb{V}}_{\rm tame}(f)$ independent of the
choice of the basis $\{\phi_i\}_{i=1}^N$ and the choice of the
opposite subspace $P$. We refer to $\cA(\hbar,f,\omega)$ as 
the {\em global ancestor potential} of $f$.
Let us denote by $\widehat{\mathbb{V}}_{\rm tame}$ the
vector bundle on $\cM_{\rm mar}^\circ$ 
whose fiber over a point $f\in \cM_{\rm mar}$ is the completed tame Fock
space $\widehat{\mathbb{V}}_{\rm mar}(f).$ We call it the {\em
  completed tame Fock  bundle} or simply the {\em abstract Fock bundle}.
The global ancestor potential may be viewed as a holomorphic function 
\ben
\cA\colon \mathcal{L}\setminus{\{0\}}\to \widehat{\mathbb{V}}_{\rm tame},\quad 
(f,\omega)\mapsto \cA(\hbar,f,\omega).
\een
Note that the above map is not a map of vector bundles. Nevertheless, we
have the following symmetry, which in some sense allows us to think of
$\hbar$ as a coordinate along the fiber of $\mathcal{L}$.
\begin{corollary}\label{scale-c}
The global ancestor potential has the following scaling property:
\ben
\cA(\hbar,f,c\,\omega) = \cA(\hbar c^{-2},f, \omega)
,\quad \forall c\in \C\setminus{\{0\}}.
\een
\end{corollary}
\proof
The statement is a Corollary of Proposition \ref{anc-change}, a) and
Lemma \ref{change-B}.
\qed

\subsection{Abstract modular forms}
Motivated by Corollary \ref{scale-c} and the generalized definition of
a quasi-modular form in the theory of the period maps (see Definition
\ref{def:qmf}), we would like to introduce the notion of a
quasi-modular form for the moduli space $\cM^\circ_{\rm mar}$.

\begin{definition}\label{def:amf}
We say that a function 
\ben
\cA:\cL-\{0\} \to \widehat{\mathbb{V}}_{\rm tame}, 
\een 
is an {\em abstract modular form} if $\cA(\hbar,f,c\,\omega) = 
\cA(\hbar c^{-2},f,\omega)$.
\end{definition}
\begin{remark}
According to Corollary \ref{scale-c}, the global ancestor potential is
an abstract modular form. 
\end{remark}

In order to compare Definitions \ref{def:amf} and \ref{def:qmf}, let
us trivialize the abstract Fock bundle over an open subset
$\mathcal{U}\subset \cM_{\rm mar}^\circ$. Let us denote by $\Jac$ 
the vector bundle over $\cM_{\rm mar}^\circ$ whose fiber over $f$ is the
Jacobi algebra of $f$. Assume that $\{\phi_i\}_{i=1}^N$ is a frame for $\Jac|_{\mathcal{U}}$,
$\omega$ is a frame for $\mathcal{L}|_{\mathcal{U}}$, and 
$P \subset \HH|_{\mathcal{U}}$ is a sub-bundle such that
$P(f)\subset \HH(f)$ is an opposite subspace for
all $f\in \mathcal{U}$. The isomorphism
\eqref{Fock-triv} is a trivialization of the abstract Fock bundle. 

Let $\cA(\hbar,f,\omega)$ be an abstract modular form. Put 
\ben
\cA(\hbar; f,\omega,\mathbf{q}):=
(\sigma^{\phi_1,\dots,\phi_N}_{\omega,P(f)})^{-1}
(\cA(\hbar,f,\omega)).
\een
Suppose that we have the following genus expansion
\ben
\cA(\hbar;f,\omega,\mathbf{q})=
\exp\left( 
\sum_{g=0}^\infty \hbar^g \cF_g(f,\omega,\mathbf{q})\right).
\een
The scaling property of $\cA$ is equivalent to 
\ben
\cF_g(f,c\,\omega,c\, \mathbf{q}) = 
c^{2-2g} \cF_g(f,\omega,\mathbf{q}),
\een
so the coefficients  of $\cF_g(f,\omega,\mathbf{q})$ in front of the
$\mathbf{q}$-monomials can be interpreted as sections of $\cL^{2g-2+\operatorname{deg}}|_{\cU}$, where
$\operatorname{deg}$ is the degree of the corresponding $\mathbf{q}$-monomial.

The notion of quasi-modularity is contained in Definition
\ref{def:amf} as follows. 
Given an abstract modular form $\cA$ and a
chart $\mathcal{U}$ such that $\Jac|_{\mathcal{U}}$ is trivial, 
we can choose the complex conjugate opposite
sub-bundle $\kappa(\HH_+)z^{-1}$ to obtain coordinate expressions
$\widetilde{\cF}_g(f,\omega,\mathbf{q})$. The latter depends
non-holomorphically on $f$, but its coefficients transform as modular forms in the
sense of Definition \ref{def:qmf}. Choosing a different opposite
sub-bundle $P \subset \HH$ with a holomorphic trivialization
of $\HH_+\cap P z$, we can obtain another coordinate expression
$\cF_g(f,\omega,\mathbf{q})$ for $\cA$. The change of the opposite
subspace is given by a matrix $B(f)$ and a symplectic
transformation $R(f,z)=1+R_1(f)z+R_2(f)z^2+\cdots$ 
(see Section \ref{ancestor-F}). 
Recalling Lemma \ref{change-B} and Lemma \ref{change-R}, 
we get that the coefficients of $\widetilde{\cF}_g(f,\omega,\mathbf{q})$ 
are polynomial expressions of the coefficients of $\cF_l(f,\omega,\mathbf{q})$, 
$l\le g$,  
with coefficients in $\C[B,R_1,R_2,\dots]$, where 
$\mathbb{C}[B,R_1,R_2,\dots]$ is the polynomial ring on the 
entries of the matrices $B(f),R_1(f),R_2(f),\dots$.  
The entries of these matrices depend non-holomorphically on $f$, so following the
terminology in Definition \ref{def:qmf}, we call them {\em anti-holomorphic
generators}. Let us point out that finding explicit formulas for the
anti-holomorphic generators is in general a difficult problem (see the
example in Section \ref{example:e6}).  

\begin{remark} 
We explain a relationship of the Fock bundle in the present paper 
to the Fock sheaf in \cite{Coates-Iritani:Fock}. 
Given an opposite subspace $P$, 
a section of the Fock sheaf in \cite{Coates-Iritani:Fock} can be locally 
identified with a function on the Givental Lagrangian cone 
of the form $Z = \exp(\sum_{g=0}^\infty \hbar^{g-1} F_g)$ 
 (this is similar to the trivialization \eqref{Fock-triv}). 
The total space of $\cL-\{0\}$ over $\cM_{\rm mar}^\circ$ can be 
identified with a finite-dimensional slice of the Givental cone, and a coordinate 
expression $\cA(\hbar;f,\omega,\mathbf{q}) 
= (\sigma^{\phi_1,\dots,\phi_N}_{\omega,P})^{-1} \cA(\hbar,f,\omega)$ 
of the abstract ancestor potential corresponds to the jet of the potential 
$Z$ at the point $(f,-z \omega)$. This is related to the 
\emph{jetness} in \cite{Coates-Iritani:Fock}. 
\end{remark}

\subsection{The holomorphic anomaly equations}
Let us pick local holomorphic frames $\{\phi_i\}_{i=1}^N$ and $\omega $ for
respectively ${\rm Jac}$ and $\mathcal{L}$.
Let us choose a local holomorphic coordinate system
$\sigma=(\sigma_1,\dots,\sigma_{N'})$ on $\cM_{\rm mar}^\circ$, so that each
$\phi_i=\phi_i(x,\sigma )$ is a weighted-homogeneous polynomial
depending holomorphically on $\sigma$. We define the {\em hybrid
  ancestor potential} $\cA_f(\hbar,\mathbf{q})$ to be the
coordinate expression
$\cA_{f,\omega}^{\omega_1,\dots,\omega_N}(\hbar,\mathbf{q})$ of the
global ancestor potential with respect to the opposite subspace
$\kappa(\HH_+)z^{-1}$ (see Section \ref{ancestor-F}). The
hybrid ancestor potential depends non-holomorphically on $f$. We would
like to derive differential equations for $\cA_f(\hbar,\omega)$, which following the physics
literature will be called {\em holomorphic anomaly equations}, that govern the
non-holomorphic dependence on $f$.

Let us begin with several remarks about the Cecotti-Vafa connection from Section
\ref{CV}. For given $f\in \cM_{\rm mar}^\circ$, we have a vector spaces isomorphism 
\ben
\Jac(f)\cong \mathbb{K}(f),\quad \phi_a\mapsto \omega_a,
\een
which allows us to interpret all connection matrices and the complex
conjugation $\kappa$ as endomorphisms and a complex conjugation of
$\Jac(f)$. Note that $C_i$ is the operator of
multiplication by $\partial f/\partial \sigma_i$ and $U$ is the operator of multiplication
by $f$. In particular, $\widetilde{U}=U=0$ because $f$ is weighted-homogeneous, so
it vanishes in $\Jac(f)$. 
The next observation is that $\{\omega_a\}_{a=1}^N$ is a
holomorphic frame for $\mathbb{K}$, because the transition  functions
in this frame are the same as the transition functions of $\Jac\otimes
\mathcal{L}$ in the frame $\{\phi_a\otimes \omega\}$ and the latter is
by definition holomorphic. This observation implies that the
connection matrices satisfy
$\Gamma_{\bar\iota}=0$ and $\Gamma_i=h^{-1}\partial h$, where
$\partial$ is the holomorphic de Rham differential on $\cM_{\rm mar}^\circ$
and  $h=(h_{ab})$ is the matrix of the Hermitian pairing
\ben
h_{ab} = K^{(0)}(\kappa(\omega_a),\omega_b).
\een
Finally, let us point out that the operators
$z\widetilde{C}_{\bar\iota},$ $1\leq 
i\leq N'$ are infinitesimal symplectic transformations of
$\Jac(f)(\!(z^{-1})\!)$. This can be proved by using the
compatibility of the Gauss--Manin connection with the higher residue
pairings and the fact that up to a holomorphic factor, the quantity
\ben
K(\omega_i,\omega_j) = K^{(0)}(\omega_i,\omega_j)
\een
 is the Grothendick residue of
$\phi_i(x,\sigma)\phi_j(x,\sigma)$, so it must be holomorphic in $\sigma\in
\cM_{\rm mar}^\circ$. 
\begin{proposition}\label{han}
The hybrid ancestor potential satisfies the following differential equations
\ben
\partial_{\overline{\sigma}_i} \cA_f (\hbar,\mathbf{q}) = 
(z\widetilde{C}_{\bar\iota}^t)^\wedge 
\cA_f (\hbar,\mathbf{q}) ,\quad 1\leq i\leq N',
\een 
where ${}^t$ is conjugation with respect to the residue pairing.
\end{proposition}
\proof
Let us fix $f\in \cM_{\rm mar}^\circ$ and denote by $P(F)$ ($F\in B_f$) the
holomorphic extension of the opposite subspace
$\kappa(\HH_+(f))$ to a family of opposite subspaces on the
parameter space $B_f$ of a miniversal unfolding of $f$.
Let us fix arbitrary holomorphic frames $\{\phi_a\}_{a=1}^N$ and $\omega$ of
$\Jac$ and $\mathcal{L}$ on $\cM_{\rm mar}^\circ\cap B_f$. For every $f'\in \cM_{\rm mar}^\circ\cap
B_f$ we have two opposite subspaces in $\HH(f')$:  the complex conjugate subspace
$\kappa(\HH_+(f'))z^{-1}$ and the holomorphic
opposite subspace $P(f')$. Let us denote by
$\{\omega_a\}_{a=1}^N\subset \mathbb{K}(f')$
and $\{\widetilde{\omega}_a\}_{a=1}^N\subset \HH_+(f')\cap
P(f')z$ the good bases such that
\ben
\omega_a \equiv  \widetilde{\omega}_a\equiv \phi_a\omega 
\mod z\HH_+(f'). 
\een
According to Proposition \ref{anc-change}, b),  the ancestor potentials 
\ben
\cA_{f'}(\hbar,\mathbf{q}):=\cA_{f',\omega}^{\omega_1,\dots,\omega_N}(\hbar;\mathbf{q})\quad
\mbox{and}\quad  
\widetilde{\cA}_{f'}(\hbar,\mathbf{q}):=
\cA_{f',\omega}^{\widetilde{\omega}_1,\dots,\widetilde{\omega}_N}(\hbar;\mathbf{q})  
\een
are related by 
\ben
\cA_{f'}(\hbar;\mathbf{q})=(R(f',z)^t)^\wedge \, 
\widetilde{\cA}_{f'}(\hbar;\mathbf{q}),
\een
where the symplectic transformation $R(f',z)$ of
$\Jac(f')(\!(z^{-1})\!)$ is represented in the basis
$\{\phi_a\}_{a=1}^N$ of $\Jac(f')$ by the matrix
$(R_{ab}(f',z))_{a,b=1}^N$  that describes the change 
\beq\label{R-def}
\omega_b(f') = \sum_{a=1}^N
\widetilde{\omega}_a(f')R_{ab}(f',\sigma),\quad 1\leq b\leq N.
\eeq
Since $\widetilde{\cA}_{f'}$ depends holomorphically on $f'$,
we just need to find the derivatives of $R(f',z)$ with respect to
$\overline{\sigma}'_i$, where $\sigma'=(\sigma_1',\dots,\sigma_{N'}')$  
denotes the coordinates of $f'$. Differentiating  \eqref{R-def} with
respect to $\overline{\sigma}'_i$ we get 
\ben
\sum_{k=1}^N \omega_k(f') (\Gamma_{\bar\iota b}^k +
\widetilde{C}_{\bar\iota b}^k z) = 
\sum_{a,k=1}^N \widetilde{\omega}_a(f')R_{ak}(f',\sigma)
 (\Gamma_{\bar\iota b}^k +
\widetilde{C}_{\bar\iota b}^k z) 
\een
for the LHS and
\ben
\sum_{a=1}^N \widetilde{\omega}_a(f')\partial_{\overline{\sigma}_i} R_{ab}(f',z)
\een
for the RHS. Comparing the coefficients in front of
$\widetilde{\omega}_a$ we get
\ben
\partial_{\overline{\sigma}_i} R_{ab}(f',z) = 
\sum_{k=1}^N 
R_{ak}(f',\sigma)
 (\Gamma_{\bar\iota b}^k +\widetilde{C}_{\bar\iota b}^k z) .
\een
In matrix form the above equation becomes
$\partial_{\overline{\sigma}_i} R(f',z) =  R(f',z)
(\Gamma_{\bar\iota}+ z \widetilde{C}_{\bar\iota}).$ Although quantization of
symplectic transformations is only a projective representation, when
restricted to the subgroup of symplectic transformations of the type
$R_0+R_1z+\cdots$ the quantization becomes a representation. Therefore  
\ben
\partial_{\overline{\sigma}_i} (R(f',z)^t)^\wedge =  
(\Gamma^t_{\bar\iota}+ z \widetilde{C}^t_{\bar\iota})^\wedge
(R(f',z)^t)^\wedge. 
\een
It remains only to use that the frame $\{\omega_a\}_{a=1}^N$ is holomorphic, so
$\Gamma_{\bar \iota}=0.$\qed

\begin{remark}
Note that by definition, the quantization of the infinitesimal symplectic
transformation $z \widetilde{C}_{\bar\iota}^t$ is
the following differential operator
\ben
\sum_{a,b=1}^N \Big( 
\frac{\hbar}{2} \widetilde{C}_{\bar\iota}^{ab}\frac{\partial^2}{\partial
  q_{0,a}\partial q_{0,b}} -
\sum_{k=0}^\infty 
\widetilde{C}_{\bar\iota a}^b q_{k,a}\frac{\partial}{\partial
  q_{k+1,b}}  \Big),
\een
where 
\ben
\widetilde{C}_{\bar\iota}^{ab}:=(\widetilde{C}_{\bar\iota}\phi^a,\phi^b)=
\overline{(C_i \kappa(\phi^a),\kappa(\phi^b))}
\een
is symmetric in $a$ and $b$. Here $\{\phi^a\}_{a=1}^N$ is a basis of
$\Jac(f)\cong \mathbb{K}(f)$ dual to $\{\phi_a\}_{a=1}^N$ with respect to
the residue pairing. From this explicit formula we see that our
differential equations have the same form as the BCOV
holomorphic anomaly equations (see \cite{BCOV}, formula (3.17)). 
\end{remark}
\begin{remark} 
The holomorphic anomaly equation was also studied in 
\cite[\S 9.3]{Coates-Iritani:Fock}.  
\end{remark} 
\subsection{Fermat simple elliptic singularity of type $E_6^{(1,1)}$ }\label{example:e6}

We would like to give an example in which our construction gives an
elegant way to investigate the modular properties of the total
ancestor potential. 
Put
\ben
f(x,Q) = x_1^3+x_2^3+x_3^3-\frac{1}{Q} x_1x_2x_3,\quad Q(1-3^3Q^3)\neq 0.
\een
This family of polynomials represents a transversal slice to the orbits
of the group of coordinate changes, so it could be viewed as an open
chart in the marginal moduli space (see Remark \ref{group:cc}). 
Let us introduce the functions 
\ben
f_0(Q) & = &  1+\sum_{k=1}^\infty \frac{(3k)!}{(k!)^3} Q^{3k},\\
f_1(Q) & = & \log Q + \sum_{k=1}^\infty \frac{(3k)!}{(k!)^3}
Q^{3k}(\log Q + h_{3k}-h_k), \\
g(Q,\overline{Q})
& = &
-\frac{ (Df_1)\overline{f}_0+(Df_0)\overline{f}_1
}{
f_1 \overline{f}_0+ f_0\overline{f}_1
},
\een
where $D=Q\partial_Q Q = (Q+Q^2\partial_Q)$ and
$h_\ell=1+\frac{1}{2}+\cdots + \frac{1}{\ell}$ are the harmonic 
numbers. Note that $f_0$ and $f_1$ are solutions to a 2nd order
Fuchsian equation defined by the differential operator
\beq\label{hge}
(Q\partial_Q)^2 - 3^3 Q^3 (Q\partial_Q+1)(Q\partial_Q+2).
\eeq
We are going to choose two sets of good bases $\{\omega_e^{\rm KS}\}$
and $\{\omega_e^{\rm GW}\}$, where the index set for $e$ is
splitted into the following 4 groups:
\ben
\{(0,0,0),\, (1,1,1)\}\sqcup \{ (1,0,0),\, (2,0,0)\}\sqcup \{ (0,1,0),\,
(0,2,0)\}\sqcup \{ (0,0,1),\, (0,0,2)\}  .
\een
The following set of forms is a good basis for the complex conjugate
opposite subspace:
\ben
& \omega^{\rm KS}_{0,0,0} & = dx/Q,\\
& \omega^{\rm KS}_{1,1,1} & = (x_1x_2x_3 + g(Q,\overline{Q}) z) dx/Q \\
& \omega^{\rm KS}_{m\, e_i} & = x_i^m dx/Q,\quad 1\leq m\leq 2,\quad
1\leq i\leq 3,
\een
where $e_i$ is the $i$th coordinate vector in $\mathbb{Z}^3$. 
Another good basis is computed from mirror symmetry at the large
radius limit point $Q=0$
\ben
& \omega_{0,0,0}^{\rm GW} & = \sqrt{-1}\, \frac{dx}{Qf_0(Q)}, \\ 
& \omega_{1,1,1}^{\rm GW} & = \sqrt{-1}\, \Big(x_1x_2x_3 - z
\frac{Df_0}{f_0} \Big)\,
f_0(Q)\operatorname{det}(I_0)^{-1}\, \frac{dx}{Q},  \\
& \omega_{m e_i}^{\rm GW} & = \sqrt{-1}\, (1-3^3Q^3)^{m/3} \, x_i^m
\frac{dx}{Q} ,\quad m=1,2,
\een 
where 
\ben
I_0 = \begin{bmatrix}
f_0(Q) & Df_0(Q) \\
f_1(Q) & Df_1(Q)
\end{bmatrix}
\een
is the Wronskian matrix of the differential equation \eqref{hge}.
The details of both computations will be presented
in a future investigation. 

Let us pick $\omega = \omega_{0,0,0}^{\rm GW}$ to be a frame for the
vacuum line bundle and denote by $\cA_f^{\rm KS}(\hbar;\mathbf{q})$
and $\cA_f^{\rm GW}(\hbar;\mathbf{q})$ the total ancestor potentials
corresponding respectively to the good bases $\{\omega_i^{\rm KS}\}$ and
$\{\omega_i^{\rm GW}\}$. Let us define 
\ben
t:=2\pi\sqrt{-1}\tau/3:=f_1(Q)/f_0(Q).
\een
Using the positive definite Hermitian pairing on $\mathbb{H}_+(f)\cap
\kappa(\HH_+(f))$, it is easy to check that $t+\overline{t}<0$, i.e.,
$\operatorname{Im}(\tau)>0$. Note that
\ben
(\omega_{0,0,0}^{\rm KS},\dots, \omega_{0,0,2}^{\rm KS}) =
(\omega_{0,0,0}^{\rm GW},\dots, \omega_{0,0,2}^{\rm GW})\,
R(Q,\overline{Q},z)\, B(Q),
\een
where the matrices $R$ and $B$ are block-diagonal 
\begin{align*}
R & =\operatorname{Diag}(R^{(1)},R^{(2)},R^{(3)},R^{(4)}),\\
B & =\operatorname{Diag}(B^{(1)},B^{(2)},B^{(3)}, B^{(4)}),
\end{align*}
where the blocks $R^{(i)}$ and $B^{(i)}$ are given by the following
formulas: 
\ben
R^{(1)}=
\begin{bmatrix}
1 & z/(t+\overline{t}) \\
0 & 1
\end{bmatrix}, \quad
R^{(2)} = R^{(3)}=R^{(4)}=
\begin{bmatrix}
1 & 0 \\
0 & 1
\end{bmatrix}
\een
and 
\begin{align*}
B^{(1)} = &  
-\sqrt{-1} \,
\begin{bmatrix}
f_0 &0 \\
0 & \operatorname{det}(I_0)/f_0 
\end{bmatrix},\\
B^{(2)} = B^{(3)}=B^{(4)} = &
-\sqrt{-1} \,
\begin{bmatrix}
(1-3^3Q^3)^{-1/3} & 0 \\
0 & (1-3^3Q^3)^{-2/3} 
\end{bmatrix}.
\end{align*}
Recalling Proposition \ref{anc-change}, we get that 
\beq\label{antiholo-compl}
\widetilde{\cA}_{\tau}(\hbar;\mathbf{q}):=(R(Q,\overline{Q},z)^t)^\wedge
\,
\cA_f^{\rm GW}(\hbar;\mathbf{q}) = \cA_f^{\rm KS}(\hbar;B^{-1}\mathbf{q}).
\eeq
Let $\Sigma=\mathbb{P}^1-\{Q(1-3^3Q^3)=0\}$
be the domain of the 
deformation parameter $Q$. The function $\tau$ gives an
identification between the unniversal covering of $\Sigma$ and the
upper-half plane $\mathbf{H}$. It is easy to check that the monodromy
transformations of $\tau$ are given by 
\ben
\tau\mapsto g(\tau)=\frac{a\tau+b}{c\tau+d},\quad 
g=\begin{bmatrix}
a & b\\
c & d
\end{bmatrix}\in \Gamma(3).
\een 
Under the analytic continuation the primitive form $\omega$ and
$B^{-1}$ are transformed respectively to 
\ben
\omega\mapsto \omega (c\tau+d)^{-1} \quad \mbox{and}\quad 
B^{-1}\mapsto B^{-1} J(g,\tau) (c\tau+d)^{-1},
\een
where 
\ben
J(g,\tau) = \operatorname{Diag}(
1,(c\tau+d)^2,\underbrace{c\tau+d,\dots,c\tau+d}_\text {6 times }).
\een
The analytic continuation of the identity \eqref{antiholo-compl}
yields
\ben
\widetilde{\cA}_{g(\tau)} (\hbar;\mathbf{q}) = 
\widetilde{\cA}_{\tau} (\hbar(c\tau+d)^2;J(g,\tau) \mathbf{q}) .
\een
Comparing the coefficients in front of the monomials in $\mathbf{q}$
we get that the coefficients of the total ancestor potential
$\widetilde{\cA}_{\tau} (\hbar;\mathbf{q})  $ transform as modular
forms on $\Gamma(3)$. 
\begin{remark}
The potential $\widetilde{\cA}_{\tau}
(\hbar;\mathbf{q})  $ coincides with the anti-holomorphic completion
constructed in an ad hoc way in \cite{MR}. Recalling also the mirror
symmetry established in Theorem \ref{CY-LG:mirror_sym}, we recover
the main result of \cite{MR}: the Gromov--Witten invariants of the
elliptic orbifold $\mathbb{P}^1_{3,3,3}$ are quasi-modular forms. 
\end{remark}
\begin{remark}
Slightly generalizing the above argument, we would like to investigate the
total ancestor potential $\cA_f^{\rm GW}(\hbar;\mathbf{q})$ as a
formal series in $q_{k,i}$, $k>0$, whose coefficients are analytic
functions of $\tau_i:=q_{0,i}$, $1\leq i\leq N$. Recalling the results
of Looijenga (see \cite{Lo}), we may identify each relevant deformation parameter
$\tau_i$ with an appropriate $\theta$-function. We expect that under
this identification, the GW invariants will turn into quasi-Jacobi
forms.  We will address this problem in a future investigation. 
\end{remark}
   
\section{Mirror symmetry for orbifold Fermat CY hypersurfaces} 

In the last three sections, we have constructed a global B-model
generating function which is modular in an appropriate generalized
sense, but non-holomorphic. In the remaining part of this paper we
will prove two mirror theorems for Fermat type polynomials satisfying
a CY condition. Namely, we will prove that the generating
functions of certain GW-theory/FJRW-theory invariants are holomorphic
limits of the global B-model generating function. In this section, we
will establish the mirror theorem for GW-theory. 

Let $W:=x_1^{d_1}+\cdots + x_n^{d_n}$ where $d_1,\dots,d_n\in \mathbb{Z}$ are positive integers satisfying 
\ben
\frac{1}{d_1}+\cdots +\frac{1}{d_n} = 1.
\een
Let $D:=\operatorname{lcm}(d_1,\dots,d_n)$ and $w_i:=D/d_i.$ Let $G$ be the group
\ben
G=\{t\in (\mathbb{C}^*)^n \ |\ t_1^{d_1}=\cdots =t_n^{d_n} \}.
\een
We define two orbifolds 
\ben
\mathbf{P} := [(\mathbb{C}^n\setminus{\{0\}})/G]=\mathbb{P}(w_1, \cdots, w_n)
\een
and a suborbifold Calabi--Yau (CY) hypersurface
\ben
Y=[Z/G],\quad Z=\{x_1^{d_1}+\cdots + x_n^{d_n}=0\}\subset \mathbb{C}^n\setminus{\{0\}}.
\een
The above quotients are taken in the category of orbifold groupoids or
equivalently in the category of smooth Deligne--Mumford stacks. 
We remark that here $$Y=[X_W/\tilde{G}_W].$$
 
\subsection{Orbifold Gromov--Witten theory}
Let $Y$ be an orbifold groupoid whose coarse moduli space $|Y|$ is a projective 
variety. Let us denote by $H:=H_{\rm CR}(Y,\mathbb{C})$ the Chen--Ruan cohomology
of $Y$. Then
$$\dim H=N:=(d_1-1)\cdots (d_n-1).$$
Our main interest is in the orbifold Gromov--Witten (GW)
invariants of $Y$
\beq\label{GW-inv}
\lan \phi_{i_1}\psi^{\ell_1},\dots,\phi_{i_k}\psi^{\ell_k}\ran_{g,k,d},
\eeq
where $\{\phi_i\}_{i=1}^N$ is a basis of $H$ and $d\in
\operatorname{Eff}(Y)\subset H_2(|Y|;\mathbb{Z})$ is an effective curve
class. The invariants  are defined through the intersection theory on
the moduli space $\overline{\cM}_{g,k}(Y,d)$. The latter is the moduli
space of degree-$d$ stable
orbifold maps 
$$f\colon (C,(z_1,g_1),\dots,(z_k,g_k))\to Y,$$ where $C$ is a genus-$g$
nodal curve equipped with an orbifold structure, $n$
marked points $z_1, \cdots, z_k$, and a choice of a generator $g_i\in
\Aut_C(z_i)$ for each $i$.  The evaluation at the $i$-th marked point
$(f(z_i),f(g_i))$ determines an evaluation map $\operatorname{ev}_i\colon 
\overline{\cM}_{g,k}(Y,d)\to IY$, where $IY$ is the inertia orbifold of
$Y$. Let us denote by $\psi_i=c_1(L_i)$, $1\leq i\leq k$, the descendant
classes, where $L_i$ is the line bundle on $\overline{\cM}_{g,k}(Y,d)$
formed  by the cotangent lines of the coarse space of $C$ at the
$i$-th marked point (see \cite{Ts}). By definition the GW invariant
\eqref{GW-inv} is obtained by pairing the  cohomology class
\ben
\operatorname{ev}_1^*(\phi_{i_1})\psi_1^{\ell_1}\cup\cdots \cup
\operatorname{ev}_1^*(\phi_{i_k})\psi_k^{\ell_k} 
\een
with the virtual fundamental cycle $[\overline{\cM}_{g,k}(Y,d)]^{\rm
  virt}$. The invariant takes its value in the Novikov ring
$\mathbb{C}[\![Q]\!]$. Let us fix an ample $\mathbb{Z}$-basis
$\{L^{(i)}\}_{i=1}^r$ of $\operatorname{Pic}(|Y|)$.  We embed 
\ben
\mathbb{C}[\![Q]\!]\subset \mathbb{C}[\![Q_1,\dots,Q_r]\!],\quad
Q^d\mapsto Q_1^{\langle c_1(L^{(1)}),d\rangle}\cdots  Q_r^{\langle c_1(L^{(r)}),d\rangle}
\een
For further details on orbifold GW theory we refer to \cite{AGrV} for
the algebraic approach and to \cite{CR} for the analytic approach. 

\subsubsection{Givental's cone}
Let us introduce a set of formal variables $\mathbf{t}=\{t_{k,i}\}$, $1\leq i\leq N$,
$k\geq 0$. The generating function
\ben
\mathcal{F}^{(g)}_Y(\mathbf{t}) = 
\sum_{k=0}^\infty 
\sum_{d\in \operatorname{Eff}(Y)}
\frac{Q^d}{k!}
\lan \mathbf{t}(\psi),\dots,\mathbf{t}(\psi)\ran_{g,k,d}
\een
is called the {\em genus-$g$ total descendant potential}. Here we
write 
$$\mathbf{t}(\psi) = \sum_{k=0}^\infty\sum_{i=1}^N
t_{k,i}\phi_i\psi^k$$ and expand the correlator multiniearly as a
formal power series in $\mathbf{t}$ whose coefficients are the GW
invariants \eqref{GW-inv}. 

Following Givental \cite{G2} we introduce the symplectic vector space
\ben
\mathcal{H}_Y=H_{\rm CR}(Y;\mathbb{C}[\![Q]\!])(\!(z^{-1})\!),\quad
\Omega(f,g) = \Res_{z=0} (f(-z),g(z))dz, 
\een 
where $(\ ,\ )$ is the orbifold Poincar\'e pairing. 
The subspaces $\mathcal{H}_Y^+:=H_{\rm CR}(Y;\mathbb{C}[\![Q]\!])[z]$
and $\mathcal{H}_Y^-:=H_{\rm
  CR}(Y;\mathbb{C}[\![Q]\!])[\![z^{-1}]\!]z^{-1}$ are Lagrangian
subspaces and define a polarization
$\mathcal{H}_Y=\mathcal{H}_Y^+\oplus \mathcal{H}_Y^-$, which allows us
to identify $\mathcal{H}_Y\cong T^*\mathcal{H}_Y^+.$ By definition,
Givental's cone $\mathcal{L}_Y$ is the graph of the differential
$d\mathcal{F}^{(0)}_Y$.  Explicitly,
\ben
\mathcal{L}_Y = \left\{ -z+\mathbf{t} + \sum_{k=0}^\infty \sum_{i=1}^N
\frac{\partial \mathcal{F}^{(0)}_Y}{\partial t_{k,i}} (\mathbf{t}) \,
\phi^{i}(-z)^{-k-1}\ \Big|\ \mathbf{t}(z)\in \mathcal{H}_Y^+
\right\},
\een
where $\{\phi^i\}\subset H$ is a basis dual to $\{\phi_i\}$ with
respect to the Poincar\'e pairing.  The above definition should be
understood in the formal sense, i.e., $\mathcal{L}_Y$ is the formal
germ at $\mathbf{t}=0$ of a cone in $\mathcal{H}_Y$.

\subsubsection{The $J$-function and the calibration}

Let us fix $\tau\in H.$ It is convenient to introduce the notation
\ben
\llangle[\Big]\alpha_1,\dots,\alpha_k\rrangle[\Big]_{g,k}(\tau) = 
\sum_{m=0}^\infty \sum_{d\in \operatorname{Eff}(Y)}
\frac{Q^d}{m!}
\lan\alpha_1,\dots,\alpha_k,\tau,\dots,\tau\ran_{g,k+m,d},
\een
where $\alpha_s=\phi_{i_s}\psi^{\ell_s}$, $1\leq s\leq k$, are arbitrary
insertions.  By definition, Givental's J-function of $Y$ is 
\ben
\widetilde{J}_Y(\tau,Q,z) = 
z+\tau+
\sum_{k=0}^\infty
\sum_{i=1}^N
\sum_{d\in {\rm Eff}(Y)} 
Q^d\, \llangle[\Big]\phi_i\psi^k\rrangle[\Big]_{0,1,d}(\tau)\, \phi^i z^{-k-1}.
\een
Note that 
\ben
\widetilde{J}_Y(\tau,Q,-z) = -z+\tau+d_{-z+\tau}\mathcal{F}^{(0)}_Y
\ \in \ \mathcal{L}_Y.
\een
Recall also the calibration series $S(\tau,Q,z)=1+S_1(\tau,Q)z^{-1}+\cdots$ defined by
\ben
(S(\tau,Q,z)\phi_i,\phi_j) = (\phi_i,\phi_j)+
\sum_{k=0}^\infty 
\sum_{d\in {\rm Eff}(Y)} 
Q^d\,\llangle[\Big]
\phi_i\psi^k,\phi_j\rrangle[\Big]_{0,2,d}(\tau)\, z^{-1-k}.
\een
Let us denote by $\widetilde{L}_\tau$ the tangent space to
$\mathcal{L}_Y$ at $\widetilde{J}(\tau,Q,-z)$, then we have 
\ben
\widetilde{J}(\tau,Q,-z)=-z\, S(\tau,Q,z)^{-1} \, 1,\quad \widetilde{L}_\tau =
S(\tau,Q,z)^{-1} \mathcal{H}_Y^+.
\een
\subsubsection{Quantum cohomology}

The quantum cup product $\bullet_\tau$ is defined by 
\ben
(\phi_a\bullet_\tau\phi_b,\phi_c) = \llangle[\Big]
\phi_a,\phi_b,\phi_c\rrangle[\Big]_{0,3}(\tau). 
\een
Let us fix $\epsilon \ll 1$, assume that the Novikov variables
$|Q_i|\leq \epsilon$ $(1\leq i\leq r)$, and denote by $B\subset H$ the
open subset of $\tau\in H$ for which the 
quantum cup product is convergent. In the absence of convergence, we
think of $B$ as a formal analytic germ at $Q=0$ and $\tau=0$. Let us introduce
also the {\em Euler vector field}
\ben
E = \sum_{i=1}^N (1-\operatorname{deg}_{\rm CR}(\phi_i) )
t_i\partial_{t_i} + c_1(Y),
\een
where $\operatorname{deg}_{\rm CR}$ denotes the Chen--Ruan
degree. Finally, let 
\ben
\theta:H\to H,\quad \theta(\phi_i) = 
\Big( \frac{\operatorname{dim}_\mathbb{C}(Y)}{2}-
\operatorname{deg}_{\rm CR}(\phi_i)\Big) \phi_i
\een
be the so-called {\em Hodge grading operator}.

By definition the quantum (or Dubrovin's) connection is the connection
$\nabla$ on the trivial $H$-bundle with base $B\times\mathbb{C}^*$
defined by 
\ben
\nabla = d + \Big( -z^{-1}\theta+z^{-2}E\bullet \Big) dz - \sum_{i=1}^N
z^{-1} (\phi_i\bullet) dt_i.
\een
It is known that the gauge transformation defined by the calibration
$S(\tau,Q,z)$ acts on $\nabla$ as follows:
\ben
S(\tau,Q,z)^{-1} \nabla S(\tau,Q,z) = d + \Big( -z^{-1}\theta+z^{-2}\rho\Big) dz,
\een
where $\rho:=c_1(Y)\cup$ is the operator of classical cup product multiplication
by $c_1(Y)$. In particular $\nabla$ is a flat connection. 

\subsubsection{The total ancestor potential}

Let $\tau\in H$ be an arbitrary parameter. The ancestor GW invariants
\ben
\llangle[\Big]
\phi_{i_1}\bar{\psi}^{\ell_1},\dots,\phi_{\ell_k}\bar{\psi}^{\ell_k}\rrangle[\Big]_{g,k,d}(\tau):=
\sum_{m=0}^\infty \frac{1}{m!} 
\langle \phi_{i_1}\bar{\psi}^{\ell_1},\dots,\phi_{i_k}\bar{\psi}^{\ell_k},\tau,\dots,\tau\rangle_{g,k+m,d}
\een
are defined in the same way as the descendant ones except that instead
of the descendant classes $\psi_i$ $(1\leq i\leq k)$ we use
$\bar{\psi}_i:={\rm ft}^*\psi_i$, where ${\rm ft}\colon 
\overline{\cM}_{g,k+m}(Y,\beta)\to \overline{\cM}_{g,k}$ is the map
that forgets the map to $Y$, the orbifold structure on the domain
curve, the last $m$ marked points, and it contracts all unstable
components. If $\overline{\cM}_{g,k}=\emptyset$, i.e., $2g-2+k\leq 0$,
then the ancestor invariant is by definition $0$. Let us point out
that in general the dependence on $\tau$ is only formal, i.e., the
ancestor invariant is a formal power series in $\tau$. 

The generating function 
\ben
\overline{\cF}_{\tau,Q}^{(g)}(\mathbf{t}) = \sum_{n=0}^\infty \sum_{d\in
  \operatorname{Eff}(Y)} \frac{Q^d}{n!}\, 
\llangle[\Big] \mathbf{t}(\bar{\psi}),\dots, \mathbf{t}(\bar{\psi})\rrangle[\Big]_{g,k,d}(\tau)
\een
is called the {\em genus-$g$ total ancestor potential} and 
\ben
\mathcal{A}_{\tau,Q}^Y(\hbar;\mathbf{t}) := \exp \Big(
\sum_{g=0}^\infty \overline{\cF}_{\tau,Q}^{(g)}(\mathbf{t}) \, \hbar^{g-1}\Big)
\een
is called the {\em total ancestor potential} of $Y$.
The relation between ancestors and descendants is completely
determined by the calibration $S(\tau,z)$ and the genus-1 primary
potential of $Y$ (see \cite{G2} for more details). Thanks to the divisor
equation we have the following symmetry
\ben
\mathcal{A}_{\tau,Q}^Y = \mathcal{A}_{\tau-\sum_{i=1}^r P_i\log Q_i,1}^Y,
\een
where $P_i=c_1(L^{(i)})$ and $Q=1$ means $Q_i=1$ for all
$i$. Therefore, without loss of generality we may set $Q_i=1$ for all
$i$ and work with $\mathcal{A}^Y_\tau:=\cA_{\tau,1}^Y$.

\subsection{$I$-function} 

Let us return to the case when $Y$ is the orbifold Fermat CY
hypersurface. 
\subsubsection{Combinatorics of the inertia orbifold}
Let 
\ben
&&
v_i = (0,\dots,d_i,\dots,0)\in \mathbb{Z}^{n-1}\quad (1\leq i\leq n-1) \\
&&
v_n = (-d_n,\dots,-d_n)\in \mathbb{Z}^{n-1}.
\een
Let $\Sigma$ be the fan consisting of all subcones of
$\tau_1,\dots,\tau_n$, where $\tau_i$ is the cone in
$\mathbb{R}^{n-1}$ spanned by the $(n-1)$ rays (without $v_i$)
\ben
v_1,\dots,\check{v}_i,\dots, v_n.
\een
Note that $\mathbf{P}$ is the toric orbifold corresponding to the fan
$\Sigma$, so according to the general theory (see \cite{BCS}) the connected components
of $\mathbf{P}$ are parametrized by the set 
\ben
\operatorname{Box}(\Sigma) = \{ c\in \mathbb{Q}^n\ |\ 
0\leq c_i < 1,\
\operatorname{supp}(c)\subset \sigma,\ 
\sum_{i=1}^n c_i v_i \in \mathbb{Z}^{n-1}\ 
\mbox{for some}\ \sigma\in \Sigma \}
\een
where $\operatorname{supp}(c)$ is the set of all $v_i$ such that $\ c_i\neq 0$.
We have 
\ben
I\mathbf{P} = \coprod_{c\in  \operatorname{Box}(\Sigma) } \mathbf{P}_c,
\quad \text{where} \quad
\mathbf{P}_c = [\{x_1c_1=\cdots = x_n c_n=0\}/G].
\een
The dimension of $\mathbf{P}_c$ is one less than the number of $i$ such that $c_i=0.$
The orbifold $\mathbf{P}_c$ is non-reduced and it has a generic
stabilizer
\ben
G_c := \prod_{i: c_i\neq 0} \boldsymbol{\mu}_{d_i}. 
\een
In particular, the order $|G_c|=\prod_{i: c_i\neq 0} d_i$.

The inertia orbifold $IY$ is a suborbifold of $I\mathbf{P}$, we have 
\ben
\dim(Y_c) = \dim(\mathbf{P}_c) -1,
\een
so the twisted sectors are parametrized by $c\in
\operatorname{Box}(\Sigma)$ such that $\dim(\mathbf{P}_c)>0$; i.e., $c$
has at least 2 entries that are 0. We denote the set of all such $c$ by
$\operatorname{Box}_Y(\Sigma)$. 

\subsubsection{Cohomology and Poincar\'e pairing}
The coarse moduli spaces of $Y$ and $\mathbf{P}$ are respectively
$\mathbb{P}^{n-2}$ and $\mathbb{P}^{n-1}$. Indeed, the map 
\beq\label{pi-map}
\pi_{\mathbf{P}}\colon (\mathbb{C}^n\setminus{\{0\}})\times G \to
(\mathbb{C}^n\setminus{\{0\}})\times \mathbb{C}^*,
\quad (x,t)\mapsto (x,t_1^{d_1})
\eeq
induces a map of orbifolds that maps the pair $(|Y|,|\mathbf{P}|)$
homeomorphically onto $(\mathbb{P}^{n-2},\mathbb{P}^{n-1})$. Let us
define $L=\pi_{\mathbf{P}}^*\mathcal{O}_{\mathbb{P}^{n-2}}(1)$ and
$p=c_1(L)\in H^2(Y,\mathbb{Z})$. A basis in $H_{\rm CR}(Y;\mathbb{C})$
can be fixed as follows:
\ben
p^k\mathbf{1}_c,\quad 0\leq k\leq \dim(Y_c),\quad c\in \operatorname{Box}_Y(\Sigma), 
\een
where $\mathbf{1}_c$ is the unit in $H(Y_c;\mathbb{C})$. 
Given a cohomology class $p^k\mathbf{1}_c \in H^{2k}(Y_c;\mathbb{C})$,
the orbifold Poincar\'e pairing $(p^k\mathbf{1}_c,p^{k'}\mathbf{1}_{c'} )$
is non-zero if and only if 
\ben
k+k'=\dim(Y_c),\quad c_i+c_i' \equiv 0 \ \operatorname{mod}\Z, \quad
1\leq i\leq n.
\een
If $(k,c)$ and $(k',c')$ satisfy the above conditions, then we have 
\ben
(p^k\mathbf{1}_c,p^{k'}\mathbf{1}_{c'} ) = \frac{1}{|G_c|}.
\een
Finally,
\ben
H_2(|Y|;\mathbb{Z})\cong \mathbb{Z},\quad \beta\mapsto d:=\langle c_1(\mathcal{O}(1)),\beta\rangle.
\een

\subsubsection{The $J$-function of $Y$}
Given $d\in \mathbb{Z}_{\geq 0}$ and $\nu\in (\mathbb{Z}_{\geq 0})^n$
we define
\ben
I_{d,\nu}(z) = 
\frac{\Gamma(1+d+pz^{-1})}{\Gamma(1+pz^{-1})}\,
\prod_{i=1}^n
\frac{\Gamma(1-c_i+(p/d_i)z^{-1})}{\Gamma(1-c_i+k_i+(p/d_i)z^{-1})}\, 
\mathbf{1}_c z^{d-|\nu|-|k|},
\een
where $k_i\in \mathbb{Z}$ and $0\leq c_i< 1$ are defined uniquely by
the identity
\ben
\frac{\nu_i-d}{d_i} = -k_i+c_i
\een
and we put $|\nu|=\nu_1+\cdots +\nu_n$ and $|k|=k_1+\cdots +k_n$.
Put
\ben
I_Y(t,Q,z) = e^{p\log Q/z}
\sum_{d=0}^\infty 
\sum_{\nu\in  \mathbb{Z}^n_{\geq 0}}
I_{d,\nu}(z)\,Q^d \frac{t^\nu}{\nu!} ,
\een
where
\ben
t^\nu=t_1^{\nu_1}\cdots t_n^{\nu_n},\quad \nu!=\nu_1!\cdots \nu_n!.
\een
Note that 
\ben
I_Y(t,Q,z) = f_0(Q)\mathbf{1} + z^{-1} f_1(t,Q)
+z^{-2}f_2(t,Q)+\cdots\ ,
\een
where $f_k(t,Q)\in H_{\rm CR}(Y;\mathbb{C})$ $(k\geq 1)$ and 
\ben
f_0(Q) = 1+\sum_{d=1}^\infty \frac{(dD)!}{(dw_1)!\cdots (d w_n)!}\, Q^d,
\een
where $D=\operatorname{lcm}(d_1,\dots,d_n)$ and $w_i=D/d_i$.

It will be convenient for our purposes to modify slightly Givental's
$J$-function and to work with 
\ben
J_Y(\tau,Q,z) = e^{p\log Q/z} \widetilde{J}_Y(\tau,Q,z)=
\widetilde{J}_Y(\tau+P\log Q,1,z),
\een 
where the 2nd equality is a consequence of the {\em divisor equation}.
The suborbifold $Y\subset \mathbf{P}$ is cut out by the section
$x_1^{d_1}+\cdots+x_n^{d_n}$ of the convex line bundle
$\pi_{\mathbf{P}}^*\mathcal{O}(1)$, where $\pi_{\mathbf{P}}\colon\mathbf{P}\to \mathbb{P}^{n-1}$ is
the map induced from \eqref{pi-map}. We
may recall Theorem 25 in \cite{CCIT} and get the following formula for
the $J$-function of $Y$. 
\begin{proposition}\label{CCIT}
If $\tau=f_1(t,Q)/f_0(Q)$, then $J_Y(\tau,1,z) f_0(Q)= z I_Y(t,Q,z)$.
\end{proposition}
The $I$-function is known to be a solution to a Picard--Fuchs differential
equation in $Q$, so it has a non-zero radius of convergence as a power
series at $Q=0$. Let $\Delta$ be the disc of convergence, $\Delta^*:=\Delta-\{0\}$, and
$\pi\colon \widetilde{\Delta^*} \to \Delta^*$ 
be the universal cover of $\Delta^*$. The function
$\tau$ in Proposition \ref{CCIT} defines a map
\beq\label{mirror-map}
\tau\colon \mathbb{C}^n\times \widetilde{\Delta^*}  \to H,\quad (t,Q) \mapsto \tau(t,Q):=f_1(t,Q)/f_0(Q)
\eeq
which will be called the {\em mirror map}.

\subsection{Mirror symmetry for $\mathcal{D}$-modules} 

Let us introduce the following family of polynomials:
\ben
f(x,t,Q) = \sum_{i=1}^n (x_i^{d_i}+t_i x_i) -\frac{1}{Q} x_1\cdots
x_n, 
\een
where 
\ben
(t,Q)\in \mathbb{C}^n\times
(\mathbb{P}^1\setminus{\{0,a_1,\dots,a_r,\infty\} }),
\een
where $a_i$ are the values of $Q$ for which the polynomial has a
non-isolated singularity. 
\begin{remark}
The radius of the disc $\Delta$ is precisely
$\operatorname{max}_{1\leq i\leq r} |a_i|.$ 
\end{remark}

\subsubsection{The twisted de Rham cohomology and the quantum
  $\cD$-module }
The main goal of this section is to construct an isomorphism between
the sheaf $\mathcal{F}$ of the twisted de Rham cohomology and quantum
$\cD$-module. 
Let $\mathcal{D}$ be the sheaf of differential operators 
\ben
\cO_{\mathbb{C}^n\times \Delta^* }[z] \Big\langle z\frac{\partial}{\partial
  t_1},\dots,z\frac{\partial}{\partial t_n},
zQ\frac{\partial}{\partial Q} \Big\rangle.
\een
We would like to construct a $\mathcal{D}$-module isomorphism  
\ben
\pi^*(\cF|_{\mathbb{C}^n\times \Delta^*}) \cong \tau^* (\cO_B\otimes H[z]),
\een
where $\pi \colon \mathbb{C}^n\times \widetilde{\Delta^*}\to
\mathbb{C}^n\times \Delta^*$ is the universal covering, $\tau$ is the
mirror map, and $B\subset H$ is the domain of convergence for the
quantum cohomology. 
The $\cD$-module structures on the LHS and the RHS of the above
isomorphism are induced respectively from the Gauss--Manin
connection and the Dubrovin's connection, i.e., 
\ben
z\partial_{t_a}   \mapsto z\partial_{ t_a} - \sum_{i=1}^N
\frac{\partial \tau_i}{\partial t_a}(t,Q)\, \phi_i \bullet,\quad 
zQ \partial_{Q}   \mapsto zQ\partial_Q - \sum_{i=1}^N
Q\frac{\partial \tau_i}{\partial Q}(t,Q)\, \phi_i \bullet, 
\een
where $\tau(t,Q)=:\tau_1(t,Q)\phi_1+\cdots +\tau_N(t,Q)\phi_N$. Put 
\ben
I^\be_Y(t,Q,z) := (z\partial_t)^\be I_Y(t,Q,z),\quad \be=(\be^{(1)}, \cdots, \be^{(n)})\in \mathbb{Z}^n_{\geq 0},
\een
where 
\ben
 (z\partial_t)^\be = (z\partial_{t_1})^{\be^{(1)}}\cdots (z\partial_{t_n})^{\be^{(n)}} .
\een
Let $|\be|:=\be^{(1)}+\cdots+\be^{(n)}$.
The Taylor's series of $I^\be_Y(t,Q,z)$ at $Q=t=0$ takes the form
\ben
I^\be_Y(t,Q,z)=e^{p\log Q/z}\sum_{d=0}^\infty\sum_{\nu\in
  \mathbb{Z}^n_{\geq 0}} Q^d \frac{t^\nu}{\nu!}
I_{d,\nu+\be}(z) z^{|\be|}.
\een
We will need the following lemma.
\begin{lemma}\label{Idnu}
Let $d\geq 0$ and $\nu \in \mathbb{Z}^n_{\geq 0}$.

a) If $m=(1,\dots,1)\in \mathbb{Z}^n$, then
\ben
I_{d+1,\nu+\be+m}(z)\, z^{|\be|+n} = (1+d+pz^{-1})I_{d,\nu+\be}(z)\, z^{|\be|+1}.
\een

b) If $e_i\in \mathbb{Z}^n$ is the vector with a non-zero entry equal to $1$ only at
the $i$-th place, then  
\ben
I_{d,\nu+\be+d_i e_i}(z) = d_i^{-1}(d-\nu_i-\be^{(i)}+pz^{-1})I_{d,\nu+\be}(z) z^{-d_i+1}.
\een 
%where $e^i$ is the $i$-th entry of $e$.
\end{lemma}
The proof follows immediately from the definitions and it is
omitted. We also need the following Lemma, which is a corollary of
Proposition \ref{H:vb}. 
\begin{lemma}
The sheaf $\cF|_{\mathbb{C}^n\times \Delta^*}$ is a free
$\cO_{\mathbb{C}^n\times \Delta^*}[z]$-module of rank  
$N.$
\end{lemma} 
The main result of this subsection can be stated as follows.
\begin{proposition}\label{Dmod-iso}
 The assignment
\ben
x^\be dx/Q\mapsto S(\tau,1,-z)\, I^\be_Y(t,Q,z) ,
\een
where $\tau=\tau(t,Q)$ is the mirror map, induces an isomorphism of
$\mathcal{D}$-modules 
\ben
\Mir\colon \pi^*(\cF|_{\mathbb{C}^n\times \Delta^*}) \to 
\tau^* (\cO_B\otimes H[z]).
\een
\end{proposition}
\proof
According to Proposition \ref{CCIT}, $S(\tau,1,-z)\, I^\be_Y(t,z) \in H[z]$.
To prove that we have an induced map $\Mir$ we have to prove that
$\Mir$ maps 
\ben
(zd+df\wedge) x^\be dx_1\cdots \check{dx_i}\cdots dx_n 
\een
to $0$ for all $i=1,2,\dots,n$ and $\be=(\be^{(1)},\dots,\be^{(n)})\in\mathbb{Z}^n_{\geq 0}$. The above form takes the form
\ben
(z\be^{(i)} x^{\be-e_i}+d_i x^{\be+(d_i-1)e_i}+t_i x^\be-Q^{-1} x^{\be+m-e_i}) dx,
\een
where $m,e_i\in \mathbb{Z}^n$ are the same as in Lemma
\ref{Idnu}. Shifting $\be\mapsto \be+e_i$ we get
\ben
(z(\be^i+1) x^{\be}+d_i x^{\be+d_ie_i}+t_i x^{\be+e_i}-Q^{-1} x^{\be+m}) dx.
\een 
It is enough to prove that $\Mir$ maps the above form to
$0$. Recalling the definition of $\Mir$ we get that the
above form is mapped to
\ben
\sum_{d,\nu} \Big(
z(e^{(i)}+1) I_{d,\nu+\be}z^{|\be|}+d_i I_{d,\nu+\be+d_ie_i}z^{|\be|+d_i}+t_i
I_{d,\nu+\be+e_i}z^{|\be|+1}-Q^{-1} I_{d,\nu+\be+m}z^{|\be|+n}\Big) 
Q^d\frac{t^\nu}{\nu!}.
\een 
The terms in the brackets are transformed as follows: for the 2nd one
we apply Lemma \ref{Idnu}, b); for the 3rd term we shift the index
$\nu\mapsto \nu-e_i$ and use that 
\ben
t_i\frac{t^{\nu-e_i}}{(\nu-e_i)!}  = \nu_i \frac{t^\nu}{\nu!};
\een 
for the 4th term we first shift the index $d\mapsto d+1$ and then
recall Lemma \ref{Idnu}, a). We get
\ben
\sum_{d,\nu}
((\be^{(i)}+1)+(d-\nu_i-\be^{(i)}+pz^{-1})+\nu_i-(1+d+pz^{-1}))I_{d,\nu+\be}(z)\,
z^{|\be|+1}Q^d\frac{t^\nu}{\nu!} = 0.
\een

To prove that $\Mir$ is a $\mathcal{D}$-module morphism
we need only to verify that
\ben
\Mir(zQ\partial_Q [x^\be dx/Q]) = zQ\partial_Q \Mir([x^\be dx/Q]).
\een
Since
\ben
\Mir(zQ\partial_Q [x^\be dx/Q] )= \sum_{d,\nu}
(I_{d+1,\nu+\be+m} z^{|\be|+n} -I_{d,\nu+\be} z^{|\be|+1})Q^d\frac{t^\nu}{\nu!} 
\een
the above identity follows from Lemma \ref{Idnu}, a). 

Finally, since 
\beq\label{IY}
I_Y(t,Q,z) = f_0(Q) S(\tau,1,-z)^{-1}\, \mathbf{1} ,
\eeq
the map $\Mir$ is surjective. Since both sheaves are free
$\cO_{\mathbb{C}^n\times \widetilde{\Delta^*}}[z]$-modules of rank
$N$, the map must be an isomorphism. 
\qed

\subsection{Pairing matches}

Let us introduce the following pairing 
\beq\label{GW-hrp}
\widetilde{K}(\omega_1,\omega_2) =
-(\Mir(\pi^*\omega_1),\Mir(\pi^*\omega_2)^*),
\quad \omega_1,\omega_2\in \cF|_{\mathbb{C}^n\times \Delta^*}, 
\eeq
where $*$ is the involution in $H[z]$ induced from $z\mapsto
-z$, and $(\ ,\ )$ is the Poincar\'e pairing.
It is convenient to expand in the powers of $z$
\ben
\widetilde{K}(\omega_1,\omega_2)=:\sum_{p\in \mathbb{Z}}^\infty \widetilde{K}^{(p)}(\omega_1,\omega_2) z^{p}.
\een
The main goal in this section, which is one of the key ingredients in
the proof of our mirror symmetry theorem, is the following
proposition. 
\begin{proposition}\label{hrp=K}
The pairing $\widetilde{K}$ coincides with K. Saito's higher resdiue pairing. 
\end{proposition}
The main idea of the proof is to use Hertling's formula \eqref{SandK} in order to
obtain a formula for Saito's pairing similar to \eqref{GW-hrp}.

\subsubsection{The mirror map for vanishing cohomology }
Recall the notation from Sections \ref{S-PHS} and \ref{pf-hrp}. Let us
fix a polynomial $f_0$ corresponding to a reference point in
$\mathbb{C}^n\times \Delta^*$ and define for all $f$ sufficiently
close to $f_0$ a linear isomorphism
\beq\label{psi-iso}
\Psi:\mathfrak{h}=H^{n-1}(f^{-1}(1);\mathbb{C})\to H=H_{\rm
  CR}(Y;\mathbb{C}), 
\eeq  
such that
\beq\label{psi-def}
\Psi ( \widehat{s}(\omega_\be,z) )=
(-z)^{-\theta}I^\be_Y(t,Q,z),
\eeq
where $\omega_\be=x^\be dx/Q$ is a fixed set of weighted-homogeneous forms that induces a
trivialization of $\cF|_{\mathbb{C}^n\times \Delta^*}$.  Using that the
isomorphism in Proposition \ref{Dmod-iso} is a $\mathcal{D}$-module
isomrphism we can check that $\Psi$ is independent of $t$ and
$Q$. Note that the homogeneity of the $I$-function can be written in
the form
\ben
(z\partial_z+\sum_{i=1}^n (1-1/d_i)t_i\partial_{t_i} +
\operatorname{deg}_{\rm CR}) \, I^\be_Y(t,z) =
(\operatorname{deg}(\omega_\be)-1) \, I^\be_Y(t,z) .
\een
From this equation we get that $\Psi$ is independent of $z$ as
well. 
\begin{remark}
The map $\Psi$ is multivalued in $f$. It can be viewed as a trivialization of the pullback via
$\pi:\mathbb{C}^n\times \widetilde{\Delta^*}\to
\mathbb{C}^n\times \Delta^*$ of 
the vanishing cohomology bundle. 
\end{remark}

\subsubsection{The polarization form and the Poincar\'e pairing}

Let us introduce the following bilinear form on $H$:
\ben
\chi(a,b) := S(\Psi^{-1}(a),\nu^{-1} \Psi^{-1}(b)) ,
\een
where $S$ is the polarization form of Steenbrink's Hodge structure
(see Section \ref{S-PHS}). 
\begin{lemma}\label{M_mar}
The claim in Proposition \ref{hrp=K} is equivalent to the identity
\ben
\chi(a,e^{-\pi\sqrt{-1}\theta}b) = (a, b),\quad a,b\in H. 
\een
\end{lemma}
\proof
Let us first establish that $p\cup$ is an infinitesimal symmetry of
$\chi$. Let us denote by $M_{\rm mar}$ the
monodromy transformation of $\mathfrak{h}$ of the Gauss--Manin
connection around $Q=0$ in counter clockwise direction. The analytic
continuation around $Q=0$ transforms the RHS of \eqref{psi-def} into
\ben
(-z)^{-\theta}\ e^{2\pi\sqrt{-1} p/z}\, I^\be_Y(t,Q,z) = 
e^{-2\pi\sqrt{-1} p}\, 
(-z)^{-\theta}I^\be_Y(t,Q,z).
\een
Therefore, 
\ben
\Psi\circ M_{\rm mar}^{-1} = e^{-2\pi\sqrt{-1}p} \circ \Psi.
\een
In particular, $M_{\rm mar}$ is unipotent and there is uniquely defined nilpotent
operator $N_{\rm  mar}:=-\frac{1}{2\pi\sqrt{-1} } \log M_{\rm mar}.$
By definition
\ben
\chi(a,b)=(-1)^{(n-1)(n-2)/2} 
\langle \Psi^{-1}(a),\operatorname{Var}\circ \Psi^{-1}(b)\rangle.
\een
Since the form
$\langle\cdot,\operatorname{Var}(\cdot)\rangle$ is 
$M_{\rm  mar}$-invariant and 
\ben
\Psi(N_{\rm mar}\, A) = p\,\Psi(A),\quad A\in \mathfrak{h},
\een
we get that $\chi(p\cup a,b)+\chi(a,p\cup b)=0.$ Using this property
we can complete the proof as follows. By definition 
\ben
\widehat{s}(\omega_\be,z) = 
\Psi^{-1} ((-z)^{-\theta}I^\be_Y(t,Q,z) ) =
e^{\log Q \, N_{\rm mar}}\, \Psi^{-1}
((-z)^{-\theta}\widetilde{I}^\be_Y(t,Q,z)) , 
\een
where $I^\be_Y(t,Q,z) =: e^{p\log Q/z} \widetilde{I}^\be_Y(t,Q,z)$.
Recalling formula \eqref{SandK} we get
\beq\label{K-ext}
K(\omega_{\be'},\omega_{\be''}) = -
\chi(e^{\pi\sqrt{-1}\theta} (-z)^{-\theta}\widetilde{I}^{\be''}_Y(t,Q,-z),(-z)^{-\theta}\widetilde{I}^{\be'}_Y(t,Q,z)),
\eeq
In order to complete the proof, we just have to notice that
\eqref{GW-hrp} can be written as 
\ben
\widetilde{K}(\omega_{\be'},\omega_{\be''}) =
- 
((-z)^{-\theta} \widetilde{I}_Y^{\be''}(t,Q,-z),(-z)^{-\theta}
\widetilde{I}_Y^{\be'}(t,Q,z) ),
\een
where we used that $S(\tau,1,z)$ is a symplectic transformation, so
\ben
S(\tau,1,z)^t S(\tau,1,-z)=1.\qed
\een

\medskip 

We will compute $\chi$ by specializing \eqref{K-ext} to $t=Q=0$, i.e.,
we will express $\chi$ in  terms of the limit of the higher residue
pairing at $t=Q= 0$.  
To begin with, let us compute the Chen--Ruan product of $Y$. 
Put $\phi_i=\mathbf{1}_{(0,\dots,1/d_i,\dots,0)}$, where the non-zero
entry is on the $i$-th place. Let 
\ben
\phi_\be:=\phi_1^{\be^{(1)}}\cdots \phi_n^{\be^{(n)}}, 
\een
where $\be=(\be^{(1)},\dots,\be^{(n)})$ is a sequence of non-negative integers and
the monomial on the RHS is defined via the Chen--Ruan cup product. 
\begin{lemma}\label{CR-product}
The following formula holds 
\ben
\phi_\be=
(d_1^{-\ell_1}\cdots d_n^{-\ell_n})\,p^\ell\mathbf{1}_c,\quad c=(c_1,\dots,c_n), 
\een
where the numbers $\ell:=\ell_1+\cdots +\ell_n$ and $c_i$ are defined by 
\ben
\frac{\be^{(i)}}{d_i}=\ell_i+c_i,\quad 0\leq c_i<1,\quad \ell_i\in
\mathbb{Z}.
\een
\end{lemma}
\proof
Note that at $Q=0$ the $J$-function of $Y$ is  
\ben
\widetilde{J}_Y(\widetilde{\tau},0,z)= ze^{\widetilde{\tau}\cup_{\rm CR} /z},
\quad
\widetilde{\tau}:= t_1 \phi_1+\cdots + t_n\phi_n,
\een
where $\cup_{\rm CR}$ denotes the Chen--Ruan product. Note also
that $\tau(t,Q) = p\log Q + \widetilde{\tau} + O(Q)$ and that the calibration
$S(0,0,z)=1$. Recalling Proposition \ref{CCIT} we get that the 
vector $\widetilde{I}_Y^\be(0,0,z) $ is polynomial in $z$ and that its
free term 
\ben
\widetilde{I}_Y^\be(0,0,0)=(d_1^{-\ell_1}\cdots
d_n^{-\ell_n})\,p^\ell\mathbf{1}_c
\een 
must coincide with
$z^{-1}(z\partial_t)^\be \widetilde{J}_Y(t,0,z)|_{t=0}=\phi_\be$. 
\qed

Choosing $\be=(\be^{(1)},\dots,\be^{(n)})$ appropriately we can
arrange that the vectors $\phi_\be$ give a basis of the Chen--Ruan cohomology. Note that the
Chen--Ruan degree of $\phi_\be$ is
$$\deg_{\rm CR}\phi_\be=\operatorname{deg}(\be):=\sum_{i=1}^n
{\be^{(i)}\over d_i}.$$
\begin{lemma}
We have
\ben
\chi(\phi_{\be''},e^{-\pi\sqrt{-1}\theta} \phi_{\be'}) =
(\phi_{\be'},\phi_{\be''}).
\een
\end{lemma}
\proof
Recalling \eqref{K-ext} we get 
\beq\label{K-chi}
\left.
K(\omega_{\be'},\omega_{\be''}) \right|_{t=Q=0} =
-
z^{-n+\operatorname{deg}(\be')+\operatorname{deg}(\be'')+2}
\chi(\phi_{\be''},e^{-\pi\sqrt{-1}\theta} \phi_{\be'}). 
\eeq
We need to check that the higher residues
$-K^{(p)}(\omega_{\be'},\omega_{\be''})|_{t=Q=0}$ vanish for $p>0$ and
coincide with the Poinare pairing $(\phi_{\be'},\phi_{\be''})$ for
$p=0$. Using that $\chi(p\cup a,b)+\chi(a,p\cup b)=0$, we may
reduce the proof to the case when $\phi_{\be'}$ is not divisible by $p$,
i.e., $0\leq \be'^{(i)}\leq d_i-1$. Finally, let us assume also that $t=0$. The rest of
the proof is splitted into four steps. 

\medskip
\noindent
{\bf Step 1.} 
We claim that if $K(\omega_{\be'},\omega_{\be''})|_{Q=0}\neq 0$, 
then $\be'^i+\be''^i \equiv 0 \mod d_i$. 
Put  $\eta_i:=e^{2\pi\sqrt{-1}/d_i}$, then the rescaling 
\ben
(x_1,\dots,x_i,\dots,x_n)\mapsto (x_1,\dots,\eta_ix_i,\dots,x_n),\quad 
Q\mapsto \eta_i Q,
\een
defines an automorphism of $\HH_+(f)$. It is easy to see that the
higher residue pairing $K$ is invariant under the rescaling. We have
\ben
K(\omega_{\be'},\omega_{\be''}) = z^r\sum_{m=0}^\infty K_{\be',\be'',m}  Q^m,  
\een
where
$r:=\operatorname{deg}(\omega_{\be'})+\operatorname{deg}(\omega_{\be''}) -
n$. Rescaling the above identity we get that
\ben
\eta_i^{\be'^{(i)}+\be''^{(i)}}K(\omega_{\be'},\omega_{\be''}) = z^p\sum_{m=0}^\infty K_{\be',\be'',m}  Q^m\eta_i^m.
\een
It follows that $\be'^{(i)}+\be''^{(i)} \equiv m \mod d_i$, so if the pairing
$K(\omega_{\be'},\omega_{\be''})|_{Q=0}\neq 0$, then $\be'^{(i)}+\be''^{(i)}\equiv 0 
\mod d_i$.

\medskip
\noindent
{\bf Step 2.} 
We claim that if $K(\omega_{\be'},\omega_{\be''})|_{Q=0}\neq 0$, then
$\operatorname{deg}(\phi_{\be'})+ \operatorname{deg}( 
\phi_{\be''})=n-2.$ To prove this, recall that
$\phi_{\be''}=(d_1^{-\ell_1''}\cdots d_n^{-\ell_n''})\,p^{\ell''}\mathbf{1}_{c''}$, where $0\leq \ell''\leq
\operatorname{dim}(Y_{c''})$, $0\leq c_i'' < 1$ are defined by
\ben
\ell''=\sum_{i=1}^n \ell_i'',\quad e_i''/d_i=\ell_i''+c_i'',\quad
\ell_i\in \mathbb{Z}.
\een 
Using that $\be'^{(i)}+\be''^{(i)}\equiv 0 \mod d_i$ we get that 
\ben
\frac{\be'^{(i)}}{d_i}+\frac{\be''^{(i)}}{d_i} = 
\begin{cases}
\ell_i''+1, & \mbox{ if } c_i''\neq0,\\
\ell_i'', & \mbox{ otherwise}.
\end{cases}
\een
Recall that $\operatorname{dim}(Y_{c''})+2$ is the number of $i$ such that $c_i''=0$. Therefore, 
\ben
\operatorname{deg}(\phi_{\be'})+\operatorname{deg}(\phi_{\be''}) = 
\operatorname{deg}(\be')+\operatorname{deg}(\be'') = n-2+ \ell''- \operatorname{dim}(Y_{c''})\leq n-2.
\een
This proves that $K^{(r)}(\omega_{\be'},\omega_{\be''})|_{Q=0}$ could be
non-zero only if $r=0$, $\ell''= \operatorname{dim}(Y_{c''})$, and
$\operatorname{deg}(\omega_{\be'})+\operatorname{deg}(\omega_{\be''}) =n.$
It remains only to check that 
\beq\label{res:po}
K^{(0)}(\omega_{\be'},\omega_{\be''}) = -(\phi_{\be'},\phi_{\be''}).
\eeq

\medskip
\noindent
{\bf Step 3.} 
We claim that it is enough to verify \eqref{res:po} in the case
when  $\be'^{(i)}+\be''^{(i)} = d_i$ for $n-2$ values
of $i$ and $\be'^{(i)}=\be''^{(i)}=0$ for the remaining two other values. Indeed,
we may assume again that $0\leq \be'^{(i)}\leq d_i-1$ and write
$\phi_{\be''}=(d_1^{-\ell_1''}\cdots d_n^{-\ell_n''})\,p^{\ell''}\mathbf{1}_{c''}$ as we did above. Since the set
$I(c''):=\{i\ |\ c_i''=0\}$ contains $\operatorname{dim}(Y_{c''})+2$ elements we can choose a
subset $J\subset I(c'')$ with $\ell''$ elements. Note that $\be'^{(j)}=0$
for all $j\in J$, because $\be'^{(j)}+d_jc_j''=0 \mod d_j$. Let us
define $\widetilde{\be}\,''^{(i)}=d_i c_i''$ $(1\leq i\leq n)$ and
\ben
\widetilde{\be}\,'^{(j)}=
\begin{cases}
d_j, & \mbox{ if } j\in J \\
\be'^{(j)}, & \mbox{ otherwise}.
\end{cases}
\een 
Using the relation $d_jx_j^{d_j} = Q^{-1} x_1\cdots x_n$ in
$\operatorname{Jac}(f)$ we get 
\ben
K^{(0)}(\omega_{\be'},\omega_{\be''}) = 
\Big(\prod_{i\notin J}d_i^{-\ell''_i}\Big)
K^{(0)}(\omega_{\widetilde{\be}\,'},\omega_{\widetilde{\be}\,''}). 
\een
Similarly, using the relation $d_i\phi_i^{d_i}=p$ in the Chen--Ruan cohomology,
we get 
\ben
(\phi_{\be'},\phi_{\be''}) = 
\Big(\prod_{i\notin J}d_i^{-\ell''_i}\Big)
(\phi_{\widetilde{\be}\,'},\phi_{\widetilde{\be}\,''}),
\een
which completes the proof of our claim. 

\medskip
\noindent
{\bf Step 4.} 
Due to permutation symmetry of our computation, it is enough to
consider only the case $\be'^{(i)}+\be''^{(i)} = d_i$ for $1\leq i\leq n-2$. By
definition 
\ben
K^{(0)}(\omega_{\be'},\omega_{\be''}) = 
\Res 
\frac{
Q^{n-2}
x_1^{d_1}\cdots x_{n-2}^{d_{n-2} }  dx_1\cdots dx_n }{
(Qd_1 x_1^{d_1-1} -x_2\cdots x_n)\cdots (Qd_n x_n^{d_n-1} -x_1\cdots
x_{n-1}) }.
\een
Since $x_i^{d_i} = (x_1\cdots x_n)/(d_i Q)$ in the Jacobi ring of $f$,
the above residue turns into 
\ben
\frac{1}{d_1\dots d_{n-2} }\,  
\Res 
\frac{
(x_1\cdots x_n)^{n-2} dx_1\cdots dx_n }{
(Qd_1 x_1^{d_1-1} -x_2\cdots x_n)\cdots (Qd_n x_n^{d_n-1} -x_1\cdots
x_{n-1}) }.
\een
Since the Poincar\'e pairing $(\phi_{\be'},\phi_{\be''})$ equals $1/(d_1\cdots
d_{n-2})$, to complete the proof we have to verify that the above
residue is $-1$ when $Q=0$. 

In order to compute the residue, recall that
\beq\label{Hess-res}
\Res \Big(
\frac{df_{x_1}}{f_{x_1}} \wedge \cdots \wedge \frac{df_{x_n}}{f_{x_n}}\Big)
=
\Res 
\Big(
\frac{
\operatorname{det}(\operatorname{Hess}(f)) }{
f_{x_1}\cdots f_{x_n} } \, dx_1\cdots dx_n \Big)= 
N.
\eeq
On the other hand $df_{x_1}\wedge\cdots\wedge df_{x_n}$ is given by 
\ben
(d_1 x_1^{d_1-1} dx_1 -Q^{-1} d(x_2\cdots x_n) ) \wedge \cdots \wedge 
(d_n x_n^{d_n-1} dx_1 -Q^{-1} d(x_1\cdots x_{n-1}) ). 
\een
This wedge product can be computed explicitly as follows:
\begin{align}\label{Hess-wedge}
\sum_{m=0}^{n}\sum_{1\leq i_1<\cdots <i_m\leq n} 
(1+m-n)\, Q^{-n+m}\,
\left(
\prod_{s=1}^m d_{i_s}(d_{i_s}-1) x_{i_s}^{d_{i_s}+n-m-2}
\right)\times \\
\notag
\times 
(x_{j_1}\cdots x_{j_{n-m}})^{n-m-2}\,
dx_1\wedge \cdots \wedge dx_n,
\end{align}
where $\{j_1,\dots,j_{n-m}\}=\{1,2,\dots,n\}\setminus
\{i_1,\dots,i_m\}$. In the derivation of the above formula we used the
following simple fact: if $g(y_1,\dots,y_k)=y_1\cdots y_k$, then 
\ben
\operatorname{det}(\operatorname{Hess}(g)) = (-1)^{k}(1-k)\, (y_1\cdots y_k)^{k-2}.
\een
This formula is applied for $k=n-m$ and $(y_1\cdots
y_k)=(x_{j_1},\dots,x_{j_{n-m}}).$ Note that the term with $m=n-1$ in
\eqref{Hess-wedge} vanishes, while the term with $m=n$ reads
\ben
\prod_{i=1}^n d_i(d_i-1) x_i^{d_i-2} dx_1\wedge \cdots \wedge dx_n.
\een
The contribution of this term to the residue \eqref{Hess-res} is
analytic at $Q=0$ and it vanishes at $Q=0$ of order at least 2.  
Therefore, we may assume that the summation range for $m$ in
\eqref{Hess-wedge} is up to $n-2$. Using the relations in the Jacobi
ring of $f$ we get 
\ben
d_i(d_i-1)x_i^{d_i+n-m-2} =(d_i-1) x_i^{n-m-2} \, (x_1\cdots
x_n)\, Q^{-1} .
\een
The sum \eqref{Hess-wedge} is transformed into
\begin{align}\notag
\sum_{m=0}^{n-2}\sum_{1\leq i_1<\cdots <i_m\leq n} 
(1+m-n)\, 
\left(
\prod_{s=1}^m (d_{i_s}-1) 
\right)\times \\
\notag
\times 
Q^{-n}\,
(x_1\cdots x_n)^{n-2}\,
dx_1\wedge \cdots \wedge dx_n,
\end{align}
Note that the sum on the first line can be computed explicitly. We
have the following identity: 
\ben
\sum_{m=0}^n \sum_{1\leq i_1<\cdots <i_m\leq n} 
(d_{i_1}-1)\cdots(d_{i_m}-1) 
x^{1-n+m} = x^{1-n} (x(d_1-1)+1)\cdots (x(d_n-1)+1).
\een
Differentiating with respect to $x$ and setting $x=1$ we get 
\ben
\sum_{m=0}^n \sum_{1\leq i_1<\cdots <i_m\leq n} 
(d_{i_1}-1)\cdots(d_{i_m}-1) (1-n+m) = 0;
\een
therefore,
\ben
\sum_{m=0}^{n-2} \sum_{1\leq i_1<\cdots <i_m\leq n} 
(d_{i_1}-1)\cdots(d_{i_m}-1) (1-n+m) = -(d_1-1)\cdots (d_n-1) = -N. 
\een
Restricting the identity \eqref{Hess-res} to $Q=0$ gives
\ben
\left.
\Res\, 
\frac{Q^{-n}\,
(x_1\cdots x_n)^{n-2}}
{f_{x_1}\cdots f_{x_n}}\, 
dx_1\wedge \cdots \wedge dx_n\right|_{Q=0} = -1. \qed
\een

\subsection{Mirror symmetry in genus $0$}

We enumerate the elements of the basis $\{p^k\mathbf{1}_c\}$ of
$H^*(Y,\mathbb{C})$ in an arbitrary way and denote by $\phi_i$ the
$i$th element. It is convenient to enumerate in such a way that  
\ben
\phi_a=\mathbf{1}_{e_a/d_a},\quad 1\leq a\leq n,
\een
where $e_a$ is the $a$-th coordinate vector in $\mathbb{Z}^n$. 
Let $\tau=(\tau_1,\dots,\tau_N)$ be the linear coordinates on
$H^*(Y,\mathbb{C})$ corresponding to the basis $\{\phi_i\}$ and put
$\partial_i:=\partial/\partial \tau_i$ $(1\leq i\leq N)$. 

\subsubsection{The big quantum cohomology of $Y$}
Let us fix a constant $Q\in \Delta^*$ and specify a value of $\log Q$,
so that the mirror map $\mathbb{C}^n\to H, t\mapsto \tau(t,Q)$ is
analytic. In this way $\mathbb{C}^n$ is identified with an analytic 
subvariety $\Sigma$ of $H^*(Y,\mathbb{C})$. 
The linear coordinates $(\tau_1,\dots,\tau_n)$ form a coordinate
system on $\Sigma$, because on $\Sigma$ we have
\ben
\tau_a = t_a \ (\operatorname{mod}\ Q),\quad 1\leq a\leq n.
\een
\begin{lemma}\label{S-diff}
There are differential operators
\ben
P_i(z,t,Q; z\partial_{t_1},\dots,z\partial_{t_n})\in
\mathbb{C}\{Q\}[z,t]\langle z\partial_{t_1},\dots,z\partial_{t_n}\rangle,\quad
1\leq i\leq N,
\een
such that
\ben
P_i \, J_Y(\tau,1,z) \quad \in \quad  z \, \phi_i + H[\![z^{-1}]\!],
\een
where $\tau=\tau(t,Q)$. Moreover, for any choice of such differential operators, we have 
\ben
P_i \, J_Y(\tau,1,-z) = -z S(\tau,1,z)^{-1}\phi_i.
\een
\end{lemma}
\proof
Put 
\ben
J^\be_Y(t,Q,z) := (z\partial_t)^\be J_Y(\tau(t,Q),1,z)=zI^\be_Y(t,Q,z)/f_0(Q).
\een
Using the quantum differential equations, we get 
\ben
J^\be_Y(t,Q,z)
=
z(z\partial_t)^\be S(\tau,1,-z)^{-1} 1 =z S(\tau,1,-z)^{-1}\, \prod_{a=1}^n
(z\partial_{t_a} -M_a)^{\be^{(a)}} 1, 
\een
where $M_a=\partial_{t_a}\bullet_t$ is the operator of quantum multiplication by $\partial_{t_a}$.
Let us choose a set of indexes $\be=(\be^{(1)},\dots,\be^{(n)})$ such that the
cohomology classes $\phi_\be:=\widetilde{I}^\be_Y(0,0,0)$ form a basis of
$H^*(Y,\mathbb{C})$. For example, for a given basis vector $p^k\mathbf{1}_c$ if we define
\ben
\be^{(1)}=(k+c_1)d_1, \quad \be^{(2)}=c_2d_2,\quad\dots,\quad \be^{(n)}=c_n d_n,
\een
then $\phi_\be=d_1^{-k} p^k \mathbf{1}_c . $ 
In order to prove the existance of the differential operators $P_i$,
it is enough to prove that the determinant 
of the matrix $C$ whose columns are  $\prod_{a=1}^n
(z\partial_{t_a}-M_a)^{\be^{(a)}}1$ is independent of $z$ and $t$. Under
this assumption, the inverse of the matrix $C$ has entries in
$\mathbb{C}\{Q\}[z,t]$, so the columns of $zS(\tau,Q,-z)^{-1}$ can be written
as linear combinations of $J^\be_Y(t,Q,z)$ with coefficients
in $\mathbb{C}\{Q\}[z,t]$, which is what we have to prove. 

Note that 
\ben
(z\partial_z +E+\operatorname{deg}_{\rm CR}) \widetilde{I}^\be_Y(t,Q,z) = \operatorname{deg}(\be) \widetilde{I}^\be_Y(t,Q,z) ,
\een
where $\operatorname{deg}(\be) :=\sum \be^{(i)}/d_i$. The determinant
$\Delta_I(t,z)$ of the matrix with columns $\widetilde{I}^\be_Y(t,Q,z)$ satisfies the
following differential equation
\ben
(z\partial_z+E) \Delta_I(t,z) = \Big(\sum_e \operatorname{deg}(e) -
\operatorname{Tr}( \operatorname{deg}_{\rm
  CR})\Big)\Delta_I(t,z) =0.
\een
Similarly, the calibration $S(\tau,Q,-z)$ is known to satisfy the
differential equation
\ben
(z\partial_z+E+\operatorname{deg}_{\rm CR} )S(\tau,Q,-z) = S(\tau,Q,-z) \operatorname{deg}_{\rm CR}.
\een
Hence, the determinant also satisfies
\ben
(z\partial_z+E)\operatorname{det}(S(\tau,Q,-z)) = 0.
\een
We get $(z\partial_z+E)\operatorname{det}(C)=0$. However, the matrix
$C$ depends holomorphically on $t$ and $z$ at $t=z=0$, so
$\operatorname{det}(C)$ is a constant independent of $t$ and $z$. 

Let us assume that $\widetilde{P}_i$ $(1\leq i\leq N)$ is another set
of differential operators such that $\widetilde{P}_i\,
J_Y(\tau(t,Q),1,z) \in z\phi_i+ H[\![z^{-1}]\!]$. Using that the
calibration solves the quantum differential equations we get 
\ben
\widetilde{P}_i\, J_Y(\tau,1,z)z^{-1} =S(\tau,1,-z)^{-1}
g(t,Q,z),\quad g\in H[z].
\een
The projection of the LHS of the above identity on $H[z]$ is by
definition $\phi_i$, so on the RHS we must have $g=\phi_i$.
\qed

\bigskip

Note that in the proof of Lemma \ref{S-diff}, we obtained an explicit
algorithm to find differential operators $P_i$ from the
$I$-function. Namely, let us choose a set of $N$ indices $\be$ such that
the vectors $\widetilde{I}^\be_Y(0,0,0)$ give a basis of $H$. The matrix $A(t,Q,z)$ whose columns are the
vectors $\widetilde{I}^\be(t,Q,z)/f_0(Q)$ has a Birkhoff factorization $A_-(t,Q,z)
A_+(t,Q,z)$ with $A_-=1+O(z^{-1})$. The entries of the $i$-th column of the matrix
$A_+(t,Q,z)^{-1}$ determine the coefficients of a differential
operator $P_i$ that has the required properties.  In particular, the
$I$-function determines explicitly $S(\tau,1,-z)=A_-(t,Q,z)^{-1}$ for all $\tau\in
\Sigma$ and the operators $M_a(t,Q)$, $(1\leq a\leq n)$ of
quantum multiplication by $\partial_{t_a}$. 
\begin{lemma}\label{reconstr}
The big quantum cohomology of $Y$ is uniquely determined by the
polynomials $P_i$, $1\leq i\leq N$ and the flatness of the Dubrovin's
connection. 
\end{lemma}
\proof
Let us denote by $\Omega_i(\tau,Q)$ the linear operator of
quantum multiplication by $\phi_i\bullet_{\tau,Q}$. Lemma \ref{S-diff} implies that
\ben
\Omega_i(\tau,Q) = P_i(0,t,Q;-M_1,\dots,-M_n),\quad \tau=\tau(t,Q)\in \Sigma.
\een
so the restriction of the multiplication operators to $\Sigma$ is also
uniquely determined. Note that $\Omega_i(0,0)$ generate the orbifold
cohomology. In fact, using that the J-function at $Q=0$ is
$e^{\tau/z}$ we get that 
\ben
\Omega_1(0,0)^{\nu_1}\cdots \Omega_n(0,0)^{\nu_n} = \Big(\prod_{i=1}^n
d_i^{-\ell_i}\Big) \, p^\ell \mathbf{1}_c,  
\een  
where the numbers $c=(c_1,\dots,c_n)$, $\ell_1,\dots,\ell_n$, and
$\ell:=\ell_1+\cdots +\ell_n$ are uniquely defined by 
\ben
\nu_i/d_i = \ell_i + c_i,\quad \ell_i\in \mathbb{Z},\quad 0\leq c_i<1.
\een
The matrix $A=A(t,Q;\Omega_1,\dots,\Omega_n)$ with columns
\ben
P_i(0,t,Q;-\Omega_1,\dots,-\Omega_n)\,\mathbf{1},\quad 1\leq i\leq N
\een
is non-degenerate, because at $t=Q=0$ it reduces to the identity
matrix. The quantum multiplication is commutative; therefore,
$\Omega_i(\tau,Q) A$ coincides with the matrix
$B=B(t,Q;\Omega_1,\dots,\Omega_n)$ whose columns are given by 
\ben
P_i(0,t,Q;-\Omega_1,\dots,-\Omega_n)\,\phi_i,\quad 1\leq i\leq N.
\een
Here $A$ and $B$ are viewed as functions in $t,Q$ and the entries of
the matrices $\Omega_1,\dots,\Omega_n$. It follows that 
$\Omega_i(\tau,Q)=BA^{-1}$ is a rational function
$R_i(\tau,Q;\Omega_1,\dots,\Omega_n)$ in the entries of $\Omega_a$
$(1\leq a\leq n)$. Using the flatness of Dubrovin's
connection we get
\ben
\partial_i \Omega_a= \partial_a\Omega_i = \partial_a
R_i(\Omega_1,\dots,\Omega_n),
\een
and we get that the restriction of all higher order derivatives in $\tau$ of
$\Omega_a(\tau,Q)$, $1\leq a\leq n$ to $\Sigma$ are uniquely
determined. In particular, we can express the higher order derivatives in $\tau$ of
$\Omega_a(\tau,Q)$ at $\tau=0$ in terms of the polynomials $P_i$,
which completes the proof.  \qed

\subsubsection{The mirror isomorphism}

Let us fix $Q\in \Delta^*$ and define
\ben
f:=f(x,0,Q)=\sum_{i=1}^n x_i^{d_i}
-\frac{1}{Q}x_1\cdots x_n.
\een 
Let us embed $\mathbb{C}^N\subset B_f$  via $t\mapsto f(x,t,Q)$. Put 
\ben
\omega_i=\sqrt{-1}\, \Mir^{-1}(\phi_i) 
\quad \in \quad \cF|_{\mathbb{C}^n\times \{Q\} },
\een
where the scalar $\sqrt{-1}$ is chosen, so that the Poincar\'e
pairing matches the residue pairing $K^{(0)}$ (see Proposition
\ref{hrp=K}).
According to Proposition \ref{hrp=K}, the forms $\omega_i$ form a good
basis, i.e., $K(\omega_i,\omega_j)\in \mathbb{C}$. Let us assume 
that $\phi_1=1$ and define $\omega:=\omega_1$ to be the
primitive form. 
The good basis $\{\omega_i\}_{i=1}^N$ extends uniquely to a good basis over the space
$B_f$ of miniversal deformations of $f$ and it determines a flat
coordinate system $\tau=(\tau_1,\dots,\tau_N)$ on $B_f$ such that 
\ben
\sum_{i=1}^N \tau_i(t)\phi_i = \tau(t,Q),\quad \forall t\in \mathbb{C}^n.
\een 
We can use the map $\Mir$ to obtain a reconstruction of the Frobenius
structure on $B_f$ similar to the reconstruction of the big quantum
cohomology given by Lemmas \ref{S-diff} and \ref{reconstr}. Namely,
using Proposition \ref{Dmod-iso}, it is easy to verify that 
the statements of both lemmas remain the same if we replace 
$J_Y(\tau,1,z)$ with the primitive form $\omega$ and Dubrovin's
connection with the Gauss--Manin connection.  Note that since $\Mir$ is
a $\cD$-module isomorphism, we can use the same set of differential
operators  $P_i$ for both reconstructions. Therefore, we can uniquely
extend the mirror map $\mathbb{C}^n\to H$, $t\mapsto \tau(t,Q)$ to an isomorphism of
Frobenius structures, i.e., we have the following proposition. 
\begin{proposition}\label{ms-0}
The map
\beq\label{full-mm}
B_f\to H:=H_{\rm CR}^*(Y;\mathbb{C}),\quad
\tau=(\tau_1,\dots,\tau_N)\mapsto \sum_{i=1}^N \tau_i\phi_i 
\eeq
induces an isomorphism of the germ of the Frobenius structure of
$B_f$ at $\tau=0$ and the quantum cohomology of $Y$ . 
\end{proposition}

\begin{remark}
The map $\Mir$ is defined in terms of an extended $I$-function
of $Y$ depending on the relevant deformation parameters
$t_1,\dots,t_n$. Using the results of \cite{CCIT} we can 
extend the $I$-function even further to include all deformation
parameters. This would give us an appropriate extension of the map
$\Mir$, which will provide us directly with a trivialization of
$TB_f[\![z]\!] \cong B\times H[\![z]\!]$ that intertwines the
Gauss--Manin  connection and the Dubrovin connection. The advantage of
using a reconstruction argument is that, after analyzing the reconstruction scheme more carefully, we can prove the convergence
of the Frobenius multiplication on $TB_f$ in the irrelevant
direction. 
\end{remark}

\subsection{Mirror symmetry in higher genus}
Proposition \ref{ms-0} implies that the quantum cohomology of $Y$ is
semi-simple. Therefore, we can recall Givental's higher genus
reconstruction \cite{G1,G2} proved by Teleman \cite{Te}. Let  
$\tau=\tau(0,Q)\in H$ be the value of the mirror map \ref{mirror-map}
at $t=0$, put $f:=f(x,0,Q)$, and  recall the good basis 
\ben
\omega_i=\sqrt{-1}\,
\Mir^{-1}(\phi_i)|_{t=0} \quad \in\quad \mathbb{H}_+(f),\quad
1\leq i\leq N,
\een 
where $\{\phi_i\}_{i=1}^N\subset H$ is a fixed basis with $\phi_1=\mathbf{1}$.  
The higher genus mirror symmetry for $Y$ can be stated as follows.
\begin{theorem}\label{CY-LG:mirror_sym}
The total ancestor potentials of $Y$ and $f$ are related by the
following formula:
\ben
\cA^Y_\tau(\hbar;\mathbf{q}) =
\mathcal{A}_{f,\omega}^{\omega_1,\dots,\omega_N}(\hbar;\mathbf{q}),
\een
where
\ben
\omega:=\omega_1=\sqrt{-1} \, \frac{dx_1\cdots dx_n}{Qf_0(Q)}.
\een
\end{theorem}

\subsection{A-model opposite subspace} 

Let $f=f(x,0,Q)$, $Q\in \Delta^*$. Recall also the notation
$\mathfrak{h}$ and $\mathfrak{h}^*$ for respectively the middle
cohomology and homology of $f$. The opposite subspace $P\subset
\mathbb{H}_+(f)$ that corresponds to the good basis of GW theory can
be characterized as follows. 
The group $\boldsymbol{\mu}_{d_i}$ of $d_i$th roots of 1 acts naturally on
$\mathbb{C}^n\times \Delta^*$ via
\ben
\eta.((x_1,\dots,x_n),Q):= ((x_1,\dots,\eta x_i,\dots,x_n),\eta Q).
\een
The function $f$ is invariant under this action, so the vanishing
homology and cohomology bundles on $\Delta^*$ become naturally 
$\boldsymbol{\mu}_{d_i}$-equivariant bundles.  Let us a define a linear map
\ben
L_i :\mathfrak{h}\to \mathfrak{h}, \langle L_i(A),\alpha\rangle  =
\langle A, \eta_i^{-1}\cdot \alpha_{\eta_i Q}\rangle ,
\een
where $\eta_i=e^{2\pi\sqrt{-1}/d_i}$, $\eta_i^{-1}\cdot$ is the
$\boldsymbol{\mu}_{d_i}$-equivariant action, and $\alpha_{\eta_iQ}$ is the parallel
trasnport of the cycle $\alpha$ along the arc $Qe^{\sqrt{-1}\theta}$, $0\leq
\theta \leq 2\pi /d_i$. Note that $L_i^{d_i}=M_{\rm mar}^{-1}$, where 
$M_{\rm mar}$ is the monodromy transformation of $\mathfrak{h}$
corresponding to a closed loop around $Q=0$ in counter clockwise
direction.  Recall that 
$M_{\rm mar}=e^{-2\pi\sqrt{-1} N_{\rm mar}}$ (see Lemma \ref{M_mar}),
where $N_{\rm mar}$ is a nilpotent operator. In particular, we can define 
$M_{\rm mar}^{1/d_i}:= e^{-(2\pi\sqrt{-1}/d_i)N_{\rm mar}}$. Note that
the linear operators $L_i$ and $M_{\rm mar}$ pairwise
commute for $1\leq i\leq n$. Therefore, the map $\eta_i\mapsto L_i\circ
M_{\rm mar}^{1/d_i}$ gives a representation of $\boldsymbol{\mu}_{d_i}$ on
$\mathfrak{h}$ for each $i=1,2,\dots,n$. Since the operators defininig
the representations pairwise commute we have a joint spectrum
decomposition 
\beq\label{Galois-decomp}
\mathfrak{h} = \bigoplus_{\be=(\be^{(1)},\dots,\be^{(n)})}  \mathfrak{h}_\be,\quad 
\mathfrak{h}_\be = \{ v\in \mathfrak{h}_\be\ :\ L_iv = \eta_i^{-\be^{(i)}}v, \
1\leq i\leq n\}. 
\eeq
where the direct sum is over $\be=(\be^{(1)},\dots,\be^{(n)})$ such that $0\leq \be^{(i)}\leq
d_i-1$ and $\mathfrak{h}_\be\neq \{0\}$. 

Let us recall the definition of a {\em weight filtration} (see
\cite{Schm}, Lemma 6.4). Given a triple $(V,m,N)$ consisting of a
vector space $V$, a positive integer $m$ (called {\em weight}), and a nilpotent operator
$N$ such that $N^m=0$, there is a unique increasing filtration
$0=W_{-1}\subset W_{0}\subset \cdots \subset W_{2m} =
V$, called a {\em weight filtration}, such that
$N(W_\ell)\subset W_{\ell-2}$ and $N^l :\operatorname{Gr}^W_{m+\ell}\to
\operatorname{Gr}^W_{m-\ell}$ is an isomorphism for all $\ell$. 

Put
\ben
N_\be=N_{\rm mar}|_{\mathfrak{h}_\be},\quad m_\be:= n+|\{i\ :\
\be^{(i)}\neq 0\}|-2\lceil \deg(\be)\rceil,
\een
where $|S|$ denotes the number of elements in a set $S$ and
$\deg(\be):=\sum_{i=1}^n \be^{(i)}/d_i$. Let us define
$W_\bullet^\be$ to be the weight filtration corresponding to the tripple
$(\mathfrak{h}_\be,m_\be,N_\be)$. 
\begin{proposition}
The opposite filtration $\{U_\bullet\}$ of $\mathfrak{h}$, which
corresponds to the GW opposite subspace $P$ via
\eqref{eq:opposite_correspondence} is given by the formula
\ben
U_\ell = \bigoplus_\be W^\be_{2\ell},
\een 
where the direct sum over $\be$ is the same as in \eqref{Galois-decomp}.
\end{proposition} 
\proof
Our argument is based on mirror symmetry. Recall the isomorphism
\eqref{psi-def} 
\ben
\Psi:\mathfrak{h}\to H.
\een 
We will find the images of the opposite filtration and the various
weight filtrations under $\Psi$ and see that the desired relation is
obvious. 
  
By definition the one-to-one correspondence,
\eqref{eq:opposite_correspondence} associates to every 
$A\in F^\ell\mathfrak{h}_s\cap U_\ell\mathfrak{h}_s$ with
$s=e^{2\pi\sqrt{-1}\alpha}$, $0\leq \alpha<1$, a homogeneous form $\omega \in
\mathbb{H}_+(f)\cap Pz$ of degree 
$
\operatorname{deg}(\omega) = n-\ell-\alpha,
$ 
such that
\ben
\widehat{s}(\omega,z) = (-z)^{-\ell-\alpha+\frac{n}{2}} \, A,
\een
where we used that for every fixed $z$ the map $\psi$ in
\eqref{eq:opposite_correspondence} is the inverse to $\widehat{s}(\
,z)$. 
By definition 
\ben
\Psi(\widehat{s}(\omega,z) ) = (-z)^{-\theta} S(\tau,1,-z)^{-1} \, \Mir (\omega),
\een
where $\tau=\tau(0,Q)$ is the image of the mirror map. 
On the other hand, due to homegeneity 
\ben
(-z)^{-\theta} S(\tau,1,-z)^{-1} (-z)^\theta
\een
is independent of $z$ and 
\ben
(-z)^{-\theta} \Mir(\omega) = (-z)^{\operatorname{deg}(\omega)
  -\frac{n}{2}} \Mir(\omega).
\een
We get 
\ben
\Psi(A) = S(\tau,1,1)^{-1}\, \ \Mir(\omega). 
\een
Therefore, $\Psi (F^\ell\mathfrak{h}_s\cap U_\ell\mathfrak{h}_s)$ is the
span of all $S(\tau,1,1)^{-1}\phi\in H$ such that $\phi$ is a homogeneous
class satisfying 
\ben
\lceil \operatorname{deg}_{\rm CR} (\phi)\rceil= n-1-\ell,\quad
\operatorname{deg}_{\rm CR} (\phi)+\alpha\in \mathbb{Z}.
\een
Recall the homogeneity condition for $S(\tau,Q,z)$:
\ben
(z\partial_z + E) S(\tau,Q,z) = [\theta,S(\tau,Q,z)].
\een
It follows that $S(\tau,1,1)=1+\sum_{k=1}^\infty S_k(\tau,1)$, where
each operator $S_k(\tau,1)$ increases the degree by $k$. 
Since 
\ben
U_\ell=\bigoplus_{\ell'\geq \ell} \bigoplus_{s} F^{\ell'}\mathfrak{h}_s\cap
U_{\ell'}\mathfrak{h}_s, 
\een
we get that $\Psi(U_\ell)$ is spanned by homogeneous classes $\phi$ in
$H$ such that
\beq\label{U-deg}
\lceil \operatorname{deg}_{\rm CR} (\phi)\rceil\geq n-1-\ell.
\eeq

Let us determine the images of the weight filtrations. 
We already proved that $\Psi\circ N_{\rm mar} = p\cup \, \Psi$
(see Lemma \ref{M_mar}). Using the identity 
\ben
L_i\widehat{s}(\omega_\be,z) = \eta_i^{-\be^{(i)}}
\widehat{s}(\omega_\be,z)|_{Q\mapsto \eta_iQ},
\een
where $\omega_\be= x^\be dx/Q$ and the definition of $\Psi$ (see
\eqref{psi-def}), it is easy to check that 
$\Psi\circ (L_iM_{\rm  mar}^{1/d_i}) = J_i \circ \Psi$, 
where $J_i:H\to H$ is the linear operator defined by  
\ben
J_i (p^k\mathbf{1}_c) = e^{-2\pi\sqrt{-1} c_i} \, p^k\mathbf{1}_c.
\een
Therefore, we have $\Psi(\mathfrak{h}_\be)=H(Y_\be;\mathbb{C})$, where
$Y_\be$ is the twisted sector corresponding to $c=(c_1,\dots,c_n)$, with
$c_i=\be^{(i)}/d_i$. The weight filtration of the triple $(H(Y_\be;\mathbb{C})
,m_\be,p\cup)$ is straightforward to find. We get that $\Psi(W^\be_{2\ell})$
is spanned by homogeneous classes $\phi\in H$, such that
\ben
\operatorname{deg}(\phi)\geq \frac{1}{2}\Big( \operatorname{dim}(Y_\be)+m_\be\Big)-\ell,
\een
where if $\phi$ is a class of (usual) real degree $2i$, then
$\operatorname{deg}(\phi)=i$. Note that 
\beq\label{wt-deg}
\operatorname{deg}_{\rm CR}(\phi) =\operatorname{deg}(\phi)
+\deg(\be).
\eeq
Comparing the inequalities \eqref{U-deg} and \eqref{wt-deg} we see
that they are equivalent when 
$m_\be=n+|\{i\ :\ \be^{(i)}\neq 0\}|-2\lceil\deg(\be)\rceil.$
\qed

\section{Mirror symmetry for Fermat CY singularities}
Now we discuss the mirror symmetry on the LG side.
In \cite{HLSW}, He-Li-Shen-Webb identified the FJRW ancestor potential (LG A-model) of invertible quasi-homogeneous polynomial singularities to the Saito-Givental ancestor potential (LG B-model) of the mirror polynomials, by using Givental-Teleman's \cite{G1, Te} unique higher genus formula for semisimple Frobenius manifolds and matching Frobenius manifolds on both sides via WDVV equations and a perturbative formula in \cite{LLS}.
In this section, inspired from toric geometry, we establish a mirror symmetry statement of $\mathcal{D}$-module structures and opposite subspaces between FJRW theory and Saito's theory for Fermat CY singularities.
For Fermat CY singularities, our result recovers He-Li-Shen-Webb's result. More general cases remain unknown due to the lack of a toric model.

\subsection{FJRW theory of Fermat CY singularities}
As before, we consider the Landau-Ginzburg side of the Fermat polynomial of Calabi-Yau type
$$W=x_1^{d_1}+\cdots+x_n^{d_n}, \quad \sum_{i=1}^{n}\frac{1}{d_i}=1.$$
Let $G_W$ be the group of diagonal symmetries of $W$, so
$$G_W:=\left\{(\lambda_1,\dots,\lambda_n)\in(\mathbb{C}^*)^n\Big\vert\,W(\lambda_1x_1,\dots,\lambda_nx_n)=W(x_1,\dots,x_n)\right\}\cong\prod_{i=1}^{n}\boldsymbol{\mu}_{d_i}.$$
For each $\gamma\in G_W$, there exist unique $\{\Theta_\gamma^{(i)}\in [0, 1)\cap\mathbb{Q}\}$, such that 
$$\gamma=\left(\exp(2\pi\sqrt{-1} \Theta_\gamma^{(1)}),\cdots, \exp(2\pi\sqrt{-1} \Theta_\gamma^{(n)})\right).$$

A mathematical construction of the LG model for a generic pair $(W,G_W)$ is given by Fan, Jarvis, and Ruan \cite{FJR, FJR2}, based on a proposal of Witten \cite{W2}.  More generally, the group $G_W$ can be replaced by any subgroup that contains the exponential grading element 
\begin{equation}\label{exponential-element}
j_W:=\left(\exp(2\pi\sqrt{-1}/d_1), \cdots, \exp(2\pi\sqrt{-1}/d_n)\right).
\end{equation}
In this paper, we only focus on the pair 
$$\left(W=x_1^{d_1}+\cdots+x_n^{d_n}, \quad G_W\right).$$ 
Its {\rm FJRW} theory consists of a graded vector space $H_W$ (called {\rm FJRW} state space), and a Cohomological Field Theory $\{\Lambda^W_{g,k}\}$.
We recall some basics in this section and refer the readers as to \cite{FJR} for more details.

Each $\gamma\in G_W$ acts on $\C^n$ by homothesis and we denote ${\rm Fix}(\gamma)\subset\C^n$ the fix locus of $\gamma$. Let $W_\gamma$ be the restriction of $W$ on ${\rm Fix}(\gamma)$.
Each $\gamma$-twisted sector $H_\gamma$ consists of $G_W$-invariant part of the middle-dimensional relative cohomology for $W_\gamma$.
$$H_{\gamma}:=\left( H^*({\rm Fix}(\gamma), (\Re W_\gamma)^{-1} (-\infty, -M); \C)\right)^{G_W}, \quad M\gg0.$$
Here $H_\gamma$ is called \emph{narrow} if ${\rm Fix}(\gamma)={\bf 0}\in \C^n$ and is called \emph{broad} otherwise. 
Each narrow sector is canonically isomorphic to $\C$,
$$H_\gamma:=H^*(\{\bf 0\}, \emptyset; \C)\cong\C.$$
We denote $\one_\gamma$ the canonical generator in $1\in\C\cong H_\gamma$.

In particular, since $W$ is Fermat CY singularity, the {\rm FJRW state space} is given by
\begin{equation*}
H_{W}=\bigoplus_{\gamma\in\mathscr{N}} H_{\gamma}\cong\bigoplus_{\gamma\in \mathscr{N}}\mathbb{C}\cdot \one_\gamma, 
\end{equation*}
where $\gamma$ belongs to the set of narrow elements
\begin{equation}\label{narrow-element}
\mathscr{N}:=\left\{\gamma\in G_W\Big| 1\leq d_j\Theta_\gamma^{(j)}\leq d_j-1,\forall 1\leq j\leq n\right\}.
\end{equation}
The cardinality of $\mathscr{N}$ is $N:=\prod_{j=1}^{n}(d_j-1)$, hence $H_{W}$ is a vector space of rank $N$.
Moreover,  $(H_{W}, \deg_W)$ is a graded vector space, where
\begin{equation}\label{fjrw-degree}
\deg_{W}\one_\gamma:=\sum_{j=1}^n\left(\Theta_\gamma^{(j)}-\frac{1}{d_j}\right).
\end{equation}
For each $\gamma\in\mathscr{N}$, we define its involution $\gamma'\in\mathscr{N}$ by
$$\Theta_{\gamma'}^{(j)}:=1-\Theta_{\gamma}^{(j)}.$$
Let $\delta_{(-)}^{(-)}$ be the Kronecker symbol. Then $H_W$ has a non-degenerate pairing $\eta_W$, where
\begin{equation}\label{fjrw-pairing}
\eta_W(\one_\alpha,\one_{\beta})=\delta_{\alpha}^{\beta'}, \quad \forall \alpha, \beta\in\mathscr{N}.
\end{equation}

The triple $(H_W, \deg_W, \eta_W)$ has a Cohomological Field Theory $\{\Lambda^W_{g,k}\}$ \cite[Definition 4.2.1]{FJR}, which consists of multilinear maps
$$\Lambda^W_{g,k}: H_W^{\otimes k}\to H^*(\overline{\mathcal{M}}_{g,k}, \C).$$ 
%Here $\overline{\mathcal{M}}_{g,k}$ is the moduli space of stable $k$-pointed curves of genus $g$. 
Here the stability condition is $2g-2+k>0$.
Letting $\gamma_j\in \mathscr{N}$, $\ell_j\geq0$, and $\bar{\psi}_j\in H^*(\overline{\mathcal{M}}_{g,k}, \mathbb{C})$ be the $j$-th $\bar{\psi}$-class,  we have the following \emph{FJRW invariant}
\begin{equation}\label{fjrw-inv}
\lan\one_{\gamma_1}\bar{\psi}_1^{\ell_1},\cdots,\one_{\gamma_k}\bar{\psi}_k^{\ell_k}\ran_{g,k}^{W}
:=\int_{\overline{\mathcal{M}}_{g,k}}\Lambda_{g,k}^{W}(\one_{\gamma_1},\cdots, \one_{\gamma_k}) \prod_{j=1}^k\bar{\psi}_j^{\ell_j}\,.
\end{equation}
Similarly as in GW theory, for any $\tau\in H_W$, we can define a formal function
\begin{equation}\label{formal-fjrw}
\llangle[\Big]\one_{\gamma_1}\bar{\psi}_1^{\ell_1},\cdots,\one_{\gamma_k}\bar{\psi}_k^{\ell_k}\rrangle[\Big]_{g,k}^{W}(\tau):=\sum_{m\geq0}\frac{1}{m!}\lan\one_{\gamma_1}\bar{\psi}_1^{\ell_1},\cdots,\one_{\gamma_k}\bar{\psi}_k^{\ell_k}, \tau, \cdots, \tau\ran_{g,k+m}^{W}.
\end{equation}
We remark that here the stable condition is $2g-2+k+m>0$.

The quantum multiplication $\bullet_\tau$ is given by
$$\eta_W\left(\one_\alpha\bullet_\tau\one_\beta, \one_\gamma\right)=\llangle[\Big]\one_\alpha,\one_\beta,\one_\gamma\rrangle[\Big]_{0,3}^W(\tau).$$
The product $\bullet_\tau$ has an identity $\one:=\one_{j_W}$ with $j_W$ defined in \eqref{exponential-element}.

Again we introduce a set of formal variables $\mathbf{t}=\{t_{k,i}\}$, $1\leq i\leq N$,
$k\geq 0$. We introduce a genus-$g$ generating function 
\ben
\overline{\cF}^{(g)}_{\tau, W}(\mathbf{t}) = 
\sum
\frac{1}{k!}
\llangle[\Big] \mathbf{t}(\bar{\psi}),\dots,\mathbf{t}(\bar{\psi})\rrangle[\Big]_{g,k}^W(\tau)
\een
and the \emph{total ancestor potential} 
\ben
\mathcal{A}_\tau^W(\hbar;\mathbf{t}) := \exp \left(
\sum_{g=0}^\infty  \hbar^{g-1}\,\overline{\cF}_{\tau,W}^{(g)}(\mathbf{t})\right).
\een

\subsubsection{J-function}
Let $\mathcal{H}_W:=H_W(\!(z^{-1})\!)$ be the infinite vector space. Let us consider the Darboux coordinate s
$${\bf p}(z)=\sum_{k\geq0}\sum_{\alpha}p_k^{\alpha}\one_\alpha z^{-k-1}, \quad {\bf q}(z)=\sum_{k\geq0}\sum_{\alpha}q_k^{\alpha}\one_\alpha z^{k}.$$
We may write an element in $\mathcal{H}_W$ as 
$$f(z)=\sum_{k\geq0}\sum_{\alpha} q_{k}^{\alpha}\one_{\alpha} z^k+\sum_{k<0}\sum_{\alpha} p_{k}^{\alpha}\one_{\alpha} z^k.$$
The infinite dimensional vector space $\mathcal{H}_W$ is equipped with a symplectic pairing
$$\Omega_W\Big(f(z), g(z)\Big)={\rm Res}_{z=0}\,\eta_W(f(-z), g(z))dz.$$
We have  $\mathcal{H}_W=\mathcal{H}_W^+\oplus\mathcal{H}_W^-$ where $\mathcal{H}_W^+=H_W[\![z]\!]$ and $\mathcal{H}_W^{-}:=z^{-1}H_W[z^{-1}]$ are Lagrangian subspaces.
Since $\mathcal{H}\cong T^*\mathcal{H}_W^+$, after the \emph{dilation shift}
${\bf q}(z)=-z\one+{\bf t}(z),$
the graph of the genus zero generating function $\overline{\cF}_W^{(0)}$ defines a formal germ of Lagrangian submanifold $\mathcal{L}_W$ in $\mathcal{H}_W$, 
$$\mathcal{L}_W: =\left\{ ({\bf p, q})\in T^*\mathcal{H}_W^+: {\bf p}=d_{\bf q}\overline{\cF}_W^{(0)}\right\}.$$
For each $\tau\in H_W$, we can define an FJRW J-function 
\begin{equation}\label{lg-j-function}
J_{\rm FJRW}(\tau,z):=z+\tau+\sum_{k=0}^\infty\sum_\gamma\llangle[\Big]\one_\gamma\bar\psi^k\rrangle[\Big]_{0,1}^W(\tau)\ \one_{\gamma'}\,z^{-k-1}.
\end{equation}
It is standard to check that
\begin{lemma}\label{j-function-cone}
The J-function $J_{\rm FJRW}(\tau,-z)$ belongs to $\mathcal{L}_W$.
\end{lemma}
We define the calibration operator $S_W(\tau,z)$ in FJRW theory by 
$$\eta_W(S_W(\tau, z)\one_\alpha, \one_\beta)=\eta_W(\one_\alpha,\one_\beta)+\sum_{k=0}^{\infty}\llangle[\Big]\one_\alpha\bar\psi^k,\one_\beta\rrangle[\Big]^{W}_{0,2}(\tau)\,z^{-k-1}.$$
Again, use String Equation, we can also rewrite the J-function as follows:
\begin{equation}\label{j-func-s}
J_{\rm FJRW}(\tau,-z)=-z\,S_W(\tau, z)^{-1}\one.
\end{equation}

\subsection{I-function in FJRW theory}
Now we introduce an FJRW $I$-function $I_{\rm LG}^0(t,z)$ in \eqref{I-function} via toric geometry. This $I$-function lies on the FJRW Langrangian cone, and Birkhoff factorization of the $I$-function will induce a mirror map \eqref{lg-mirror-map}. 
\subsubsection{Toric setup and an {\rm FJRW} I-function.}
Let $\{b_0, b_1,\cdots, b_n\}$ be vectors in $\Z^n$ such that
$$b_0=(1,\cdots,1), \quad b_i^{(j)}=\delta_i^j\, d_i, \quad \forall j=1,\cdots,n.$$
Let $\mathfrak{S}$ is be a set of vectors in $\Z^n$, $$\mathfrak{S}=\mathfrak{R}\coprod \mathfrak{B}, \quad \mathfrak{R}=\{b_1, \cdots, b_n\}$$
and $\mathfrak{B}$ is a set of ghost vectors
\begin{equation}\label{good-basis-lg}
\mathfrak{B}:=\left\{b=(b^{(1)}, \cdots, b^{(n)})\in \Z^n\big| 0\leq b^{(j)}\leq d_j-2, \forall j=1,\cdots, n\right\}.
\end{equation}
In LG side, we have an exact \emph{fan sequence}
\begin{equation}\label{fan-sequence}
0\to \mathbb{L}\rightarrow \mathbb{Z}^\mathfrak{S} \xrightarrow{\varphi} \mathbb{Z}^n.
\end{equation}
%$$G=\left\{b\in\Z^n \Big| \, 0\leq b^{(i)}\leq d_i-2, \forall i=1,\cdots, n.\right\}$$
For each $b\in \mathfrak{S}$, the map $\varphi$ in \eqref{fan-sequence} is defined by
$$\varphi(b)=b\in\Z^n.$$

Let $\sigma$ be the cone generated by the vectors $b_1,\cdots,b_n$, and $\Psi_{\rm LG}: \sigma\cap\mathbb{Z}^n\to\mathbb{Q}^\mathfrak{S}$, where $\Psi_{\rm LG}({\bf e})=\{\Psi_{\rm LG}^b({\bf e})\}$ and the rational coefficients of $b\in \mathfrak{S}$ are given by
$$\Psi_{\rm LG}^b({\bf e})=
\left\{
\begin{array}{ll}
{\bf e}^{(j)}/d_j, & b=b_j\in \mathfrak{R};\\
0, & b\in \mathfrak{B}.
\end{array}
\right.$$ 
Let $\nu({\bf e})=\sum_{b\in \mathfrak{S}}\nu_b({\bf e})\, b\in \mathbb{Q}^\mathfrak{S}$ be defined by
\begin{equation}\label{nu-index}
\nu({\bf e}):=-\Psi_{\rm LG}({\bf e}+b_0)+\sum_{b\in \mathfrak{B}}\nu_b({\bf e})\,\xi_b,
\end{equation}
where for each $b\in \mathfrak{B}$, $\nu_b({\bf e})\in\Z_{\geq0}$ and 
$$\xi_b:=b-\sum_{c\in R}\Psi_{\rm LG}^{c}(b)c\in \mathbb{L}_\mathbb{Q}:=\mathbb{L}\otimes_{\mathbb{Z}}\mathbb{Q}.$$
Thus for each $j=1,\cdots, n$, we have
\begin{equation}\label{nu-part}
\nu_j({\bf e})=-\frac{1}{d_j}\left({\bf e}^{(j)}+1+\sum_{b\in  \mathfrak{B}}\nu_b({\bf e})\, b^{(j)}\right)\in\mathbb{Q}_{<0}.
\end{equation}
For each $\nu\in \mathbb{Q}^\mathfrak{S}$, we can assign an element $\gamma_{\nu}\in G_W$ 
\begin{equation}\label{toric-element}
\gamma_{\nu}=\left(\exp(2\pi\sqrt{-1}\langle-\nu_1\rangle),\cdots,\exp(2\pi\sqrt{-1}\langle-\nu_n\rangle)\right)\in G_W.
\end{equation}
Let $t^{(-)}: \mathbb{L}\to\mathbb{C}$ be a formal function given by
$$t^{\Psi_{\rm LG}({\bf e}+b_0)+\nu}:=\prod_{b\in \mathfrak{B}}t_b^{\nu_b}.$$
Recalling the definition of $\nu$ in \eqref{nu-index}, we define the box element $\Box_{\nu,z}$ to be
\begin{equation}\label{box-element}
\Box_{\nu,z}:=\frac{\prod_{j=1}^{n}\prod_{k=1}^{\lf-\nu_j\rf} \left(\nu_j+k\right)z}{\prod_{b\in \mathfrak{B}}\prod_{k=1}^{\nu_b}(kz)}.
\end{equation}
If there exists $j\in\{1,\cdots,n\}$ such that $-\nu_j\in\mathbb{Z}$, then $\Box_{\nu,z}=0$. Thus for $\Box_{\nu,z}\neq0$, we know $\gamma_\nu\in\mathscr{N}$, and it makes sense to introduce 
\begin{equation}
I_{\rm LG}^{\bf e}(t,z)=
\sum_{\nu\in\mathbb{Q}^S}
\left(\prod_{b\in \mathfrak{B}}t_b^{\nu_b}\right)
\,\Box_{\nu,z}
\,\one_{\gamma_{\nu}}
\in H_{W}[\![t]\!](\!(z)\!).
\end{equation}
Here $H_W[\![t]\!]:=H_W\otimes_{\C}\C[\![t_b;b\in \mathfrak{B}]\!]$. 
Taking ${\bf e}=0$, we get the $I$-function in the LG side:
\begin{equation}\label{I-function}
I_{\rm LG}^0(t,z)=
\sum_{\nu}
\left(\prod_{b\in \mathfrak{B}} t_b^{\nu_b}\right)\,\Box_{\nu,z}\,
\one_{\gamma_{\nu}}\quad\in H_W[\![t]\!](\!(z)\!).
\end{equation}

Now we assign the following degree:
$$\deg_W t_b=1-\sum_{j=1}^{n}\frac{b^{(j)}}{d_j}, \quad \deg_W z=1.$$
When we apply \eqref{nu-part} and \eqref{fjrw-degree}, we see that each term in $I_{\rm LG}^{\bf e}(t,z)$ has degree
\begin{eqnarray*}
\deg_W \left(\prod_{b\in \mathfrak{B}}t_b^{\nu_b}
\,\Box_{\nu,z}
\,\one_{\gamma_{\nu}}\right)
&=&\sum_{b\in\mathfrak{B}}\nu_b\left(1-\sum_{j=1}^{n}\frac{b^{(j)}}{d_j}\right)+
\sum_{j=1}^{n}\lf-\nu_j\rf
-\sum_{b\in\mathfrak{B}}\nu_b
+\sum_{j=1}^{n}\left(\langle-\nu_j\rangle-\frac{1}{d_j}\right)\\
&=&\sum_{j=1}^{n}\frac{{\bf e}^{(j)}}{d_j}.
\end{eqnarray*}
This depends on ${\bf e}$ only, so we know $I_{\rm LG}^0(t,z)$  is homogeneous of degree zero; i.e.,
\begin{equation}\label{i-func-deg}
\deg_W\left(I_{\rm LG}^0(t,z)\right)=0.
\end{equation}

Following \cite[Section 2.3]{CCIT}, we say $F(y)$ is an \textit{$\mathcal{H}_W[\![y]\!]$-valued point} in the Lagrangian cone $\mathcal{L}_W$ if 
\begin{equation}\label{lag-point}
F(y)=-z\one+{\bf t}(z)+\sum_{\gamma\in H_W}\llangle[\Big]\frac{\one_\gamma}{-z-\psi}\rrangle[\Big]^{W}_{0,1}({\bf t})\ \one_{\gamma'}\in \mathcal{H}_W[\![y]\!]
\end{equation}
for some ${\bf t}(z)\in \mathcal{H}_W^+[\![y]\!]$ such that ${\bf t}(z)|_{y=0}=0$.
The following result is known to experts.
\begin{proposition}
\label{lg:mirror}
The formal function $-z\,I_{\rm LG}^0(t,-z)$ is an $\mathcal{H}_W[\![t]\!]$-valued point in the Lagrangian cone $\mathcal{L}_W$. 
\end{proposition}
One way to prove this statement is generalizing the method of twisted FJRW theory in \cite{CHR2}. 
For the readers' convenience, we give a proof in Appendix \ref{appendix} using localization computation, following the method of Ross and Ruan \cite{RR}. 

\subsubsection{Convergence of the I-function in LG side}
Let $e_i\in\Z^n$ be the $i$-th standard vector; i.e. $e_i^{(j)}=\delta_i^j$, $i=1,\cdots, n$.
Let $t_i$ be the parameter of $e_i$. From now on, we let $\sigma$ be the parameter of $b_0=(1,\cdots,1)\in\Z^n$. 
We restrict $I_{\rm LG}^0(t,z)$ to the following subspace of $\C^{\mathfrak{B}}\times\C\ni(\{t_b\}_{b\in\mathfrak{B}},z)$:
$$\C^{n+1}\times\C\ni (t_1,\cdots,t_n,\sigma,z).$$
We denote the restriction by
$I_{\rm LG}^0(t_1,\cdots, t_n, \sigma,z).$
From to \eqref{i-func-deg}, we know 
$$\deg_W\left(I_{\rm LG}^0(t_1,\cdots, t_n, \sigma,z)\right)=0.$$
On the other hand, since
$$\deg_W t_i=1-\frac{1}{d_i}>0, \quad \deg_W (\sigma)=0, \quad \deg_W z=1, \quad 0\leq\deg_W\one_\gamma\leq n-2,$$
we can rewrite the function $I_{\rm LG}^0(t_1,\cdots,t_n,\sigma,z)$ as
$$I_{\rm LG}^0(t_1,\cdots,t_n,\sigma,z)=\sum_{k=0}I^W_{k}(t_1,\cdots,t_n,\sigma)z^{-k+1}\in H_W[\![t_1,\cdots,t_n, \sigma]\!][\![z^{-1}]\!].$$
Hence $I^W_{0}(t_1,\cdots,t_n,\sigma)$ is homogeneous of degree zero. If $D=l.c.m(d_1,\cdots, d_n)$, then
\begin{equation}\label{birkhoff-positive}
I^W_{0}(t_1,\cdots,t_n,\sigma)=f_0^W(\sigma)\,\one:=\left(1+\sum_{m\geq1}\frac{\sigma^{mD}}{(mD)!}\prod_{j=1}^{n}\prod_{k=1}^{mD/d_j}\left(k-\frac{mD+1}{d_j}\right)\right)\one.
\end{equation}
The ratio test shows that $f_0^W(\sigma)$ is analytic in a neighborhood of $\sigma=0$. Further more, we have the following convergence result.
\begin{corollary}\label{convergence-cor}
For each $k\geq0$, $I^W_{k}(t_1,\cdots,t_n,\sigma)\in H_W[t_1,\cdots,t_n]\{\sigma\}$.
\end{corollary}
\begin{proof}
The polynomiality of $t_1, \cdots, t_n$ follows from degree counting and $\deg_W t_i>0$ for each $i=1,\cdots, n$. For any fixed homogeneous element in $H_W[t_1,\cdots,t_n]$, we can use the ratio test to obtain the analyticity of $\sigma$ a neighborhood of $\sigma=0$.
\end{proof}
More generally, recall that $W=x_1^{d_1}+\cdots+x_n^{d_n}$ is an element in $\mathcal{M}$ in \eqref{eq:universal_polynomial}. 
We define
$$\mathfrak{B}_{\rm rel}=\{b\in\mathfrak{B}|\deg_W t_b>0\}, \quad \mathfrak{B}_{\rm mar}=\{b\in\mathfrak{B}|\deg_W t_b=0\}.$$
We may consider a neighborhood of $W\in\mathcal{M}$ consisting of 
$$W+\sum_{b\in \mathfrak{B}_{\rm rel}\cap\mathfrak{B}_{\rm mar}}t_b x^b, \quad |t_b|<\delta\quad \textit{if}\quad b\in \mathfrak{B}_{\rm mar}.$$
We denote this neighborhood by $\mathcal{M}_{{\rm LG}, \delta}$,
$$\mathcal{M}_{{\rm LG}, \delta}\cong\C^{\mathfrak{B}_{\rm rel}}\times\Delta_\delta^{\mathfrak{B}_{\rm mar}}.$$
If $\delta$ is sufficiently small, then a discussion similar to that in Corollary \ref{convergence-cor} shows
\begin{equation}\label{convergence-i}
I_{\rm LG}^0(t,z)|_{\mathcal{M}_{{\rm LG}, \delta}}\in H_W\otimes_{\C}\mathcal{O}_{\mathcal{M}_{{\rm LG}, \delta}}[\![z^{-1}]\!].
\end{equation}

\subsubsection{Mirror map}
The Birkhoff factorization allows us to rewrite  $I_{\rm LG}^0(t,z)$ as
\begin{equation}\label{birkhoff-lg}
I_{\rm LG}^0(t,z)=\mathfrak{L}(t,z)\Upsilon_{\rm LG}(t,z),
\end{equation}
with 
$$\Upsilon_{\rm LG}(t,z)\in\mathcal{H}_W^+[\![t]\!] \quad \textit{and} \quad \mathfrak{L}(t,z):=1+\sum_{k\geq1}\mathfrak{L}_k(t)\,z^{-k}\in {\rm End}(H_W[\![t]\!])[\![z^{-1}]\!].$$
A mirror map $\tau: \C^{\mathfrak{B}}\to H_W[\![t]\!]$ is given by
\begin{equation}\label{lg-mirror-map}
\tau(t):=\mathfrak{L}_1(t)({\bf 1})\in H_W[\![t]\!].
\end{equation}
%Combining Lemma \ref{j-function-cone} and Proposition \ref{lg:mirror}, we have the equality
%$$J_{\rm FJRW}(\tau,-z)=-z\mathfrak{L}(t,-z)\one.$$
%Here we identify the calibration operator $S_W(\tau(t),z)^{-1}$ with the operator $\mathfrak{L}(t,z)$ via the mirror map \eqref{lg-mirror-map}.

By \eqref{convergence-i}, the restriction of $I_{\rm LG}^0(t,z)$ to ${\mathcal{M}_{{\rm LG}, \delta}}$ will imply
$$\Upsilon_{\rm LG}(t,z)|_{\mathcal{M}_{{\rm LG}, \delta}}\in \mathcal{O}_{\mathcal{M}_{{\rm LG}, \delta}} \cdot\one,$$
and the mirror map restricts to an analytic map
\begin{equation}\label{lg-map}
\tau: {\mathcal{M}_{{\rm LG}, \delta}}\longrightarrow H_W.
\end{equation}

\begin{remark}
In particular, if we restrict to the $(t_1,\cdots,t_n,\sigma)$-plane, then 
$$\Upsilon_{\rm LG}(t_1,\cdots,t_n,\sigma,z)=f_0^W(\sigma)\,\one,$$
where $f_0^W(\sigma)$ is given in \eqref{birkhoff-positive}, and the mirror map restricts to
$$\tau(t_1,\cdots,t_n,\sigma)=\frac{I^W_{1}(t_1,\cdots,t_n,\sigma)}{f_0^W(\sigma)}\in H_W^{\leq1}[t_1,\cdots,t_n]\{\sigma\}.$$
Here $H_W^{\leq1}$ are the elements of $H_W$ with $\deg_W\leq1$.
\end{remark}

\subsection{Mirror symmetry to FJRW theory}

\subsubsection{An isomorphism between $\mathcal{D}$-modules}
Let $\mathcal{T}_W$ be the tangent spaces of $\mathcal{L}_W$ and $d$ be the trivial connection. 
We pull back via the mirror map in \eqref{lg-map} to get a $\mathcal{D}$-module over $\mathcal{M}_{{\rm LG}, \delta}$, which we denote again by $(\mathcal{T}_W, d)$.  
%Then $\mathcal{T}_W$ inherits a $\mathcal{D}$-module over $H_W$, denoted by $(\mathcal{T}_W, d)$. 
Here 
$$\mathcal{D}:=\C[z][\![t_b\colon b\in\mathfrak{B}]\!]\lan z\partial_{t_b}\colon b\in\mathfrak{B}\ran.$$
We obtain another $\mathcal{D}$-module over $H_W$, denoted by $(\mathcal{T}_W(-z), d)$. It is obtained by the following transformation
\begin{equation}\label{sign-change}
\mathcal{T}_W\xrightarrow{S_W(\tau,z)} \mathcal{H}_W^+\xrightarrow{z\mapsto -z} \mathcal{H}_W^+\xrightarrow{S_W(\tau,z)^{-1}}\mathcal{T}_W(-z).
\end{equation}
Now we construct a $\mathcal{D}$-module isomorphism between $(\hHH_+(W), \nabla)$ and  $(\mathcal{T}_W(-z), d)$. 
\begin{lemma}\label{gkz-lg}
The set $\{I_{\rm LG}^{\bf e}(t, z)|{\bf e}\in \Z_{\geq0}^n\}$ satisfies the following differential equations:
\begin{eqnarray}
&&z\partial_{t_b} I_{\rm LG}^{\bf e}(t, z)=I_{\rm LG}^{{\bf e}+b}(t, z), \quad \forall b\in \mathfrak{B}; \label{gkz-lg1}
\\
&&z({\bf e}^{(i)}+1)\,I_{\rm LG}^{\bf e}(t, z)+\sum_{b\in \mathfrak{B}}b^{(i)}t_b\,I_{\rm LG}^{{\bf e}+b}(t, z)+\sum_{j=1}^n\,b_j^{(i)}\,I_{\rm LG}^{{\bf e}+b_j}(t, z)=0.\label{gkz-lg2}
\end{eqnarray}
\end{lemma}
\begin{proof}
Recalling \eqref{nu-part}, we will simply denote $\nu=\nu({\bf e})$, $\nu'=\nu({\bf e}+b)$ and $\nu''=\nu({\bf e}+b_j)$.

For the first equation, we shall compare the coefficient of $t_b^{\nu_b-1}\prod_{c\neq b}t_c^{\nu_c}$ on both sides. The corresponding vector $\{\nu_c'\}_{c\in \mathfrak{B}}$ on the right hand side should satisfy 
\begin{equation}\label{gkz-lg1}
\nu_c'=\nu_c-\delta_c^b, \quad c\in \mathfrak{B}.
\end{equation}
Both coefficients are $0$ when $\nu_b=0$, and thus it is enough to match them when $\nu_b\geq1$. 

When $\nu_b\geq1$, on the right hand side, similar to \eqref{nu-part}, we have 
$$\nu_j'=-\frac{1}{d_j}\left({\bf e}^{(j)}+b^{(j)}+1+\sum_{c\in \mathfrak{B}}\nu_c' c^{(j)}\right)\in\mathbb{Q}_{<0}.$$
Equation \eqref{gkz-lg1} implies
$\nu_j'=\nu_j$ and  
$$\Box_{\nu',z}=\frac{\prod_{j=1}^{n}\prod_{k=1}^{\lf-\nu_j'\rf}(\nu_j'z+kz)}{\prod_{c\in G}\prod_{k=1}^{\nu_c'}(kz)}=(\nu_bz)\Box_{\nu,z}.$$
Thus the coefficient of $t_b^{\nu_b-1}\prod_{c\neq b}t_c^{\nu_c}$ on the right hand side is 
$$\Box_{\nu',z}\one_{\nu'}=(\nu_bz)\Box_{\nu, z}\one_{\gamma_{\nu}}.$$

Now let us prove the second identity. There are three terms and we will consider the coefficient of $\prod_{c\in \mathfrak{B}}t_c^{\nu_c}$ for a fixed vector $\nu=\{\nu_c\}_{c\in \mathfrak{B}}$. For each $b\in \mathfrak{B}$, the contribution from $I_{\rm LG}^{{\bf e}+b}(t, z)$ comes from the vector $\{\nu_c'\}_{c\in \mathfrak{B}}$ such that
$$
\nu_c'=\nu_c-\delta_c^b, \quad \forall c\in \mathfrak{S}.
$$
For each $j$, the contribution from $I_{\rm LG}^{{\bf e}+b_j}(t, z)$ comes from the vector $\{\nu_c''\}_{c\in \mathfrak{B}}$ such that
$$
\nu_c''=\nu_c-\delta_c^j, \quad \forall c\in \mathfrak{S}.
$$
\iffalse
$$
\nu_c'=
\left\{
\begin{array}{lll}
\nu_c-\delta_c^b, &\textit{if} & c\in \mathfrak{B};\\
\nu_j, & \textit{if} & 1\leq c=j\leq n.
\end{array}
\right.
$$
$$
\nu_c''=
\left\{
\begin{array}{lll}
\nu_c, &\textit{if} & c\in G;\\
\nu_j-1, & \textit{if} & 1\leq c=j\leq n;\\
\nu_i, & \textit{if} & c=i\neq j, 1\leq i\leq n.
\end{array}
\right.$$
\fi
Thus 
$$\Box_{\nu',z}\one_{\gamma_{\nu'}}=(\nu_bz)\Box_{\nu,z}\one_{\gamma_{\nu}},\quad \Box_{\nu'',z}\one_{\gamma_{\nu''}}=(\nu_jz)\Box_{\nu,z}\one_{\gamma_{\nu}}.$$
Put everything together, the coefficient of $\prod_{c\in \mathfrak{B}}t_c^{\nu_c}$ of the LHS in \eqref{gkz-lg2} is given by
$$\left(z({\bf e}^{(i)}+1)+\sum_{b\in \mathfrak{B}}b^{(i)}\nu_bz+\sum_{j=1}^{n}b_j^{(i)}\nu_jz\right)
\Box_{\nu,z}\one_{\gamma_{\nu}}.$$
Since $b_j^{(i)}=d_j\delta_j^i$, Equation \eqref{nu-part} implies that the formula above vanishes. 

Finally, we check the constant term in equation \eqref{gkz-lg2}. The constant in the second term vanishes by definition. The remaining two terms give
$$\left(z({\bf e}^{(i)}+1)+\sum_{j=1}^{n}b_j^{(i)}\nu_jz\right)
\Box_{\nu, z}\one_{\gamma_{\nu}}=0.$$
This again follows from Equation \eqref{nu-part}, where now $\nu_b=0$ for all $b\in \mathfrak{B}$. 
\end{proof}
%According to this lemma, there exist a $\mathcal{D}$-module on $\mathcal{M}_{{\rm LG}, \delta}$. 
Using \eqref{sign-change}, \eqref{gkz-lg1}, and Proposition \ref{lg:mirror}, we see that $-I_{\rm LG}^{\bf e}(t, z)\in\mathcal{T}_W(-z)$.
%We denote it by $(L_W, \nabla, \eta_W)$. We can identify it to the Saito's structure near the Gepner point $(c_b=0, \forall b\in G)$.
\begin{proposition}\label{lg:d-module}
The following map
\begin{equation}\label{mirror-map-lg}
\operatorname{Mir}_{W}: \hHH_+(W)\to \mathcal{T}_W(-z), \quad \quad [x^{\bf e} dx]\mapsto -I_{\rm LG}^{\bf e}(t,z).
\end{equation}
extends to a $\mathcal{D}$-module isomorphism between $(\hHH_+(W), \nabla)$ and  $(\mathcal{T}_W(-z), d)$.
\end{proposition}
\begin{proof}
The surjectivity is obvious and the injectivity is a consequence of Equation \eqref{gkz-lg2}. The result follows since both $\mathcal{D}$-modules has the same rank. 
\end{proof}

\subsubsection{Matching pairings.}
We extend the pairing $\eta_W$ in \eqref{fjrw-pairing} $\C[\![z]\!]$-linearly, and still denote the extension by $\eta_W: H_W[\![z]\!]\times H_W[\![z]\!]\to\C[\![z]\!].$ We define
$$\widetilde{K}_W: \hHH_+(W)\times  \hHH_+(W)\to \C[\![z]\!],$$
by
$$\widetilde{K}_W\left(\omega_1,\omega_2\right):=\eta_W\left(\operatorname{Mir}_{W}(\omega_1),\operatorname{Mir}_{W}(\omega_2)^*\right).$$

For Fermat singularities $W:=x_1^{d_1}+\cdots+x_n^{d_n}$, we recall the set $\mathfrak{B}$ in \eqref{good-basis-lg}.
The set $\{[x^{b}dx]\in \hHH_+(W)\mid b\in\mathfrak{B}\}$ forms a good basis \cite[Theorem 2.10]{HLSW}, and
$$K_W([x^{b}dx], [x^{c}dx])=\delta_{b}^{c}.$$
We identify the set $\mathfrak{B}$ with $\mathscr{N}$, the set of indicies of narrow elements in FJRW theory, by a {\em shifting map} as follows.
\begin{equation}\label{shift-index}
{\rm Sh}:\mathfrak{B}\to\mathscr{N},
\end{equation}
such that ${\rm Sh}(b)=\gamma\in\mathscr{N},$ where $\Theta_\gamma^{(j)}=\frac{b^{(j)}+1}{d_j}$.

\begin{proposition}\label{lg:pairing}
For any $b, c\in\mathfrak{B}$, we have
$$K_W([x^b dx], [x^{c}dx])=\widetilde{K}_W([x^{b}dx], [x^{c}dx]).$$
\end{proposition}
\begin{proof}
To prove the mirror map \eqref{mirror-map-lg} preserves the pairing, it suffices to verify
\begin{equation}\label{i-function-constant}
\eta_W\left(-I_{\rm LG}^{b}(t,z)\mid_{t=0}, -I_{\rm LG}^{c}(t, -z)\mid_{t=0}\right)=\delta_{b}^{c}.
\end{equation}
Now let us compute $I_{\rm LG}^{\bf e}(t,z)$ when $t=0$. According to Equation \eqref{nu-part},
$$\nu=\sum_{j=1}^n\left(-\frac{{\bf e}^{(j)}+1}{d_j}\right)b_j.$$
Thus when $\mathbf{e}\in\mathfrak{B}$, using \eqref{box-element}, a direct calculation will show:
$$
I_{\rm LG}^{\bf e}(t=0,z)=
\one_{\gamma_{\nu}}.%=\one_{{\bf e}+b_0}.
$$
And Equation \eqref{i-function-constant} follows.
\end{proof}

\subsubsection{Matching opposite subspaces.}
The vector space of the good basis 
$$H_{\mathfrak{B}}:=\{[x^{b}dx]\in \hHH_+(W)\mid b\in\mathfrak{B}\}$$ induces an opposite subspace $P_{\mathfrak{B}}$ in $\hHH(W)$. Recalling \eqref{opposite-good} in Section \ref{opposite-section}, we have
$$P_{\mathfrak{B}}=H_{\mathfrak{B}}[z^{-1}]z^{-1}.$$
On the other hand, $H_W(\!(z)\!)$ has a natural opposite subspace $H_W[z^{-1}]z^{-1}.$ Then the restriction of $\operatorname{Mir}_{W}$ on $H_{\mathfrak{B}}$ is induced by the shifting map \eqref{shift-index}
$$\operatorname{Mir}_{W}: H_{\mathfrak{B}}\to H_W, \quad [x^{b}dx]\mapsto-\one_{{\rm Sh}(b)}.$$
It is easy to see that 
\begin{proposition}\label{lg:opposite}
The map $\operatorname{Mir}_{W}$ matches the opposite subspaces $P_{\mathfrak{B}}$ with $H_W[z^{-1}]z^{-1}$.
\end{proposition}

\subsection{Proof of main theorem}

Recall our main theorem: 
 \begin{theorem}
 \label{thm:main} 
          Suppose that $W$ is a Fermat polynomial %with $d=\sum_i c_i$ 
$$W=x_1^{d_1}+\cdots+x_n^{d_n}, \quad \sum_{i=1}^{n}\frac{1}{d_i}=1.$$
Hence  $X_W$ defines a Calabi-Yau hypersurface. Then, 
          \begin{itemize}
         \item[(1)]  LG/CY correspondence conjecture holds for the pair $(W, G_W)$.
         \item[(2)] The modularity conjecture holds for $[X_W/\tilde{G}_W]$.
         \end{itemize}
          \end{theorem}
\begin{proof}
The proof of modularity conjecture follows directly from the definition of B-model generating function (Definition \ref{def:amf}), which is modular but non-holomorphic; and GW-mirror theorems (Theorem \ref{CY-LG:mirror_sym}), which express B-model generating function as the anti-holomorphic completion of GW-generating function. Although our main interest is GW-theory, a similar statement holds for FJRW-theory as well.

To prove the LG-CY correspondence, we need to consider the analytic continuation of holomorphic generating function of GW/FJRW-theory. This can be 
done as follows. Using two mirror theorems (Theorem \ref{CY-LG:mirror_sym}, Proposition \ref{lg:d-module}, \ref{lg:pairing}, and \ref{lg:opposite}), we identify GW/FJRW-generating function to the local generating functions near large complex structure/Gepner limits
on the B-model moduli space. Now we can use the complex coordinates (not flat coordinates) on the B-model moduli space. The GW/FJRW-generating functions
were induced by the GW/FJRW-opposite subspaces. Now, we use the Gauss-Manin connection to parallel transport the GW-opposite subspace
at the large complex structure to the Gepner limit along a path. Note that Gauss-Manin connection preserves the Givental symplectic
vector space and Givental cone, so the parallel transport of an opposite subspace will remain Lagrangian and opposite. In such a way, we obtain a holomorphic generating function in
a neighborhood of a path connecting large complex limit to Gepner limit. Namely, we construct an analytic continuation of the GW-generating function. But the 
analytic continuation of the GW-generating function to the Gepner limit is not the FJRW-generating function because the parallel transport of GW-opposite subspace is different
from FJRW-opposite subspace. By Lemma 5.6, the two generating functions differ by the quantization of the symplectic transformation mapping one opposite subspace to
other.
\end{proof}

\appendix
\section{A proof of Proposition \ref{lg:mirror}}
\label{appendix}
\subsection{Weighted invariants and concavity} 
\begin{definition}
For any $\epsilon\in\mathbb{Q}_{>0}$, we say $(\mathcal{C}, \mathcal{L}_1,\cdots, \mathcal{L}_n)$ is a $(G_W,\epsilon)$-stable structure if the following conditions are satisfied:
\begin{itemize}
\item $\mathcal{C}$ is a connected proper one-dimensional DM stack of genus $0$ with weight-$1$ marked points $x_1, \cdots, x_m$, and weight $\epsilon$ points $y_1,\cdots,y_\ell$. The total weight at each point $p\in \mathcal{C}$ is bounded by $1$. Stacky point can only occur at marked points and nodal points. 
%$$\deg\left(\sum_{i=1}^{m}x_i+\sum_{i=1}^{n}y_i\right)|_p\leq 1.$$
\item Let  $G_p$ be the local isotropy group at the stacky point $p$. There is a faithful representation $r_p: G_p\to G_W.$
\item Each $\mathcal{L}_j$ is an invertible sheaf over $\mathcal{C}$ and there exists integers $\xi_{i,j}\in[1,d_j)$ such that 
\begin{equation}\label{line-spin}
\mathcal{L}_j^{\otimes d_j}\cong \omega_{\mathcal{C},\log}\left(-\sum_{i=1}^{\ell}\xi_{i,j}[y_i]\right).
%\cong \omega_{C}\left(\sum_{i=1}^{m}[x_i]+\sum_{i=1}^{n}\ell_{i,j}[y_i]\right).
\end{equation}
\end{itemize}
\end{definition}
At each marking $x_i$, the local representation sends the generator $1\in G_{x_i}$ to some
$$\gamma_i:=\left(\exp(2\pi\sqrt{-1}\Theta_{\gamma_i}^{(1)}),\cdots\exp(2\pi\sqrt{-1}\Theta_{\gamma_i}^{(n)})\right)\in G_W, \quad \Theta_{\gamma_i}^{(j)}\in [0,1)\cap\mathbb{Q}.$$
We fix the decorations $\gamma=(\gamma_1, \cdots, \gamma_m)$ and $\xi=(\xi_1, \cdots, \xi_\ell)$ such that $\Theta_{\xi_i}^{(j)}=\xi_{i,j}/d_j$. We denote the moduli of $(G_W,\epsilon)$-stable structures by $\overline{\mathcal{W}}^{\epsilon}_{m|\ell}(\gamma|\xi)$.
$\mathcal{L}_j$ has a desingularization, which is a line bundle $L_j$ on the coarse curve $C$ \cite[Prop. 4.1.2]{CR04}. When $\overline{\mathcal{W}}^{\epsilon}_{m|\ell}(\gamma|\xi)$ is nonempty,
\begin{equation}\label{line-deg}
\deg L_j=\frac{1}{d_j}\left(-2+m+\ell-\sum_{i=1}^{\ell}\xi_{i,j}\right)-\sum_{i=1}^{m}\Theta_{\gamma_i}^{(j)}\in\Z.
\end{equation}
 According to \cite{FJR}, nonempty $\overline{\mathcal{W}}^{\epsilon}_{m|\ell}(\gamma|\xi)$ is a smooth Deligne-Mumford stack properly fibered over the Hasset moduli space of stable weighted rational curves \cite{Has}, which we denoted by $\overline{\mathcal{M}}_{m|\ell}$. Furthermore,
\begin{equation}\label{weighted-dim}
\dim_{\C}\overline{\mathcal{W}}^{\epsilon}_{m|\ell}(\gamma|\xi)=-3+m+\ell.
\end{equation}

In \cite{FJR}, if the vector space consists of narrow sectors and compact sectors only, the authors use cosection technique \cite{CL, CLL} to construct a virtual fundamental cycle on the moduli space. We denote such a cycle by $[\overline{\mathcal{W}}^{\epsilon}_{m|\ell}(\gamma|\xi)]^{\rm vir}$. In this section, we let $d:=\operatorname{lcm}(d_1, \cdots, d_n)$. Let $\one_{\gamma_i}$ be the insertion at the marked point $x_i$ and $\one_{\xi_i}$ be the insertion at the marked point $y_i$, then the following weighted-$\epsilon$ invariant is defined by
\begin{equation}\label{weight-fjrw}
\lan \one_{\gamma_1}\bar\psi_1^{k_1},\cdots,\one_{\gamma_m}\bar\psi_m^{k_m}\Big|\one_{\xi_1},\cdots,\one_{\xi_\ell}\ran_{m|\ell}^{\epsilon}=
d\cdot\int_{\left[\overline{\mathcal{W}}^{\epsilon}_{m|\ell}(\gamma|\xi)\right]^{\rm vir}}\prod_{i=1}^{m}\bar\psi_{i}^{k_i}.
\end{equation}

\subsubsection{Concavity} %From now on we restrict ourselves to the genus zero case and omit $g=0$ in the notation.   
The following Concavity Lemma is very useful.
\begin{lemma}\label{concavity}
Each geometric fiber $(\mathcal{C}, \mathcal{L}_1, \cdots, \mathcal{L}_n)\in \overline{\mathcal{W}}^{\epsilon}_{m|\ell}(\gamma|\xi)$ is concave, i.e., 
\begin{equation}\label{concave}
H^0\left(\mathcal{C}, \mathcal{L}_j\right)=0, \quad \forall \quad 1\leq j\leq n.
\end{equation}
\end{lemma}
\begin{proof}
If $\mathcal{C}$ is smooth, since $d_j\Theta_{\gamma_i}^{(j)}\geq1$ and $\xi_{i,j}\geq1$, \eqref{line-deg} implies $\deg L_j<0$ and \eqref{concave} follows.

If $\mathcal{C}$ is a nodal, then for each $j=1, \cdots, n$, the normalization induces a long exact sequence:
\begin{equation*}
0\to H^0\left(\mathcal{C}, \mathcal{L}_j\right)\to
\bigoplus_{v}H^0\left(\mathcal{C}_v, \mathcal{L}_j\right)\xrightarrow{\varrho}
\bigoplus_{p}H^0\left(p, \mathcal{L}_j\right)
\to H^1\left(\mathcal{C}, \mathcal{L}_j\right)\to
\bigoplus_{v}H^1\left(\mathcal{C}_v, \mathcal{L}_j\right)\to0.
\end{equation*}
Here $\mathcal{C}_v$ runs over all the components after normalization and $p$ runs over all the nodes.
It is enough to prove ${\rm Ker}(\varrho)=0$. If $p$ is a narrow node, then $H^0\left(p, \mathcal{L}_j\right)=0$ and we can split the exact sequence into two different sequences and then discuss individually. Thus we may assume all the nodes are broad. 

We call a broad node \textit{external} if one of the component attached to this node has exactly one node. Otherwise we call a broad node \textit{internal}. We denote the number of external broad nodes by $E$ and the number of internal broad nodes by $I$. Since $C$ is a genus zero nodal curve, there are $E+I+1$ components in the normalization, where $E$ of them contain exactly one node. Also, we must have $E\geq2$. If we denote the number of broad nodes on the component $C_v$ by $B_v$, then 
$$\sum_{v}B_v=E+2I.$$
Moreover, since $d_j \Theta_{\gamma_i}^{(j)}\geq1$, $\xi_{i,j}\geq0$, and $d_j\geq 2$, the formula \eqref{line-deg} implies 
\begin{equation}\label{deg-component}
\deg_{C_v} L_j \leq{B_v\over2}-1.
\end{equation}
By definition of $\varrho$, any nonzero section $\sigma$ such that $\varrho(\sigma)=0$ must vanish on each external node. Thus the total degree of $\sigma$ is at least $E$. However, \eqref{deg-component} implies the number of zeros of $\sigma$ is at most 
$$\sum_{v; B_v\geq2}({B_v\over2}-1)=\sum_{v; B_v\geq2}{B_v\over2}-(I+1)<{E\over2}+I-(I+1)={E\over2}-1<E.$$
This is a contradiction. Thus we must have ${\rm Ker}(\varrho)=0$.
\end{proof}
Let $\mathscr{C}$ be the universal curve, $\pi: \mathscr{C}\to \overline{\mathcal{W}}^{\epsilon}_{m|\ell}(\gamma|\xi)$ be the universal family, and $\mathscr{L}_j$ be the $j$-th universal line bundle. Use Lemma \ref{concavity}, we can define
\begin{equation}\label{weight-concave}
\left[ \overline{\mathcal{W}}^{\epsilon}_{m|\ell}(\gamma|\xi)\right]^{\rm vir}=
c_{\rm top}\left(R^1\pi_*\bigoplus_{j=1}^n\mathscr{L}_j\right)\cap\left[ \overline{\mathcal{W}}^{\epsilon}_{m|\ell}(\gamma|\xi)\right]\in H_*(\overline{\mathcal{W}}^{\epsilon}_{m|\ell}(\gamma|\xi),\mathbb{Q}).
\end{equation}
In particular, if $\epsilon>1$, then there is no weighted point and we get the FJRW virtual cycle up to a sign \cite[Theorem 5.6]{CLL}.

\subsection{Graph spaces and localization}
\begin{definition}
We consider a graph space $(f:\mathcal{C}\to\mathbb{P}^1, \mathcal{L}_1,\cdots, \mathcal{L}_n)$ where 
\begin{itemize}
\item The rational coarse curve $C$ contains a component $C_1\cong\mathbb{P}^1$ with $\deg(f)|_{C_1}=1.$
\item $(\overline{\mathcal{C}/\mathcal{C}_1}, \mathcal{L}_1,\cdots, \mathcal{L}_n)$ is $(G_W,\epsilon)$-stable on each component of $\overline{\mathcal{C}/\mathcal{C}_1}$.
\end{itemize}
\end{definition}
For fixed decorations $(\gamma|\xi)$, we denote the moduli space of graph spaces by $\mathcal{G}^{\epsilon}_{m|\ell}(\gamma|\xi)$.
For any $x_i, y_j\in \mathcal{C}_1$, there are evaluation morphisms 
$${\rm ev}_i, \widetilde{\rm ev}_j: \mathcal{G}^{\epsilon}_{m|\ell}(\gamma|\xi)\to \mathbb{P}^1$$
which send $(f\colon \mathcal{C}\to\mathbb{P}^1, \mathcal{L}_1,\cdots, \mathcal{L}_n)$ to $f(x_i)$ and $f(y_j)$ respectively.
%According to \cite{FJR}, this moduli space $\mathcal{G}^{\epsilon}_{m|\ell}(\gamma|\xi)$ is a DM-stack of dimension $m+\ell$.
The moduli of graph spaces has a GLSM model description \cite[Example 4.2.22]{FJR3}, so the cosection technique \cite{KL, CLL} allows \cite{FJR3} to construct a perfect obstruction theory on $\mathcal{G}^{\epsilon}_{m|\ell}(\gamma|\xi)$. In particular, since the genus $g$ is $0$,  the moduli space is concave. We write the obstruction sheaf as
\begin{equation}\label{graph-obs}
{\rm Ob}_{\mathcal{G}^{\epsilon}_{m|\ell}(\gamma|\xi)}=R^1\pi_*\bigoplus_{j=1}^{n}\mathscr{L}_j.
\end{equation}

Now we consider a $\C^*$-action on $\mathbb{P}^1$:
\begin{equation}\label{c*action}
\lambda\cdot [x_0:x_1]=[\lambda x_0: x_1], \quad \lambda\in\C^*.
\end{equation}
Let us denote $[0]:=c_1^{\C^*}(\mathcal{O}(1))\in H_{\C^*}^2(\mathbb{P}^1)$ if the weight of the $\C^*$-action is $1$ at $0=[0:1]\in\mathbb{P}^1$ and $0$ at $\infty=[1:0]\in\mathbb{P}^1$. We also denote 
$[\infty]:=c_1^{\C^*}(\mathcal{O}(1))\in H_{\C^*}^2(\mathbb{P}^1)$ if the weight of $\C^*$-action is $0$ at $0\in\mathbb{P}^1$ and $-1$ at $\infty\in\mathbb{P}^1$. 

Let $(f\colon \mathcal{C}\to\mathbb{P}^1, \mathcal{L}_1,\cdots, \mathcal{L}_n)$ be a graph space such that $f(x_{m+1})=\infty$, $f(y_{\ell+1})=0$. We fix the type of decorations by
$$\big(\gamma(\alpha)|\xi(\beta)\big):= (\gamma_1, \cdots, \gamma_m, \alpha | \xi_1, \cdots, \xi_\ell, \beta).$$
The $(m+1)$-th marked point $x_{m+1}$ is decorated by $\one_\alpha$ and the $(\ell+1)$-th weighted point $y_{\ell+1}$ is decorated by $\one_\beta$. 
For simplicity, we denote the moduli space by
$$\mathcal{G}^\epsilon:=\mathcal{G}^{\epsilon}_{m+1|\ell+1}\big(\gamma(\alpha)|\xi(\beta)\big).$$

The $\mathbb{C}^*$-action \eqref{c*action} induces a $\mathbb{C}^*$-action on the moduli $\mathcal{G}^\epsilon$ and on its universal bundle. The moduli stack $\mathcal{G}^\epsilon$ has a $\C^*$-equivariant virtual cycle $\left[\mathcal{G}^\epsilon\right]_{\C^*}^{\rm vir}\in A^{\C^*}_*(\mathcal{G}^\epsilon)$.
%\begin{equation}\label{pullback-class}\widetilde{\rm ev}_{\ell+1}^*([0])=z, \quad {\rm ev}_{m+1}^*([\infty])=-z.\end{equation}

We label the fixed locus by decorated dual graph $\Gamma$. 
Let $m_0, n_0$ be the number of marked points and number of weighted points on $f^{-1}(0)=\mathcal{C}_0$. Here neither the node, nor the weighted point with decoration $\one_\beta$ is included. Similarly, we define $m_\infty$ and $n_\infty$ on $f^{-1}(\infty)=\mathcal{C}_\infty$. Thus 
$$m_0+m_\infty=m, \quad n_0+n_\infty=\ell.$$
Let $\Gamma_0$ and $\Gamma_\infty$ be the decorated dual graph of $\mathcal{C}_0$ and $\mathcal{C}_\infty$ respectively. Let $ \overline{\mathcal{W}}^{\epsilon}(\Gamma_0)$ and $\overline{\mathcal{W}}^{\epsilon}(\Gamma_\infty)$ be the corresponding moduli spaces of $(G_W, \epsilon)$-structures. Let $\mathcal{F}_\Gamma$ be the fixed locus labeled by $\Gamma$. Again it has an obstruction bundle ${\rm Ob}_{\mathcal{F}_\Gamma}=\oplus R^1\pi_*\mathcal{L}_j$ and the virtual fundamental cycle
$$[\mathcal{F}_\Gamma]^{\rm vir}=c_{\rm top}\left(\bigoplus_{j=1}^{n}R^1\pi_*\mathcal{L}_j\right)\cap [\mathcal{F}_\Gamma].$$
 We have morphisms
$$\iota_\Gamma\colon  \mathcal{F}_\Gamma=\overline{\mathcal{W}}^{\epsilon}(\Gamma_0)\times\overline{\mathcal{W}}^{\epsilon}(\Gamma_\infty)\longrightarrow\mathcal{G}^\epsilon.$$
Let $\mathcal{N}_\Gamma$ be the normal bundle of $\mathcal{F}_\Gamma$ in $\mathcal{G}^\epsilon$. 
We use Atiyah-Bott localization to obtain 
$$\left[\mathcal{G}^\epsilon\right]_{\C^*}^{\rm vir}=\sum_{\Gamma}(\iota_\Gamma)_*\left(\frac{[\mathcal{F}_\Gamma]^{\rm vir}_{\C^*}}{e_{\C^*}(\mathcal{N}_{\Gamma}^{\rm vir})}\right).$$
%Here $[\mathcal{G}^\epsilon]^{\rm vir}_{\rm loc}$ is a cosection localized $\C^*$ equivariant virtual cycle on $\mathcal{G}^\epsilon$ \cite{KL, CLL, FJR3}.
Then the $\C^*$-integral
\begin{eqnarray}
&&\lan \one_{\gamma_1}\bar\psi_1^{k_1},\cdots,\one_{\gamma_m}\bar\psi_m^{k_m}, \one_{\alpha}\Big|\one_{\xi_1}, \cdots, \one_{\xi_\ell}, \one_\beta\Big|{\rm ev}_{m+1}^*([\infty])\cup\widetilde{\rm ev}_{\ell+1}^*([0])\ran^{\epsilon,\C^*}_{m+1|\ell+1}\label{graph-inv}\\
&&=d\cdot \int_{\left[\mathcal{G}^\epsilon\right]_{\C^*}^{\rm vir}}\left(\prod_{i=1}^m \bar\psi_i^{k_i}\right) {\rm ev}_{m+1}^*([\infty])\cap\widetilde{\rm ev}_{\ell+1}^*([0])
\in\C[\![z]\!]\nonumber
\end{eqnarray}
allows us to define a formal power series with variables $t$ and $y$.
\begin{eqnarray}
&&\llangle[\Big] \one_{\alpha}\Big|\one_\beta\Big|{\rm ev}^*([\infty])\cup\widetilde{\rm ev}^*([0])\rrangle[\Big]^{\epsilon,\C^*}_{1+\bullet|1+\bullet}(t,y)\label{weighted-function}
\\
&&=\sum_{m,\ell}\frac{1}{m!\ell!}\lan t, \cdots, t, \one_{\alpha}\Big| y, \cdots, y, \one_\beta\Big|{\rm ev}_{m+1}^*([\infty])\cup\widetilde{\rm ev}_{\ell+1}^*([0])\ran^{\epsilon,\C^*}_{m+1|\ell+1}.\nonumber
\end{eqnarray}

%We split the fixed locus into two parts: the stable part and the unstable part.
We slightly abuse the language to call fixed locus \emph{stable} if both $p_0:=f^{-1}(0)\cap\mathcal{C}_1$ and $p_\infty:=f^{-1}(\infty)\cap\mathcal{C}_1$ are nodes.
	Otherwise, it is called \emph{unstable}. Then we analyze the contribution of \eqref{weighted-function} when the fixed locus is stable or unstable.

\iffalse
\begin{remark}
Replace $t$ and $y$ by parameters $\tau$ and $t$ respectively for ordinary points and weighted points.
\end{remark}
\fi
 
\subsubsection{Stable contribution.}
Let $(f\colon \mathcal{C}\to\mathbb{P}^1, \mathcal{L}_1,\cdots, \mathcal{L}_n)$ be a geometric point in the stable fixed locus $\mathcal{F}_\Gamma$. 
Then 
$$m_0+1+(n_0+1)\epsilon>2 \quad \textit{and}  \quad m_\infty+2+n_\infty\epsilon>2.$$
%For simplicity, we will denote $$\mathcal{W}_0=\overline{\mathcal{W}}^{\epsilon}_{1+m_{0}|1+n_{0}}(\Gamma_{0}), \quad \mathcal{W}_\infty=\overline{\mathcal{W}}^{\epsilon}_{m_{\infty}+2|n_{\infty}}(\Gamma_{\infty}).$$
The normalization $\widetilde{ \mathcal{C}}=C_0\coprod  \mathcal{C}_1\coprod  \mathcal{C}_{\infty}\to  \mathcal{C}$ induces the following long exact sequence:
\begin{eqnarray}\label{stable-normalization}
0\to\bigoplus_{j=1}^n H^0\left( \mathcal{C}, \mathcal{L}_j\right)\to
\bigoplus_{j=1}^n\bigoplus_{a\in\{0,1,\infty\}}H^0\left( \mathcal{C}_a, \mathcal{L}_j\right)\to
\bigoplus_{j=1}^n\bigoplus_{a\in\{0,\infty\}}H^0\left(p_a, \mathcal{L}_j\right)\to\nonumber\\
\to \bigoplus_{j=1}^n H^1\left( \mathcal{C}, \mathcal{L}_j\right)\to
\bigoplus_{j=1}^n \bigoplus_{a\in\{0,1,\infty\}} H^1\left( \mathcal{C}_a, \mathcal{L}_j\right) \to0.
\end{eqnarray}
%The last term is zero since the points $p_0$ and $p_\infty$ are of dimenension $0$. 
Both $ \mathcal{C}_0$ and $ \mathcal{C}_\infty$ contain at most one broad insertion. Thus we can proceed as in Lemma \ref{concavity} to obtain 
$$H^0( \mathcal{C}_0, \mathcal{L}_j)=H^0( \mathcal{C_\infty},  \mathcal{L}_j)=0.$$
On the other hand, for the component $ \mathcal{C}_1$, the formula \eqref{line-deg} implies
$$\deg_{C_1} L_j=
\left\{\begin{array}{ll}
0, & \textit{two nodes are broad},\\
-1, & \textit{two nodes are narrow}.
\end{array}
\right.$$
In either case, we have
$$H^1( \mathcal{C}_1,  \mathcal{L}_j)=H^1(C_1, L_j)=0.$$

If both nodes are narrow, then the first line in \eqref{stable-normalization}  will vanish.
Thus we get isomorphic $\C^*$-equivariant vector spaces
\begin{equation}\label{narrow-product}
H^1(\mathcal{C}, \mathcal{L}_j)\cong H^1(\mathcal{C}_0, \mathcal{L}_j)\bigoplus H^1(\mathcal{C}_\infty, \mathcal{L}_j).
\end{equation}

%Now let ${\rm Ob}_{\mathcal{G}}$ be the obstruction sheaf over the moduli space $\mathcal{G}$. 
Now if both nodes are broad, then there exists some $j$ such that 
$$H^0\left( \mathcal{C}_1, \mathcal{L}_j\right)\cong\C, \quad H^0\left( p_0, \mathcal{L}_j\right)\bigoplus H^0\left(p_\infty, \mathcal{L}_j\right)\cong\C^2.$$
The first line in \eqref{stable-normalization} contains a summand of $\C\hookrightarrow\C^2$, with $\C^*$ acting trivially on $\C^2$. 
Recall \eqref{graph-obs} and \eqref{narrow-product}, we have 
\begin{equation}
c^{\C^*}_{\rm top}({\rm Ob}_{\mathcal{F}_\Gamma})=
\left\{
\begin{array}{ll}
c^{\C^*}_{\rm top}({\rm Ob}_{\overline{\mathcal{W}}^{\epsilon}(\Gamma_0)}) c^{\C^*}_{\rm top}({\rm Ob}_{\overline{\mathcal{W}}^{\epsilon}(\Gamma_\infty)}), & \textit{if both $0$ and $\infty$ are narrow}.\\
0, & \textit{if both $0$ and $\infty$ are broad}.
\end{array}
\right.
\end{equation}
On the other hand,  we have 
$$c_1^{\C^*}(T_{p_0}\mathcal{C}_{0}\oplus T_{p_0}\mathcal{C}_{1})=\frac{z-\psi_{p_0}}{d}, \quad c_1^{\C^*}(T_{p_\infty}\mathcal{C}_{\infty}\oplus T_{p_\infty}\mathcal{C}_{1})=\frac{-z-\psi_{p_\infty}}{d}.$$
Since the $\C^*$-equivariant Euler class of deformation of the maps $f: \mathcal{C}_1\to \mathbb{P}^1$ equals to ${\rm ev}^*([\infty])\cup\widetilde{\rm ev}^*([0])$.  We obtain that the stable contribution of 
\eqref{weighted-function} is 
\begin{equation}\label{stable-cont}
\llangle[\Big]\one_\alpha\Big|\one_\beta\Big|{\rm ev}^*([0])\cup\widetilde{\rm ev}^*([\infty])\rrangle[\Big]^{\epsilon,\mathbb{C}^*}_{\rm stable}
=\sum_{\gamma}\llangle[\Big]\frac{\one_{\gamma'}}{z-\psi}\Big|\one_\beta\rrangle[\Big]^{\epsilon}_{1|1}\llangle[\Big]\one_\alpha, \frac{\one_{\gamma}}{-z-\psi}
\rrangle[\Big]^{\epsilon}_{2|0}.
\end{equation}
We remark that by definition of \eqref{weight-fjrw}, the RHS only contains stable terms.

\subsubsection{A weighted $I$-function}
Before we start to compute the unstable terms, let us introduce a weighted $I$-function. 
Recall that when ${\bf e}=0$, $\nu$ and $\Box_{\nu,z}$ are given by \eqref{nu-index} and \eqref{nu-part}.
\iffalse
i.e.,
$$\nu:=-\Psi_{\rm LG}(b_0)+\sum_{b\in G}\nu_b\xi_b,$$
and 
$$\nu_j=-\frac{1}{d_j}\left(1+\sum_{b\in G}\nu_b b^{(j)}\right)\in\mathbb{Q}_{<0}.$$
\fi
Recall that $\mathfrak{B}$ is the set of ghost variables defined in \eqref{good-basis-lg}.
For any $\epsilon\in\mathbb{Q}_{>0}$, we consider
$$\mathbb{L}_\epsilon:=\left\{\nu=\{\nu_b\}_{b\in \mathfrak{B}}\Big|\nu_b\geq0, \quad \sum_{b\in \mathfrak{B}}\nu_b\leq{1\over\epsilon}\right\}.$$
In particular, $\lim_{\epsilon\to0}\mathbb{L}_{\epsilon}=\mathbb{L}$ in \eqref{fan-sequence}.
We define an $I^\epsilon$-function
\begin{equation}\label{weighted-i-function}
I_{\rm LG}^{0,\epsilon}(y,z)=\sum_{\nu\in \mathbb{L}_\epsilon}\prod_{b\in \mathfrak{B}}y_b^{\nu_b}\Box_{\nu,z}\one_{\gamma_{\nu}}
\end{equation}
Now we view $\beta$ as its preimage $\big(d_1\Theta_\beta^{(1)}-1, \cdots, d_n\Theta_\beta^{(n)}-1\big)\in\mathfrak{B}$ under the shifting map, and consider the derivative $\partial_{y_{\beta}}$. We calculate similarly as in Lemma \ref{gkz-lg} to obtain
$$z\partial_{y^{\beta}}I^{0,\epsilon}_{\rm LG}(y,z)=\sum_{\nu\in \mathbb{L}_\epsilon}\prod_{b\in \mathfrak{B}}y_b^{\widetilde{\nu}_b}\frac{\prod_{j=1}^{n}\prod_{k=1}^{\lf-\widetilde{\nu}_j\rf} \left(\widetilde{\nu}_j+k\right)z}{\prod_{b\in \mathfrak{B}}\prod_{k=1}^{\widetilde{\nu}_b}(kz)}\one_{\gamma}=\one_{\beta}+\cdots,
$$
where $\widetilde{\nu}_b=\nu_b-\delta_b^\beta\geq0$, $\gamma=\big(\exp(\langle-\nu_1\rangle),\cdots, \exp(\langle-\nu_n\rangle)\big)$,
and
$$\widetilde{\nu}_j=-\frac{1}{d_j}\left(1+(d_j\Theta_\beta^{(j)}-1)+\sum_{b\in\mathfrak{B}}\widetilde{\nu}_b b^{(j)}\right)=\nu_j.$$
\subsubsection{Unstable contribution}
There are three unstable situations:
\begin{enumerate}
\item Both $f^{-1}(0)$ and $f^{-1}(\infty)$ are unstable.
\item $f^{-1}(0)$ is stable but $f^{-1}(\infty)$ is unstable.
\item $f^{-1}(0)$ is unstable but $f^{-1}(\infty)$ is stable.
\end{enumerate}

For the first case of unstable terms, $\mathcal{C}=\mathcal{C}_1$, and 
$$m_0=m_\infty=n_\infty=0, \quad (n_0+1)\epsilon\leq 1.$$
Thus $m=0$ and $\ell=n_0$.
We have an isomorphism $\mathcal{C}_1\cong\mathbb{P}[d,1]$ with $p_\infty$ the only orbifold point. 
All the weighted points $ y_1, \cdots, y_\ell, y_{\ell+1}$ stack at $p_0$. For each $y_i$, we have some $\xi_{i,j}\in\{ 1, \cdots, d_j-1\}$ such that
$$\mathcal{L}_j^{\otimes d_j}\cong \omega_{\mathcal{C}, {\rm log}}\left(-\sum_{i=1}^{\ell}\xi_{i,j}[y_i]-d_j\Theta_{\beta}^{(j)}[y_{\ell+1}]\right).$$
Here $\xi_{\ell+1,j}=d_j\Theta_{\beta}^{(j)}$ since $y_{\ell+1}$ is decorated with the narrow element $\one_\beta$.  
Also we have
$$\mathcal{L}_j\cong\mathcal{O}_{\mathbb{P}[d,1]}\left(\frac{d}{d_j}(-1-\sum_{i=1}^{\ell+1}(\xi_{i,j}-1))[\infty]\right).$$
Since the node $p_\infty$ is decorated by a narrow element $\one_\alpha\in H_W$, use \eqref{line-deg}, we must have
\begin{equation}\label{narrow-node}
\Theta_{\alpha}^{(j)}=\langle a_j\rangle\neq0, \quad a_j:=-\frac{1}{d_j}\left(1+\sum_{i=1}^{\ell+1}(\xi_{i,j}-1)\right). %\lan -\frac{1}{d_j}-\frac{1}{d_j}\sum_{i=1}^{\ell+1}(\xi_{i,j}-1)\ran \in (0,1)\cap\mathbb{Q}.
%\langle \nu_j\rangle, \quad \nu_j=\frac{1}{d_j}\left(-1-\sum_{i=1}^{\ell+1}(\xi_{i,j}-1)\right).
\end{equation}
%Since $\mathbb{P}[d_j,1]$ only has one nontrivial orbifold at $\infty$ with local group $\mu_{d_j}$, 

We extend the $\C^*$ action $\lambda [x_0; x_1]=[\lambda x_0; x_1]$ on $\mathcal{L}_j$ and obtain the following weights
$$c_1^{\C^*}(T_{p_0}\mathcal{C})=z, \quad c_1^{\C^*}(T_{p_\infty}\mathcal{C})=-\frac{z}{d}, \quad c_1^{\C^*}(\mathcal{L}_j|{p_\infty})=0,$$
According  to \cite[Example 98]{Liu}, we have
\begin{equation*}
e_{\C^*}\left(\bigoplus_{j=1}^{n}R^1\pi_*\mathscr{L}_j\right)
%=\prod_{j=1}^{n}\prod_{m=1}^{\lfloor -\nu_j\rfloor}\big(w_3+(m-\langle \nu_j\rangle)w_1\big)
%&=&\prod_{j=1}^{n}\prod_{m=1}^{\lfloor -v_j\rfloor}(m-\langle v_j\rangle)w_1\\
=\prod_{j=1}^{n}\prod_{k=1}^{\lfloor -a_j\rfloor}(a_j+k)z, 
\end{equation*}
\iffalse
We denote the local coordinate near $p_0$ and $p_\infty$ by $u$ and $v$ respectively,  
$$u=\frac{x_0}{x_1}, \quad v^{d_j}=\frac{x_1}{x_0}.$$
Locally near $\infty\in C$, that the finite group $\mu_{d_j}$ acts on $L_j$ by 
$$\lambda\cdot (v, s)=(\lambda v, \lambda^{d_j\Theta_{g,j}}s).$$
We notice that $H^1(C, L_j)$ is generated by $v^{(\nu_j+k)d}u^{-k}, \quad k=1, 2, \cdots, \lfloor -v_j\rfloor$.
The $\C^*$-weight of $v^{(\nu_j+k)d}u^{-k}$ is $-(\nu_j+k)$. Thus 
$$
e_{T}\left(\Big(\bigoplus_{j=1}^{n}R^1\pi_*\mathscr{L}_j\Big)^\vee\right)
=\prod_{j=1}^{n}\prod_{k=1}^{\lfloor -v_j\rfloor}(\nu_j+k)z
$$
%The minus sign in the equality is due to the dual of the vector bundle. 

\fi

%Now for any $\ell$ such that $(\ell+1)\epsilon\leq 1$, 
We replace $\one_{\xi_i}$ by some $b\in \mathfrak{B}$ such that $b^{(j)}=\xi_{i,j}-1$. Let $\widetilde{\nu}_b$ be the number of $b\in \mathfrak{B}$ and we parametrize such an element by $y_b$. Then $\sum\widetilde{\nu}_b=\ell$ and 
$$a_j=-\frac{1}{d_j}\left(1+\sum_{i=1}^{\ell+1}(\xi_{i,j}-1)\right)=-\frac{1}{d_j}\left(1+(d_j\Theta_\beta^{(j)}-1)+\sum_{b\in\mathfrak{B}}\widetilde{\nu}_b b^{(j)}\right)=\nu_j.$$
Since the $\C^*$-equivariant Euler class of deformation of the maps $f: \mathcal{C}\to \mathbb{P}^1$ is $z^{\ell+1}(-z/d)$ and here ${\rm ev}^*([\infty])\cup\widetilde{\rm ev}^*([0])=-z^2/d$, we can collect all possible $\ell$ to obtain the first unstable part of \eqref{weighted-function} 
\begin{equation}\label{unstable-cont}
%\lan\one_{\alpha}\Big| y, \cdots, y, \one_\beta\Big|{\rm ev}_{1}^*([\infty])\cup\widetilde{\rm ev}_{\ell+1}^*([0])\ran^{\epsilon,\C^*}_{1|\ell+1}\nonumber\\
%&&=
\sum_{\ell+1\leq\frac{1}{\epsilon}}\sum_{\sum\widetilde{\nu}_b=\ell}{-z^2/d\over z^{\ell+1}(-z/d)}\prod_{b\in \mathfrak{B}}{y_b^{\widetilde{\nu}_b}\over \widetilde{\nu}_b !}\prod_{j=1}^{n}\prod_{k=1}^{\lfloor -a_j\rfloor}(a_j+k)z
%=\frac{1}{z}\sum_{\widetilde\nu\in \mathbb{L}^{*}_\epsilon}\prod_{b\in G}y_b^{\widetilde\nu_b}\Box_{\nu,z}.
%\llangle[\Big]\phi_s, \frac{\phi_{g'}}{-z-\psi}\Big| \phi_r\rrangle[\Big]^{\epsilon}_{m_\infty+2|n_\infty}
%=\eta_W\left(z\partial_{y^{\beta}}I^{0,\epsilon}_{\rm LG}(y,z), \one_\alpha\right).\label{unstable-cont}
=\eta_W\left(z\partial_{y_{\beta}}I^{0,\epsilon}_{\rm LG}(y,z), \one_\alpha\right).
\end{equation}

%By definition of \eqref{nu-part} and \eqref{toric-element}, we get an element $\one_{\gamma_{\nu}}$ such that $\one_{\gamma_{\nu}}=\one_{\gamma'}$. Here $\gamma$  satisfies \eqref{narrow-node} and $\gamma'$ is the involution of $\gamma$.

%Here $z^{n_0+1}$ a factor in the $\C^*$-equivariant Euler class by smoothing $n_0+1$ weighted points; the denomenator and numberator $-z^2$ are explained as before.
\iffalse
\subsection{Final answer}
Finally, we get
\begin{eqnarray}
&&\frac{1}{m!n!}\llangle[\Big]\phi_s\Big|\phi_r\Big|{\rm ev}_{m+1}^*([0])\cap\widetilde{\rm ev}_{n+1}^*([\infty])\rrangle[\Big]^{\epsilon,\mathbb{C}^*}_{m|n}\nonumber\\
&&=\sum_{m_0}\frac{1}{m_0!n_0!m_\infty!n_\infty!}\llangle[\Big]\phi_s\Big|\phi_r,\frac{\phi_g}{z-\psi}\rrangle[\Big]^{\epsilon}_{m_0+1|n_0+1}\llangle[\Big]\phi_s, \frac{\phi_{g'}}{-z-\psi}\Big| \phi_r\rrangle[\Big]^{\epsilon}_{m_\infty+2|n_\infty}\label{localization}\\
&&+%\sum_{}\prod_{b\in G}y_b^{\nu_b}\left(\frac{\prod_{j=1}^{N}\prod\limits_{\substack {\nu_j<a\leq 0 \\ \langlea\rangle =\langle\nu_j\rangle }}(az)}{z^{n+1}\prod\limits_{b\in G}\nu_b!}\right)
\llangle[\Big]\phi_s, \frac{\phi_{g'}}{-z-\psi}\Big| \phi_r\rrangle[\Big]^{\epsilon}_{m_\infty+2|n_\infty}
\nonumber
\end{eqnarray}
\fi
Similarly, we can calcuate the contribution from the other two situations of the unstable part to get
\begin{equation}\label{unstable-rest}
\llangle[\Big]{\one_\alpha\over z-\psi} \Big|\one_{\beta}\rrangle[\Big]_{1|1}^{\epsilon}+\eta_W\left(z\partial_{y^{\beta}}I^{0,\epsilon}_{\rm LG}(y,z),  \sum_{\gamma}\llangle[\Big]{\one_\gamma\over -z-\psi}, \one_\alpha\rrangle[\Big]_{2|0}^{\epsilon}\one_{\gamma'}\right).
\end{equation}

\subsection{A proof of Proposition \ref{lg:mirror}}
\subsubsection{Regularity}
We define a $J^{\epsilon}$-function:
\begin{equation}\label{j-epsilon}
J^{\epsilon}({\bf t},y,z):=z\,I^{0,\epsilon}_{\rm LG}(y,z)+{\bf t}(-z)+
\sum_{\gamma}\llangle[\Big]{\one_\gamma\over z-\psi}\rrangle[\Big]_{1|0}^{\epsilon}\one_{\gamma'}.
\end{equation}
\iffalse
Here
\begin{eqnarray*}
\llangle[\Big]{\one_\gamma\over z-\psi}\rrangle[\Big]_{1|0}^{\epsilon}
=\sum_{m,n}\frac{1}{m!}\frac{1}{n!}\lan  t, \cdots, t, {\one_\gamma\over z-\psi} \Big| y, \cdots, y\ran_{m+1| n}^{\epsilon}.
\end{eqnarray*}
We remark that the formula above counts the stable invariants, i.e., $m+n\epsilon>1$.  

%In genus 0, stablity requires $1+m+n\epsilon>2$. 
So for the unstable part, we have 
$$(m,n)=(1,0), \quad \textit{or} \quad (m=0, n\epsilon\leq1).$$
In FJRW theory, i.e., $\epsilon>1$, for the unstable terms, we have $(m,n)=(0,0), (1,0)$.
\fi

%Thus $J^{\epsilon}(t,0,z)$ lies on the Lagrangian cone $\mathcal{L}_W$.

\begin{proposition}\label{regularity-thm}
We have 
\begin{equation}\label{regularity}
\eta_W\left(\partial_{y_\beta}J^{\epsilon}({\bf t},y,z), \partial_{t_0^\alpha}J^{\epsilon}({\bf t},y,-z)\right)
%=\llangle[\Big] \one_{\alpha}\Big|\one_\beta\Big|{\rm ev}^*([\infty])\cup\widetilde{\rm ev}^*([0])\rrangle[\Big]^{\epsilon,\C^*}_{1|1}
\in\C[\![z]\!].
%\\&&=\lan t, \cdots, t, \one_\alpha\Big|y, \cdots, y, \one_\beta\Big|{\rm ev}_{m+1}^*([\infty])\cup\widetilde{\rm ev}_{n+1}^*([0])\ran^{\epsilon, T}_{m+1|n+1}\in\C[\![z]\!].
\end{equation}
\end{proposition}
\begin{proof}
Use \eqref{stable-cont}, \eqref{unstable-cont}, and \eqref{unstable-rest}, we obtain
$$\eta_W\left(\partial_{y_\beta}J^{\epsilon}({\bf t},y,z), \partial_{t_0^\alpha}J^{\epsilon}({\bf t},y,-z)\right)
=\llangle[\Big] \one_{\alpha}\Big|\one_\beta\Big|{\rm ev}^*([\infty])\cup\widetilde{\rm ev}^*([0])\rrangle[\Big]^{\epsilon,\C^*}_{1|1}.$$

%Let $\widetilde\nu_b:=\nu_b-\delta_b^\beta$. 
%We define $$\mathbb{L}^{*}_{\epsilon}:=\left\{\widetilde\nu=\{\widetilde\nu_b\}_{b\in G}\Big| \{\widetilde\nu_b\}_{b\in G}\in\mathbb{L}_{\epsilon},\quad \sum_{b\in G}\widetilde\nu_b=\sum_{b\in G}\nu_b-1\leq{1\over\epsilon}-1\right\}$$
\iffalse
Thus
$$
\partial_{y_\beta} J^{\epsilon}({\bf t},y,z)=z\partial_{y^{\beta}}I^{0,\epsilon}_{\rm LG}(y,z)+
\sum_{\gamma}\llangle[\Big]{\one_\gamma\over z-\psi} \Big|\one_{\beta}\rrangle[\Big]_{1|1}^{\epsilon}\one_{\gamma'}
$$
and each $\nu_j$ satisfies
\begin{equation}\label{index-j-mod}
\nu_j=-q_j-\nu_\beta \Theta_\beta^{(j)}-\sum_{b\neq \beta}\nu_b b^{(j)}
=-q_j-\Theta_\beta^{(j)}-\sum_{b}\widetilde\nu_b b^{(j)}.
\end{equation}
For the second term, we get 
$$\partial_{t_0^\alpha} J^{\epsilon}(t,y,-z)
=\one_\alpha+\sum_{\gamma}\llangle[\Big]{\one_\gamma\over -z-\psi}, \one_\alpha\rrangle[\Big]_{2|0}^{\epsilon}\one_{\gamma'}.$$
The identity follows from matching the formula above with \eqref{stable-cont}, \eqref{unstable-cont}, and \eqref{unstable-rest}.
\fi
As a consequence of \eqref{graph-inv}, we know the LHS is regular at $z=0$.
\end{proof}

\subsubsection{Reconstruction}
Let 
$$zI^{0,\epsilon}_{\rm LG}(y,z)=z\one+G(y,z)$$ and denote the projection of $G(y,z)$ to $\mathcal{H}_W^+$ and $\mathcal{H}_W^-$ by $G_+(y,z)$ and $G_-(y,z)$. Then 
$G_+(y,z)\in H_W[\![y,z]\!], G_+(0,z)=0$, and 
\begin{equation}\label{j-func-split}
J^{\epsilon}({\bf t},y,z)=z\one+{\bf t}(-z)+G_+(y,z)+
G_{-}(y,z)+
\sum_{\gamma}\llangle[\Big]{\one_\gamma\over z-\psi}\rrangle[\Big]_{0,1}^{\epsilon}\one_{\gamma'}.
\end{equation}
We introduce multi-indices ${\bf m}$ and ${\bf n}$: 
$${\bf m}=\left(\cdots, m_i^{\gamma}, \cdots\right)\in (\Z_{\geq0})^{\infty},
\quad 
{\bf n}=\left(\cdots, n_0^{\gamma}, \cdots\right)\in(\Z_{\geq0})^{N},
\quad \forall  i\geq0, \gamma\in\mathscr{N}.
$$
Here all but finitely many $m_i^{\gamma}$ are nonzero. We adopt the notation 
\begin{equation*}
|{\bf m}|=\sum_{i\geq0}\sum_{\gamma} m_{i}^{\gamma}, \quad |{\bf n}|=\sum_{\gamma} n_{0}^{\gamma},  
\end{equation*}
We define two vectors ${\bf m}(\alpha)$ and ${\bf n}(\beta)$ such that 
$${\bf m}(\alpha)_i^{\gamma}=m_{i}^{\gamma}+\delta_{i}^{0}\delta_{\alpha}^{\gamma}, \quad {\bf n}(\beta)_{0}^{\gamma}=n_{0}^{\gamma}+\delta_{\beta}^{\gamma}.$$
\iffalse
\begin{equation*}
 \mathbf{m}+(1)_0^s:=\left(m_0^{\phi_1}, \cdots, m_0^{\phi_{s}}+1, \cdots, m_0^{\phi_\mu}, m_1^{\phi_1}, \cdots\right), \quad \mathbf{n}+(1)_r:=\left(n_0^{\phi_1}, \cdots, n_0^{\phi_r}+1, \cdots, n_0^{\phi_\mu}\right)
\end{equation*}
\fi
We define coefficients $A_{{\bf n},j\geq0,\gamma}$ and $B_{\mathbf{m},j<0,\gamma} $ by expanding 
\begin{eqnarray}
&&G_+(y,z):=\sum_{\mathbf{n}}\sum_{j\geq0}\sum_{\gamma}\, A^{\epsilon}_{\mathbf{n},j,\gamma}\, y^{\mathbf{n}}\, z^j\, \one_\gamma, \label{i-func-plus}\\
&&\sum_{\gamma}\llangle[\Big]\frac{\one_\gamma}{z-\psi}\rrangle[\Big]^{\epsilon}_{0,1}\Big\vert_{y=0}\one_{\gamma'}:=\sum_{\mathbf{m}}\sum_{j\leq-1}\sum_{\gamma}\, B^{\epsilon}_{\mathbf{m}, j, \gamma}\, {\bf t}^{\mathbf{m}} \,(-z)^j\,\one_\gamma.\label{j-func-minus}
\end{eqnarray}
Let us write 
\begin{equation}\label{j-coefficient}
J^{\epsilon}({\bf t},y,z):=\sum_{\mathbf{n}}\sum_{\mathbf{m}}\sum_{j\in\Z}\sum_{\gamma}\,C^{\epsilon}_{\mathbf{m},\mathbf{n},j,\gamma}\, {\bf t}^{\mathbf{m}}\, y^{\mathbf{n}}\, z^j\,\one_\gamma.
\end{equation}
%\begin{definition}
If $C^{\epsilon}_{\mathbf{m}, \mathbf{n}, j, \gamma}$ is determined by the coefficients in \eqref{i-func-plus} and \eqref{j-func-minus}, then 
we say $C^{\epsilon}_{\mathbf{m}, \mathbf{n}, j, \gamma}\in \mathfrak{Y}.$
%\end{definition}
\begin{proposition}\label{determine-prop}
The function $J^{\epsilon}({\bf t},y,z)$ is determined by the coefficients $\{A^{\epsilon}_{\mathbf{n},j\geq0,\gamma}\}$ in \eqref{i-func-plus} and $\{B^{\epsilon}_{\mathbf{m}, j<0, \gamma}\}$ in \eqref{j-func-minus}.
\end{proposition}
\begin{proof}
Let $\widetilde{C}^{\epsilon}_{\mathbf{m},\mathbf{n},j,\gamma}=(-1)^j\,C^{\epsilon}_{\mathbf{m},\mathbf{n},j,\gamma}.$
Direct calculation shows
\iffalse
\begin{eqnarray*}
\frac{\partial}{\partial y_r}J^{\epsilon}(t,y,z)
=\sum_{\mathbf{n}}\sum_{\mathbf{m}}\sum_{j\in\Z}\sum_{k=1}^{\mu} (n_0^{\beta}+1) C^{\epsilon}_{\mathbf{m},\mathbf{n}+(1)_\beta,j,k}\ t^{\mathbf{m}}y^{\mathbf{n}}\phi_k z^j;\\
\frac{\partial}{\partial t_0^s}J^{\epsilon}(t,y,z)
=\sum_{\mathbf{n}}\sum_{\mathbf{m}}\sum_{j\in\Z}\sum_{k=1}^{\mu} (m_0^{\alpha}+1) C^{\epsilon}_{\mathbf{m}+(1)_0^\alpha,\mathbf{n},j,k}\ t^{\mathbf{m}}y^{\mathbf{n}}\phi_k z^j.
\end{eqnarray*}
Thus we have
\fi
\begin{eqnarray*}
&&\eta_W\left(\partial_{y_\alpha}J^{\epsilon}({\bf t},y,z), \partial_{t_0^\beta}J^{\epsilon}({\bf t},y,-z)\right)\\
&=&\sum_{\gamma}\sum_{\mathbf{n},\mathbf{n'}}\sum_{\mathbf{m},\mathbf{m'}}\sum_{j,j'}(n_0^{\beta}+1) ({m'}_0^{\alpha}+1) C^{\epsilon}_{\mathbf{m},\mathbf{n(\beta)},j,\gamma}  \widetilde{C}^{\epsilon}_{\mathbf{m'(\alpha)},\mathbf{n'},j',\gamma'}\ {\bf t}^{\mathbf{m}+\mathbf{m'}}y^{\mathbf{n}+\mathbf{n'}}z^{j+j'}.
\end{eqnarray*}
For any fixed $\mathbf{M},\mathbf{N}$, and a positive integer $K$, the regularity formula \eqref{regularity} implies
\begin{equation}\label{regularity-eq}
\sum_{\gamma}\sum_{\mathbf{n}+\mathbf{n'}=\mathbf{N}}\sum_{\mathbf{m}+\mathbf{m'}=\mathbf{M}}\sum_{j}(n_0^{\beta}+1) ({m'}_0^{\alpha}+1) C^{\epsilon}_{\mathbf{m},\mathbf{n(\beta)},j,\gamma}  \widetilde{C}^{\epsilon}_{\mathbf{m'(\alpha)},\mathbf{n'},-K-j, \gamma'}=0.
\end{equation}

Now we prove $C^{\epsilon}_{\mathbf{m}, \mathbf{n}, j, \gamma}\in \mathfrak{Y}$ by three steps. 

\medskip
\noindent
{\bf Step 1.} 
Let $\mathbf{n}=\mathbf{0}:=(0,\cdots,0)$, we know 
$C^{\epsilon}_{\mathbf{m},\mathbf{0},j,\gamma}\in\mathfrak{Y}$ since for any $\mathbf{m}$ and $\gamma$, we have
\begin{equation}\label{j-minus-cond}
C^{\epsilon}_{\mathbf{m},\mathbf{0},j,\gamma} 
=
\left\{
\begin{array}{ll}
B^{\epsilon}_{\mathbf{m},j,\gamma}, & \textit{if} \quad j<0;\\ 
1, & \textit{if} \quad j=1, \mathbf{m}=\mathbf{0}, \one_\gamma=\one;\\
1, & \textit{if} \quad j\geq0, m_i^{\xi}=\delta_i^j\delta_{\gamma}^{\xi};\\
0, & \textit{otherwise}.
\end{array}
\right.
\end{equation}

\medskip
\noindent
{\bf Step 2.} 
For any integer $n_0\in\Z_{\geq0}$, assume that if $|\mathbf{n}|\leq n_0$, we have
\begin{equation}\label{j-minus-assumotion}
C^{\epsilon}_{\mathbf{m},\mathbf{n},j,\gamma}\in\mathfrak{Y}, \quad \forall j\in\Z, \gamma\in\mathscr{N}.
\end{equation}
Now we fix $|\mathbf{N}|=n_0$. Then $\forall \beta\in\mathscr{N}$, $|\mathbf{N(\beta)}|=n_0+1.$ We consider the coefficient
$
C^{\epsilon}_{\mathbf{M},\mathbf{N(\beta)}, j, \gamma}.
$
By the definitions in \eqref{i-func-plus} and \eqref{j-coefficient},we know that 
\begin{equation}\label{j-plus-cond}
C^{\epsilon}_{\mathbf{m},\mathbf{N(\beta)}, j, \gamma}=
\left\{
\begin{array}{ll}
0, & \textit{if}\quad j\geq0, \mathbf{m}\neq0.\\
A_{\mathbf{N(\beta)}, j, \gamma}, &  \textit{if}\quad j\geq0, \mathbf{m}=0.
\end{array}
\right.
\end{equation}
This shows 
$C^{\epsilon}_{\mathbf{M},\mathbf{N(\beta)}, j, \gamma}\in\mathfrak{Y}$ for all $j\geq0$.

\medskip
\noindent
{\bf Step 3.} 
We do induction on the positive integer $K$ by assuming 
\begin{equation}\label{m-assumption}
C^{\epsilon}_{\mathbf{m},\mathbf{N(\beta)}, j, \gamma}\in\mathfrak{Y}, \quad \forall -K-j<0.
\end{equation} 
This is true when $K=1$ by \eqref{j-plus-cond}.
Thus it is enough to prove  for all  $\mathbf{M}$ and $\alpha'\in\mathscr{N}$,
$$C^{\epsilon}_{\mathbf{M},\mathbf{N(\beta)}, -K, \alpha'}\in\mathfrak{Y}.$$
In order to prove this, we rewrite equation \eqref{regularity-eq} as
\begin{eqnarray}\label{reconstruction}
0&=&\sum_{\gamma}\sum_{j\in\Z} (N_0^\beta+1) \,C^{\epsilon}_{\mathbf{M},\mathbf{N(\beta)},j,\gamma}  \widetilde{C}^{\epsilon}_{\mathbf{0(\alpha)},\mathbf{0},-K-j, \gamma'}
\nonumber\\
&&+
\sum_{\gamma}\sum_{\mathbf{m}\neq\mathbf{M}}\sum_{j\in\Z}(N_0^\beta+1)({m'}_0^{\alpha}+1) C^{\epsilon}_{\mathbf{m},\mathbf{N(\beta)},j,\gamma}  \widetilde{C}^{\epsilon}_{\mathbf{m'(\alpha)},\mathbf{0},-K-j, \gamma'}\\
&&+
\sum_{\gamma}\sum_{\mathbf{n}\neq\mathbf{N}}\sum_{\mathbf{m}+\mathbf{m'}=\mathbf{M}}\sum_{j\in\Z}(n_0^{\beta}+1) ({m'}_0^{\alpha}+1) C^{\epsilon}_{\mathbf{m},\mathbf{n(\beta)},j,\gamma}  \widetilde{C}^{\epsilon}_{\mathbf{m'(\alpha)},\mathbf{n'},-K-j, \gamma'}.\nonumber
\end{eqnarray}
Let us analyze the RHS of \eqref{reconstruction} line by line. From \eqref{j-minus-cond}, we observe that
$$
C^{\epsilon}_{\mathbf{0(\alpha)},\mathbf{0},-K-j, \gamma'}=
\left\{
\begin{array}{ll}
B^{\epsilon}_{\mathbf{0(\alpha)},-K-j, \gamma'}\in\mathfrak{Y}, & \textit{if} \quad -K-j<0,\\
\delta^{\alpha}_{\gamma'}, & \textit{if} \quad -K-j=0,\\
0, & \textit{if} \quad -K-j>0.
\end{array}
\right.$$
When $-K-j<0$, by induction \eqref{m-assumption}, $C^{\epsilon}_{\mathbf{M},\mathbf{N(\beta)},j,\gamma}\in\mathfrak{Y}$.
The first line of the RHS of \eqref{reconstruction} is a sum the target term $(N_0^{\beta}+1)\,C^{\epsilon}_{\mathbf{M},\mathbf{N(\beta)}, -K, \alpha'}$ and some elements in $\mathfrak{Y}.$
%$$\sum_{\gamma}\sum_{-K-j<0} (N_0^\beta+1) \,C^{\epsilon}_{\mathbf{M},\mathbf{N(\beta)},j,\gamma}  \widetilde{C}^{\epsilon}_{\mathbf{0(\alpha)},\mathbf{0},-K-j, \gamma'}\in\mathfrak{Y}.$$

The second line belongs to $\mathfrak{Y}$, since if $-K-j<0$, 
$C^{\epsilon}_{\mathbf{m},\mathbf{N(\beta)},j,\gamma}\in\mathfrak{Y}$ by \eqref{m-assumption},
and for $\mathbf{m'}\neq\mathbf{0}$,
$$
C^{\epsilon}_{\mathbf{m'(\alpha)},\mathbf{0},-K-j, \gamma'}
=\left\{
\begin{array}{ll}
0, & \textit{if} \quad -K-j\geq0,\\
B^{\epsilon}_{\mathbf{m'(\alpha)},-K-j, \gamma'}\in\mathfrak{Y}, & \textit{if} \quad -K-j<0.
\end{array}
\right.
$$

Since $\mathbf{n}\neq\mathbf{N}$, the formula \eqref{j-minus-assumotion} implies that the third line of \eqref{reconstruction} is in $\mathfrak{Y}$. 

Finally, since $N_0^{\beta}+1\neq0$ and $\alpha\in\mathscr{N}$ is arbitrary, we know $C^{\epsilon}_{\mathbf{M},\mathbf{N(\beta)}, -K, \alpha'}\in\mathfrak{Y}.$ This finishes the induction argument on \eqref{m-assumption}. 
\end{proof}

\subsubsection{A proof of Proposition \ref{lg:mirror}}
\begin{proof}
Let
$${\bf t}_{\epsilon}(z):={\bf t}(z)+G_+(y,-z)$$
and then consider
$$\widetilde{J}({\bf t}_{\epsilon},z)=z\one+{\bf t}_{\epsilon}(-z)+\sum_{\gamma}\llangle[\Big]\frac{\one_\gamma}{z-\psi}\rrangle[\Big]^{\infty}_{0,1}({\bf t}_{\epsilon})\ \one_\gamma.$$
Using the same method in Proposition \ref{regularity-thm}, we can check
$$\eta_W\left(\partial_{y_\beta} \widetilde{J}({\bf t}_{\epsilon}, z), \partial_{t_0^\alpha} \widetilde{J}({\bf t}_{\epsilon}, -z)\right)\in\C[\![z]\!].$$
Thus the function $\widetilde{J}({t}_{\epsilon},z)$ satisfies the same reconstruction procedure as the function $J^\epsilon({\bf t},y,z)$ in Proposition \ref{determine-prop}. 
Moreover, the initial reconstruction data (see \eqref{i-func-plus} and \eqref{j-func-minus}) are identical for both functions. This implies that 
$$J^\epsilon({\bf t},y,z)=\widetilde{J}({\bf t}_{\epsilon},z).$$
On the other hand, by definition, $\widetilde{J}({\bf t}_{\epsilon},-z)$ is an $\mathcal{H}_W[\![y]\!]$-valued point in the Lagrangian cone $\mathcal{L}_W$.
We let ${\bf t}(z)=0$ and $\epsilon\to0$, the last two terms in \eqref{j-epsilon} vanish. 
In particular, the second term vanishes due to the unstability condition
$1+n\epsilon\leq2.$
Thus by taking the limit, we obtain that $-zI^0_{\rm LG}(y,-z)$ is an $\mathcal{H}_W[\![y]\!]$-valued point in the Lagrangian cone $\mathcal{L}_W$.
%The result follows by choosing an appropriate completion.
\end{proof}

\bibliographystyle{amsalpha}

\end{document}